\documentclass[12pt]{amsart}
\usepackage{graphicx} 
\usepackage[utf8]{inputenc}
\usepackage[T1]{fontenc}

\usepackage[margin=1in,marginparwidth=.7in]{geometry}

\usepackage[dvipsnames]{xcolor}
\usepackage{bbm}
\usepackage{tikz}
\usetikzlibrary{arrows}
\usetikzlibrary{shapes.geometric}
\usepackage{tikz-cd}
\usepackage{pgffor}
\usepackage{amssymb,amsmath,amsfonts,mathrsfs,amsthm}
\usepackage{enumitem}
\usepackage{adjustbox}
\usepackage{hyperref}
\usepackage{relsize}
\usepackage{graphicx}
\usepackage[font={small}]{caption}
\usepackage{float}
\usepackage{subcaption}
\usepackage{changepage}
\usepackage{changes}
\usepackage[toc]{appendix}
\usepackage{pgfplots}
\pgfplotsset{compat=1.7}
\usepackage{adjustbox}
\usepackage{varwidth}
\usepackage{chngcntr}
\usepackage{standalone}
\usepackage{tcolorbox}
\usepackage{mathtools}
\usepackage{rotating}
\usepackage{makecell}

\usepackage{wasysym}  % for full / new moon

\usepackage{tikz}
\usetikzlibrary{knots}
\usetikzlibrary{arrows,decorations.markings}

\usepackage{subfiles}
\graphicspath{{images/}}

\definecolor{lightgray}{rgb}{.85,.85,.85}

\newtheorem{theorem}{Theorem}[section]
\newtheorem{lemma}[theorem]{Lemma}
\newtheorem{proposition}[theorem]{Proposition}
\newtheorem{corollary}[theorem]{Corollary}

\newtheorem{convention}[theorem]{Convention}

\theoremstyle{remark}
\newtheorem{example}[theorem]{Example}

\theoremstyle{remark}
\newtheorem{remark}[theorem]{Remark}

\theoremstyle{definition}
\newtheorem{definition}[theorem]{Definition}

\theoremstyle{remark}
\newtheorem{rmk/}{Remark}

\usetikzlibrary{arrows.meta,calc,decorations.markings,math,arrows.meta}
\newcommand{\R}{\mathbb{R}}
\newcommand{\Q}{\mathbb{Q}}
\newcommand{\Z}{\mathbb{Z}}

\newcommand{\End}{\operatorname{End}}

\newcommand{\id}{\operatorname{id}}

\renewcommand{\l}{\ell}

\renewcommand{\sl}{\mathfrak{sl}}

\newcommand{\myrightleftarrows}[1]{\mathrel{\substack{\xrightarrow{#1} \\[-.15ex] \xleftarrow{#1}}}}

\newcommand{\origin}{\mathcal{O}} % origin of \Lambda
\newcommand{\Hull}{\mathrm{Hull}}% convex hull

\newcommand{\A}{\mathbb{A}}
\newcommand{\DD}{\mathbb{D}}
\renewcommand{\SS}{\mathbb{S}}

\newcommand{\F}{\mathcal{F}}
\newcommand{\til}[1]{\widetilde{#1}}
\renewcommand{\b}[1]{\overline{#1}}
\newcommand{\h}[1]{\widehat{#1}}

\newcommand{\ggmod}{{\text -}\operatorname{ggmod}}

\newcommand{\qdeg}{\operatorname{qdeg}}
\newcommand{\adeg}{\operatorname{adeg}}

\newcommand{\rk}{\operatorname{rk}}

\renewcommand{\u}[1]{\underline{#1}}

\renewcommand{\o}{\otimes}

\newcommand{\eps}{\varepsilon}

\newcommand{\brak}[1]{\ensuremath{\left\langle #1\right\rangle}}
\newcommand{\Fr}{\operatorname{Fr}}

\newcommand{\st}{\ | \ }

\newcommand{\resp}{resp.\ }

\newcommand{\wt}{\operatorname{wt}} % MZ: is this weight? RA:yes 

\renewcommand{\P}{\mathcal{P}} % punctured plane 

\newcommand{\gggmod}{{\text -}\operatorname{g}_3\operatorname{mod}}
\newcommand{\AFoam}{\operatorname{\mathbf{AFoam}}}

\newcommand{\ACob}{\operatorname{\mathbf{ACob}}} %category of annular cobordisms
\newcommand{\ATan}{\operatorname{\mathbf{ATan}}} %2-category of annular tangles and cobordisms between them
\newcommand{\AWeb}{\operatorname{\mathbf{AWeb}}} %2-category of annular webs and foams between them

\newcommand{\len}{\operatorname{len}} % string length
\renewcommand{\line}{\mathcal{L}}

\newcommand{\Bmatchings}{B} % (0,n) webs/tangles up to isotopy, no closed components 
\newcommand{\Bnoisotopy}{\widehat{B}} % (n,m) webs/tangles not up to isotopy, perhaps with closed components 
\newcommand{\Buptoisotopy}{\til{B}} %% (n,m) webs/tangles up to isotopy, perhaps with closed components 

\newcommand{\Imove}{\mathrm{I}}

\newcommand{\Hdegreeone}{\widetilde{H}} %degree one subalgebra of H^n_\A 

\newcommand{\gBimod}{\operatorname{\mathbf{gBimod}}}
\newcommand{\gggBimod}{{\mathbf{g}_3\mathbf{Bimod}}}

\newcommand{\Tw}{\operatorname{Tw}} %Dehn twist 
\newcommand{\mcF}{\mathcal{F}}
\newcommand{\lra}{\longrightarrow}

\newcommand{\growalg}{\mathsf{G}}
\newcommand{\mcpalg}{\mathsf{M}}

\newcommand{\growthwebH}[4]{
    \begin{tikzpicture}[scale=.75]
        \draw (0,0) -- (1,0);
        \draw (0,0) -- +(120:1cm) node[label={90:$#1$}] {};
        \draw (1,0) -- +(60:1cm) node[label={90:$#2$}] {};
        \draw (0,0) -- +(240:1cm) node[label={-90:$#3$}] {};
        \draw (1,0) -- +(-60:1cm) node[label={-90:$#4$}] {};
    \end{tikzpicture}
}

\newcommand{\growthwebU}{
    \begin{tikzpicture}[scale=.75]
        \draw (0,0) node[label={90:$1$}] {} 
            arc (180:360:1cm) 
            node[label={90:$-1$}] {};
    \end{tikzpicture}
}

\newcommand{\growthwebY}[3]{
    \begin{tikzpicture}[scale=.75]
        \draw (0,0) -- (150:1cm) node[label={90:$#1$}] {};
        \draw (0,0) -- (30:1cm) node[label={90:$#2$}] {};
        \draw (0,0) -- (270:1cm) node[label={-90:$#3$}] {};
    \end{tikzpicture}
}

\newcommand{\smallgrowthwebH}[4]{
    \begin{tikzpicture}[scale=.5, baseline=0ex]
        \draw (0,0) -- (1,0);
        \draw (0,0) -- +(120:1cm) node[label={120:$#1$}] {};
        \draw (1,0) -- +(60:1cm) node[label={60:$#2$}] {};
        \draw (0,0) -- +(240:1cm) node[label={240:$#3$}] {};
        \draw (1,0) -- +(-60:1cm) node[label={-60:$#4$}] {};
    \end{tikzpicture}
}

\newcommand{\smallgrowthwebU}{
    \begin{tikzpicture}[scale=.5, baseline=-1ex]
        \draw[white] (0,1.5)--(0,-1.5); 
        \draw (0,0) node[label={90:$1$}] {} 
            arc (180:360:1cm) 
            node[label={90:$-1$}] {};
    \end{tikzpicture}
}

\newcommand{\smallgrowthwebY}[3]{
    \begin{tikzpicture}[scale=.5, baseline=-1ex]
        \draw (0,0) -- (150:1cm) node[label={150:$#1$}] {};
        \draw (0,0) -- (30:1cm) node[label={30:$#2$}] {};
        \draw (0,0) -- (270:1cm) node[label={270:$#3$}] {};
    \end{tikzpicture}
}

\newcommand{\sidecolor}{cyan!50}
\newcommand{\boundarycolor}{red}
\newcommand{\cutcolor}{Purple} % min cut color
\newcommand{\flowcolor}{pink}

  \newcommand*\circled[1]{\tikz[baseline=(char.base)]{
            \node[shape=circle,draw,inner sep=1pt] (char) {${#1}$};}} %shifted dots

\usepackage{todonotes}
\title{Annular SL(2) and SL(3) web algebras}
\author{Rostislav Akhmechet, Mikhail Khovanov, and Melissa Zhang}
\date{August 7, 2025}

\begin{document}

\begin{abstract}
    We use annular foam TQFTs introduced by the first two authors to define equivariant $SL(2)$ and $SL(3)$ web algebras in the annulus. To a diagram of a tangle in the thickened annulus we assign a complex of bimodules over these algebras whose chain homotopy type is an invariant of the tangle. Several properties of algebras and bimodules are established. An essential technical part of the paper provides a bijective correspondence between non-elliptic annular $SL(3)$ webs and closed paths in the $SL(3)$ weight lattice. This generalizes an analogous bijection in the planar setting.    
\end{abstract}

\maketitle

\tableofcontents

\section{Introduction}

Extending link homology to tangles and tangle cobordisms \cite{Bar-Natan, KhFunctorValued, Kh_invt_of_tangle_cobordisms, StroppelTL, MazorchukStroppel, Robert-thesis, MackaayYonezawa}   leads to a variety of interesting categories, both abelian and triangulated, that appear all over in representation theory, algebraic and symplectic geometry, and mathematical physics~\cite{Cautis-Kamnitzer,MPTwebalgebra}.  
Tangles acts on these (triangulated) categories by exact functors and tangle cobordisms---by natural transformations between these functors. 
This viewpoint and structure is part of deep connections between several areas of mathematics and mathematical physics. 

In many cases these categories can be described by starting with appropriate algebras built out of TQFTs for surfaces and for foams (certain two-dimensional CW-complexes) embedded in $\R^3$. 
Foam TQFTs deliver a key approach to constructing link homology theories; see~\cite{sl3-link-homology,MSV,RoseWedrichdeformations,QueffelecRosefoams,RobertWagner,ETW} and other papers.  
Their relative versions, for foams with boundary and corners, can be described via these algebras, which we refer to as \emph{arc} or \emph{web algebras}. 
Web algebras in the planar case (to distinguish from our annular setting) were introduced in \cite{KhFunctorValued} and \cite{MPTwebalgebra, Robert-thesis} in the context of $SL(2)$ and $SL(3)$ link homology, respectively. 

We introduce equivariant annular $SL(2)$ and $SL(3)$ web algebras. 
Here \emph{equivariant} refers to the ground ring being the graded polynomial ring in $N$ variables 
\[
R_\alpha = \Z[\alpha_1, \ldots, \alpha_N]
\]
where each variable $\alpha_i$ is in degree $2$, for $N=2,3$. 
This ring is the $U(1)^N$-equivariant cohomology of a point. 
Our construction uses the annular foam TQFTs introduced in \cite{AK}. 
These TQFTs assign to a contractible circle the $U(1)^N$-equivariant cohomology of $\mathbb{CP}^{N-1}$.
Notably, the annular setting requires working over the full polynomial ring rather than its subring of symmetric polynomials, the latter of which corresponds to passing to $U(N)$-equivariant cohomology. Equivariant homology of links in $\R^3$ can be defined using the smaller ring, in contrast to the annular setting. 

For $SL(2)$ link homology, using the $U(1)\times U(1)$-equivariant framework is also crucial to Sano's fix of functoriality  \cite{Sano-functoriality} that avoids the use of foams or other decorations; see also Vogel \cite{Vogel}. Moreover, specializing variables $\alpha_1, \alpha_2$ to $0$ recovers the APS (Asaeda-Przytycki-Sikora) annular TQFT \cite{APS}. 

Annular $SL(2)$ arc algebras were first studied by Anno and Nandakumar~\cite{Anno-Nandakumar}, motivated by consideration of appropriate Ext groups of coherent sheaves on the quiver variety for the $(n,n)$ nilpotent Springer fiber. 
Ehrig and Tubbenhauer~\cite{EhrigTubbenhauer} introduced a different version based on the APS annular TQFT. This latter version is closely related to the annular $SL(2)$ algebras in this paper; see Remark \ref{rmk_Anno_Tubbenhauer}. Both of these previously defined algebras are non-equivariant, in the sense that their ground ring is either $\Z$ or a field, living entirely in degree $0$. As far as we know, our $SL(2)$ and  $SL(3)$  (equivariant) annular web algebras are new. 

The organization of this paper is as follows. In Section \ref{sec:annular arc algebras and bimodules} we define annular $SL(2)$ web\footnote{In the $SL(2)$ setting these may also be called \emph{arc} algebras.} algebras $H^n_\A$ for $n\geq 0$, bimodules associated to flat annular tangles $T\subset \A:= \SS^1 \times [0,1]$, and chain complexes of bimodules associated to diagrams of tangles in the thickened annulus $\A \times  [0,1]$ whose chain homotopy type is an invariant of the ambient isotopy class of the tangle.  

As a summary of the algebras, for a fixed $n$ we consider $n$ disjoint embedded unoriented arcs in the annulus $\A:= \SS^1 \times [0,1]$ whose boundary is $2n$ points in $\SS^1\times \{1\}$. Pairing two such annular webs $a,b$ together by gluing the reflection of $b$ through $\SS^1\times \{1/2\}$ to $a$ along their common $2n$ endpoints results in a collection of disjoint circles $\b{b}a \subset \A$. Applying the TQFT from \cite{AK} to $\b{b}a$ gives a finitely-generated free graded module $\brak{\b{b}a}$ over $\Z[\alpha_1, \alpha_2]$. The algebra is a direct sum over all such $a,b$, and multiplication is given by canonical foam cobordisms.  The construction is analogous to that in \cite{KhFunctorValued}.

In Section \ref{sec:annular sl3 webs} we move to the $SL(3)$ setting. In this case, we consider a fixed \emph{sign string} $S = (s_1, \ldots, s_n) \in \{\pm 1\}^n$ and (oriented) non-elliptic\footnote{Sometimes these are called \emph{irreducible}, but we follow the terminology in \cite{Kup-spiders}.} $SL(3)$ webs in $\A$ whose boundary in $\SS^1\times \{1\}$ is given by $S$. Pairing two such webs $w,v$ results in a closed web $\b{v}w \subset \A$, and applying the TQFT from \cite{AK} yields a finitely-generated free graded $\Z[\alpha_1,\alpha_2, \alpha_3]$-module $\brak{\b{v}w}$. 

Unlike the $SL(2)$ setting, describing the set of non-elliptic annular $SL(3)$ webs with boundary $S$ is less straightforward and is a key technical part of this paper. 
Planar $SL(3)$ webs (in a disk) have been extensively studied, starting with Kuperberg \cite{Kup-spiders}. Of particular importance for us is \cite{KhKup}, which gives a bijection between non-elliptic webs in a disk and paths in the $SL(3)$ weight lattice from the origin to itself that are confined to the dominant Weyl chamber. Mutually inverse maps in each direction are explicit, given by algorithms called $\growalg$ (grow) and $\mcpalg$ (minimal cut path): 
\[
    \{ 
    \text{Dominant $SL(3)$ lattice paths from $0$ to $0$}
    \}
    \mathrel{\mathop{\myrightleftarrows{\rule{.75cm}{0cm}}}^{\growalg}_{\mcpalg}} \{
    \text{Non-elliptic $SL(3)$ webs in a disk}
    \}
\]

Our main result is an extension of this to annular webs. There are many more non-elliptic annular webs with boundary $S$ than there are non-elliptic planar webs with boundary $S$.  
We drop the dominant condition on lattice paths and establish the following, which is stated as Theorem \ref{thm:MCP-GROW-inverses} in Section \ref{sec:bijection between webs and paths}. 

\begin{theorem}
\label{thm:main sl3 thm}
    There are well-defined maps 
    \[
\{ 
\text{$SL(3)$ lattice paths from $0$ to $0$}
\}
\mathrel{\mathop{\myrightleftarrows{\rule{.75cm}{0cm}}}^{\growalg}_{\mcpalg}} \{
\text{Non-elliptic $SL(3)$ webs in an annulus}
\}
\]
which are mutually inverse bijections. 
\end{theorem}

In particular, for a fixed boundary condition $S$ (a sign string), the set of non-elliptic annular $SL(3)$ webs with boundary $S$ is finite, explicitly in bijection with a certain (evidently finite) collection of lattice paths.  With this in hand, we apply the annular $SL(3)$ TQFT from \cite{AK} to define the algebra $H^S_\A$ (guaranteed to be finitely-generated over $\Z[\alpha_1, \alpha_2, \alpha_3]$ by Theorem \ref{thm:main sl3 thm}). Multiplication is, as usual, given by canonical foam cobordisms. We then define bimodules for flat annular webs and complexes of bimodules for tangles in the thickened annulus in the natural way.  

We show in Proposition~\ref{prop:2-functor sl(2)} and Proposition~\ref{prop:2-functor sl(3)} that our constructions give $2$-functors from the $2$-categories of cobordisms between annular $SL(2)$ and $SL(3)$ webs, respectively, to the 2-category of bimodules for the annular web rings. 
This construction is then extended to invariants of annular tangles given by complexes of bimodules over these rings in the chain homotopy category (Theorems~\ref{thm_isotopy_invariance},~\ref{thm_isotopy_invariance sl3}). 
We hope to establish functoriality of our constructions under annular tangle cobordisms in $\A \times [0,1]^2$ in future work. 

\vspace{0.1in} 

\subsection*{Acknowledgements}
M.K. was partially supported by NSF grants DMS-2204033, DMS-2446892  and the Simons Collaboration Award 994328.

\section{Annular \texorpdfstring{$SL(2)$}{SL(2)} arc algebras and bimodules} 
\label{sec:annular arc algebras and bimodules}

\subsection{Equivariant annular \texorpdfstring{$SL(2)$}{SL(2)} homology}
\label{sec:Equivariant annular SL(2) homology}

We review the construction of equivariant annular $SL(2)$ homology via anchored foam evaluation introduced in \cite{AK}. 
Consider the graded ring $R_\alpha= \Z[\alpha_1, \alpha_2]$ with $\deg(\alpha_1) = \deg(\alpha_2)=2$.  

Let $\line \subset \R^3$ denote the $z$-axis. A \emph{closed anchored surface} is a closed, smoothly embedded surface $S\subset \R^3$ which intersects $\line$ transversely. 
The surface $S$ may be decorated by finitely many dots which are disjoint from $\line$ and can float freely on components of $S$. 
Intersection points of $S$ with $\line$ are called \emph{anchor points}, and we denote the set of anchor points by $p(S) = S\cap \line$. 
The surface $S$ moreover comes with a labeling of each anchor point by $1$ or $2$; that is, there is a fixed function $\ell : p(S) \to \{1,2\}$. 
Evaluation of anchored surfaces $\brak{S}\in R_\alpha$ was defined in \cite[Section 2.1]{AK}. 

Let $\P = \R^2\setminus \{0,0\}$ denote the punctured plane. 
Using universal construction and the evaluation $\brak{-}$, one can associate an $R_\alpha$-module, called the \emph{state space}, to every collection of finitely many disjoint simple closed curves in $\P$, as follows. 
Let $C_0, C_1 \subset \P$ be two collections of finitely many disjoint simple closed curves.  
An \emph{anchored cobordism from $C_0$ to $C_1$} is a smoothly and properly embedded dotted compact surface $S\subset \R^2\times [0,1]$ with boundary $\partial S = C_0 \sqcup C_1$, such that $C_i \subset \R^2\times \{i\}$ for $i=0,1$. Intersection points of $S$ with the line segment $\line_{[0,1]} := \{(0,0)\} \times [0,1]$ are still called anchor points and are denoted by $p(S)$. 
These intersections are required to be transverse and come equipped with a labeling $\ell: p(S) \to \{1,2\}$. 
We view $S$ as a morphism in an appropriate category to be defined shortly and write $S: C_0 \to C_1$. 
A closed anchored surface may be viewed as an anchored cobordism $\varnothing \to \varnothing$.  
Given $S: C_0\to C_1$, let $\b{S} : C_1\to C_0$ denote the anchored cobordism obtained by reflection through the horizontal plane $\R^2\times \{1/2\}$.  

Let $\Fr(C_0)$ denote the free $R_\alpha$-module generated by all anchored cobordisms $U : \varnothing \to C_0$. Given basis elements $U, V\in \Fr(C_0)$, let $\b{U} V$ denote the closed anchored foam obtained by gluing $\b{U}$ to $V$ along their common boundary $C$. 
Define the bilinear form 
\[
    (-,-) : \Fr(C_0) \times \Fr(C_0) \to R_\alpha
\]
by $(U,V) = \brak{\b{U}V}$. The definition of $\brak{-}$ readily implies that $(-,-)$ is symmetric. 
The \emph{state space of $C_0$}, denoted $\brak{C_0}$, is defined to be the quotient
\[
    \Fr(C_0)/\ker((-,-))
\]
where 
\[
    \ker((-,-)) = \{x\in \Fr(C_0) \mid (x,y) = 0 \text{ for all } y\in \Fr(C_0)\}
\]
is the kernel of $(-,-)$. For a basis element $U\in \Fr(C_0)$, we write $[U]$ to be its equivalence class in $\brak{C_0}$. 

The state space $\brak{C_0}$ is bigraded via the \emph{quantum grading} $\qdeg$ and the \emph{annular grading} $\adeg$, defined as follows. 
Given an anchored cobordism $S:C_0\to C_1$ with $d(S)$ dots, define 
\begin{equation}
\label{eq:qdeg sl2}
    \qdeg(S) = -\chi(S) + 2d(S) + \vert p(S) \vert.
\end{equation}
Write the anchor points of $S$ as $p_1, \ldots, p_m$, ordered from bottom to top, and let $n$ denote the number of non-contractible circles in $C_0$. Define 
\begin{equation}
\label{eq:adeg sl2}
    \adeg(S) = (-1)^n\sum_{i=1}^m (-1)^{i+ \ell(p_i)}.
\end{equation}
Both gradings are additive under composition of anchored cobordisms by \cite[Lemma 2.13]{AK}. 

We define a trivial annular grading on $R_\alpha$ by setting $\adeg(\alpha_1) = \adeg(\alpha_2) = 0$. 
Note that if $S$ is a closed anchored surface then bigradings are compatible with $\brak{-}$, in the sense that $\qdeg(S) = \deg(\brak{S})$ and $\adeg(S) = \adeg(\brak{S}) = 0$.  
As shown in \cite[Section 2.2]{AK}, these gradings descend to well-defined gradings on state spaces. 
In particular, consider an anchored cobordism $U: \varnothing \to C_0$ with anchor points $p_1, \ldots, p_m$ ordered from top to bottom. 
Then $p_i$ for odd $i$ contributes contributes  $1$ to $\adeg([U])$ if $\ell(p_i) = 1$ and contributes $-1$ if $\ell(p_i) = 2$; if $i$ is even then $p_i$ contributes $-1$  to $\adeg([U])$ if $\ell(p_i) = 1$ and contributes $1$ if $\ell(p_i) = 2$.
 
For a bigraded module $M = \oplus_{i,j} M_{i,j}$ over some commutative domain such that each $M_{i,j}$ has finite rank, define the \emph{graded rank} of $M$ to be $\sum_{i,j} \operatorname{rank}(M_{i,j}) q^i a^j$. 

\begin{theorem}\cite[Theorem 2.11]{AK}
\label{thm:state spaces SL(2)}
    Let $C\subset \P$ consist of $n$ contractible circles and $m$ non-contractible circles. 
    Then $\brak{C}$ is a free graded $R_\alpha$-module of graded rank $(q+q^{-1})^n (a+a^{-1})^m$. Moreover, a homogeneous basis (standard basis) for $\brak{C}$ is given by anchored cobordisms $\varnothing \to C$ such that 
    \begin{itemize}
        \item each component is a disk,
        \item each component with contractible boundary carries at most one dot and is disjoint from $\line$, and
        \item each component with non-contractible boundary intersects $\line$ exactly once. 
    \end{itemize}
\end{theorem}

For example, if $C$ consists of a single circle then the standard basis of $\brak{C}$ is shown in Figure \ref{fig:basis for circle SL(2)}. 

If $C$ is a single contractible circle, then $\brak{C}$ is a commutative Frobenius algebra which we denote by $A_\alpha$. We can identify 
\begin{equation}\label{eq_A_R}
    A_{\alpha} \ \cong \ R_\alpha[X]/((X-\alpha_1)(X - \alpha_2)),
\end{equation}
where $1$ and $X$ correspond to an undotted and once-dotted cup, respectively. 
This Frobenius algebra was studied in \cite{KRfrobext}. 

\begin{figure}
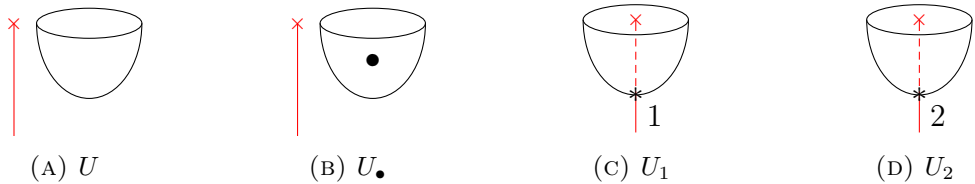

\centering
\subcaptionbox{ $U$
\label{fig:cup}}[.22\linewidth]
{\includestandalone{images/cup}
}
\subcaptionbox{ $U_\bullet$
\label{fig:dotted cup}}[.22\linewidth]
{\includestandalone{images/dotted_cup}
}
\subcaptionbox{ $U_1$
\label{fig:cup1}}[.22\linewidth]
{\includestandalone{images/cup_1}
}
\subcaptionbox{ $U_2$
\label{fig:cup2}}[.22\linewidth]
{\includestandalone{images/cup_2}
}
\caption{The standard basis elements for the state space of a single circle $C$. The first two surfaces $U, U_\bullet$ form a basis if $C$ is contractible, and the last two $U_1, U_2$ form a basis if $C$ is non-contractible.}\label{fig:basis for circle SL(2)}
\end{figure}

We denote by $\ACob$ the category whose objects are collections of finitely many disjoint simple closed curves in $\P$ and whose morphisms are anchored cobordisms considered up to ambient isotopy of $\R^2\times I$ which fixes $\partial (\R^2\times I)$ pointwise and maps $\line_{[0,1]}$ to itself. 
The composition $S_1 S_0$ of $S_0 :C_0 \to C_1$ and $S_1:C_1\to C_2$ is given by gluing $S_1$ to $S_0$ along their common boundary $C_1$ and rescaling in the thickening direction. 
The identity element $\id_C$ is the product cobordism $C\times I$. 
The universal construction assembles into a functor 
\begin{equation}\label{eq_funct}
    \brak{-} : \ACob \to R_\alpha \ggmod
\end{equation}
where $R_\alpha\ggmod$ denotes the category of bigraded $R_\alpha$-modules. 

An anchored cobordism $S:C_0\to C_1$ induces a map $\brak{S} :\brak{C_0} \to \brak{C_1}$, given  by $\brak{S}([U]) = [SU]$ for a basis element $U\subset \Fr(C_0)$. This assignment is evidently functorial, in the sense that $\brak{S_1 S_0} = \brak{S_1} \circ \brak{S_0}$ for $S_0 :C_0 \to C_1$ and $S_1: C_1\to C_2$, and $\brak{\id_C} = \id_{\brak{C}}$. 

Both $\ACob$ and $R_\alpha\ggmod$ are bigraded categories, and  \cite[Lemma 2.13]{AK} implies $\brak{-}$ is graded in the the sense that if $S:C_0\to C_1$ is an anchored cobordism, then $\brak{S}$ is a map of bidegree $(\qdeg(S), \adeg(S))$. 
Moreover, both $\qdeg$ and $\adeg$ are additive under composition. 

\begin{remark}
\label{rmk:ACob0}
    The construction in \cite{Akh} yields a functor $\ACob_0 \to R_\alpha \ggmod$, where $\ACob_0 $ denotes the subcategory of $\ACob$ consisting of anchored cobordisms which are disjoint from $\line$, i.e., cobordisms in $\P \times [0,1]$.  
    This functor and the restriction $\brak{-}\vert_{\ACob_0}$ (and hence the resulting link homologies) are naturally isomorphic by \cite[Theorem 2.20]{AK}. 

    The evaluation of anchored surfaces takes values in $R_\alpha$ rather than in its subring of symmetric polynomials. 
    Consequently, the construction of functor \eqref{eq_funct} (or of its restriction to $\ACob_0$) requires $R_{\alpha}$ as the ground ring, or a commutative ring $R$ equipped with a homomorphism $R_{\alpha}\lra R$. 
    In particular, functor \eqref{eq_funct} does not have an obvious modification with $R_{\alpha}$ replaced by its subring $R_E$ of symmetric polynomials, $R_E\cong \Z[\alpha_1+\alpha_2,\alpha_1\alpha_2]$, nor with the quotient ring of $R_E$ given by setting $\alpha_1+\alpha_2=0$. 
    Both of the latter rings  are commonly used ground rings for $SL(2)$ link homology theories for links in $\R^3$.  
    Similarly there is no obvious modification of the annular TQFT where the Frobenius algebra corresponding to a contractible circle is given by 
    $R_E[X]/(X^2 -E_1 X+ E_2)$. 
\end{remark}

It is useful to introduce the \emph{shifted dot} decoration $\circled{i}$ for $i=1,2$, originally considered in \cite{KRfrobext}. 
Shifted dots are disjoint from $\line$ and are allowed to float freely along the connected component on which they lie. 
A shifted dot $\circled{i}$ is the difference between a dot and $\alpha_i$, 
\begin{equation}
\label{eq:shifted dot}
\begin{aligned}
    \includestandalone{shifted_dot}.
\end{aligned}
\end{equation} 

There are four types of \emph{elementary saddles}, shown in Figure \ref{fig:elementary saddles}. Note that each of them is disjoint from $\line$. 
We  recall the corresponding maps from \cite[Examples 2.14--2.17]{AK} in terms of the standard basis from Theorem \ref{thm:state spaces SL(2)}.

\begin{figure}
\centering 
\subcaptionbox{Type A \label{fig:el cob A}}[.22\linewidth]
{\includegraphics{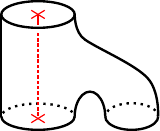}}
\subcaptionbox{Type B \label{fig:el cob B}}[.22\linewidth]
{\includegraphics{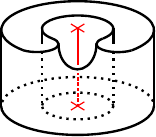}
}
\subcaptionbox{Type C \label{fig:el cob C}}[.22\linewidth]
{\includegraphics{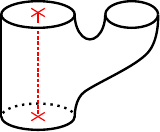}
}
\subcaptionbox{Type D \label{fig:el cob D}}[.22\linewidth]
{\includegraphics{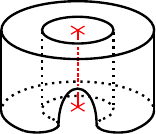}
}
\caption{Elementary saddles involving non-contractible circles.}\label{fig:elementary saddles}
\end{figure}

Figure \ref{fig:el cob A} map:
\begin{equation}
\begin{aligned}
\label{eq:type A}
     \includestandalone{type_A}   
\end{aligned}
\end{equation}

Figure \ref{fig:el cob B} map:
\begin{equation}
\begin{aligned}
\label{eq:type B}
     \includestandalone{type_B}   
\end{aligned}
\end{equation}

Figure \ref{fig:el cob C} map:
\begin{equation}
\begin{aligned}
\label{eq:type C}
     \includestandalone{type_C}   
\end{aligned}
\end{equation}

Figure \ref{fig:el cob D} map:
\begin{equation}
\begin{aligned}
\label{eq:type D}
     \includestandalone{type_D}   
\end{aligned}
\end{equation}

\subsection{Annular \texorpdfstring{$SL(2)$}{SL(2)} arc algebras} 

Let $\A = \SS^1\times I$ denote the annulus. For $i=0,1$, we set $\partial_i \A = \SS^1 \times \{i\}$. 
For each $n\geq 0$,  fix a collection of $2n$ marked points in $\SS^1$.  By a \emph{flat annular $(n,m)$-tangle} we mean a smooth and proper embedding of a compact $1$-manifold $T$ into $\A$ such that exactly $2m$ points of $\partial T$ are mapped to the $2m$ fixed points in $\SS^1\cong \partial_0 \A$ and $2n$ points of $\partial T$ are mapped to the $2n$ fixed points in $\SS^1\cong \partial_1 \A$.  
Two flat annular tangles $T_1$ and $T_2$ are isotopic if there is an orientation-preserving diffeomorphism of $\A$ which fixes $\partial \A$ pointwise, is isotopic to the identity, and takes $T_1$ to $T_2$. 

\begin{definition}
\label{def:various sets of annular tangles}
    Let $\Bnoisotopy^n_m$ be the space of flat annular $(n,m)$-tangles, and let $\Buptoisotopy^n_m$ denote a set of representatives for each planar isotopy class. 
    Write $\Bmatchings^n$ to be a set of representatives for isotopy classes of $(n,0)$-tangles which do not contain closed components (note that $\Bmatchings^n \neq \Buptoisotopy^n_0$, since tangles in the latter may contain closed components).  
\end{definition}
  
We will draw elements of $\Bmatchings^n$ in a punctured disk, with the boundary of the punctured disk corresponding to $\partial_1 \A$. For example, $\Bmatchings^1$ consists of two elements, shown in \eqref{eq:H1_tangles}. 
\begin{equation}
\label{eq:H1_tangles}
    \begin{aligned}
        \includestandalone{images/H1_tangles}
    \end{aligned}
\end{equation}

\begin{lemma}
    The cardinality of $\Bmatchings^n$ is $\binom{2n}{n}$. 
\end{lemma}
\begin{proof}
    The number of crossingless matchings of $2n$ points on the boundary of a disk is the $n$-th Catalan number $\frac{1}{n+1}\binom{2n}{n}$. 
    Each such crossingless matching cuts the disk into $n+1$ regions. 
    Placing the puncture in one of these regions produces an element of $\Bmatchings^n$, and moreover every element of $B^n$ is uniquely obtained in this way. 
\end{proof}

Denote by $\id_n \in \Bnoisotopy_n^n$ the flat annular tangle which is the direct product of the $2n$ marked points on $\SS^1$ and the interval $I$. 
For $T_1\in \Bnoisotopy_n^m$, let $\b{T_1} \in \Bnoisotopy_m^n$ denote the tangle obtained by reflecting $T_1$ through $\SS^1\times \{1/2\} \subset \A$. 
Given $T_2 \in \Bnoisotopy_m^k$, let $T_2 T_1 \in \Bnoisotopy_n^k$ denote the tangle obtained by stacking $T_2$ on top of $T_1$. 
This composition descends to a composition on the corresponding isotopy classes (rel boundary) of annular flat tangles. 

In what follows, we consider the state space $\brak{\b{b}a}$ for $a,b\in \Bmatchings^n$ as graded only via $\qdeg$. 
Let $\{k\}$ denote an upwards grading shift by $k\in \Z$. 

\begin{definition}
    Define $H^n_\A := \bigoplus\limits_{a,b \in B^n} \brak{\b{b} a } \{ n\}$. 
    Multiplication  $H^n_\A \o_{R_\alpha} H^n_\A \to H^n_\A$ on the direct summand $\brak{\b{d}c}\{n\} \o_{R_\alpha} \brak{\b{b} a} \{n\} $ is zero if $c\neq b$.  
    For $a,b,c\in B^n$, there is an evident minimal cobordism $\b{c} b \b{b} a\to \b{c} a$ in $\A\times I$ that is the union of the cobordism $b\b{b} \to \id_n$ shown in Figure \ref{fig:multiplication cobordism} and the product cobordisms $\b{c} \times I$ and $a \times I$. 
    Since $\brak{\b{c} b \b{b} a} \cong \brak{\b{c}b} \o_{R_\alpha} \brak{\b{b}a}$ canonically, this cobordism yields a degree-zero map 
    \[
    \brak{\b{c}b}\{n\} \o_{R_\alpha} \brak{\b{b}a} \{n\} \to \brak{\b{c}a} \{n\}
    \]
    which defines the multiplication on the direct summand $\brak{\b{c}b}\{n\} \o_{R_\alpha} \brak{\b{b} a}\{n\}$ of $H^n_\A \o_{R_\alpha} H^n_\A$. 
    That multiplication is associative follows from far-commutativity of cobordisms. 

    For each $a\in B^n$, $\b{a}a$ consists of $n$ contractible circles. 
    Let $1_a \in \brak{\b{a}a}\{n\}$ denote the basis element consisting of a disjoint union of undotted cup cobordisms which are all disjoint from $\line$. 
    It is straightforward to see that $\{1_a \mid a\in B^n\}$ consists of pairwise orthogonal idempotents, and that the unit for $H^n_\A$ is given by $1 :=\sum_{a\in B^n} 1_a$.
\end{definition}

\begin{figure}
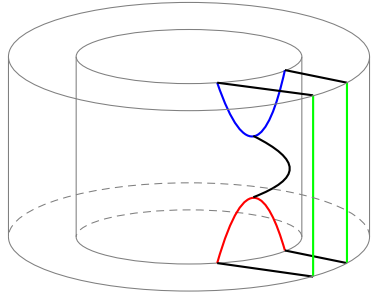

    \centering
    \includestandalone{multiplication_cobordism}
    \caption{The annular cobordism (with corners) from $b\b{b}$ to $\id_1$ when $b$ consists of a single component. We draw $\A$ vertically; $b$ is blue, $\b{b}$ is red, and $\id_1$ is green. When $b$ has multiple components, the cobordism is a disjoint union of saddle cobordisms of this form.}
    \label{fig:multiplication cobordism}
\end{figure}

The algebra $H^n_\A$ is a free graded (with respect to $\qdeg$) $R_\alpha$-module  supported in non-negative gradings. 
Note that each $1_a$ is in quantum grading zero. Moreover, the idempotents $1_a$ over all $a\in B^n$ form an $R_\alpha$-basis for the subalgebra $\left( H^n_\A \right)_0$ of $H^n_\A$ consisting of elements in quantum grading zero. 
Consequently, $\left( H^n_\A \right)_0$, viewed as a ring, is a finite direct product of copies of $R_\alpha$. We also have $1_b H^n_\A  1_a = \brak{\b{b} a}\{n\}$. 
By a \emph{standard basis element of $H^n_\A$} we mean a standard basis element in some summand $\brak{\b{b}a}\{n\}$. 

\begin{remark}
\label{rmk_Anno_Tubbenhauer}
The first version of the annular arc algebras was introduced by Anno and Nandakumar~\cite{Anno-Nandakumar}. 
Ehrig and Tubbenhauer~\cite{EhrigTubbenhauer} introduced and studied a related annular arc algebra based on the APS (Asaeda-Przytycki-Sikora) TQFT \cite{APS}. 

In the ring $R_{\alpha}$ there is the ideal $I_{\alpha}$ of positive degree elements, and $R_{\alpha}/I_{\alpha}\cong \Z$. Our annular TQFT, when restricted to $\ACob_0$ (see Remark \ref{rmk:ACob0}) and modded out by this ideal of the ground ring $R_{\alpha}$, is isomorphic to the APS TQFT \cite{APS} by \cite[Theorem 1.1]{Akh}. 
This leads to an isomorphism of graded rings between the quotient $H^n_{\A}/I_{\alpha}H^n_{\A}$ and the Ehrig-Tubbenhauer annular arc ring in~\cite[Section 5]{EhrigTubbenhauer}. 
\end{remark}

\begin{example}
\label{ex:H for two points}
The set $B^1$ consists of two elements $a$ and $b$, shown in \eqref{eq:H1_tangles}.

The closures $\b{a}a$ and $\b{b}b$ each consist of one contractible circle, and the closures $\b{b}a$ and $\b{a}b$ each consist of one non-contractible circle.  
Let 
$1_r$ and $x_r$ denote the standard basis elements in $\brak{\b{r}r}$, for $r\in \{a,b\}$, consisting of an undotted and once-dotted cup, respectively, and let $y_{r,s}^i$ for $r,s\in \{a,b\}$, $r\neq s$, and $i=1,2$, denote the standard basis element in $\brak{\b{r}s}$ consisting of a cup with one anchor point labeled $i$ (note that $\b{r}s$ is a non-contractible circle in the annulus). 
We have 
\begin{align*}
    1_r \cdot 1_r &= 1_r, \\ 
    1_r \cdot x_r &= x_r \cdot 1_r = x_r, \\ 
    x_r \cdot x_r &= (\alpha_1 + \alpha_2) x_r - \alpha_1\alpha_2 1_r,
\end{align*}
where the last equation can be rewritten as $(x_r-\alpha_1 1_r)(x_r - \alpha_2 1_r)=0$.   
These can be deduced from the multiplication in the Frobenius algebra $A_\alpha$, see \eqref{eq_A_R}. 

Let $\tau$ denote the nontrivial involution of $\{1,2\}$. Then by \eqref{eq:type A} and \eqref{eq:type B},
\begin{align*}
    1_r \cdot y^i_{r,s} & = y^i_{r,s} = y^i_{r,s}\cdot 1_s, \\
    x_r \cdot y^i_{r,s} &= \alpha_{\tau(i)} y^i_{r,s} = y^i_{r,s} \cdot x_s, \\
    y^i_{s,r} \cdot y^i_{r,s} & = x_s - \alpha_i 1_s, \\
    y^i_{s,r} \cdot y^j_{r,s} &= 0 \text{ if } i\neq j,
\end{align*}
giving us the multiplication table for the standard basis of the $R_{\alpha}$-algebra $H^1_\A$, which is a free $R_{\alpha}$-module of rank $8$.  
\end{example}

The degree shifts $\{n\}$ ensure that the quantum grading of $H^n_\A$ is additive under multiplication. 
Multiplication changes the annular grading as follows. 
For $x\in \brak{\b{b}a}$, let $w(x)$ denote the number of non-contractible circles in $\b{b}a$. 
Given $x\in \brak{\b{c}b}$  and $y\in \brak{\b{b}a}$, we have  
\[
\adeg(xy) = \adeg(y) + (-1)^{w(y)} \adeg(x).
\]
The sign is due to the alternating nature of how $\adeg$ is computed. 
Note that $w(xy) \equiv w(x) + w(y) \mod 2$.

For $a\in B^n$, let $\eps_a : \brak{\b{a}a}\{n\}\to R_\alpha$ be the map that sends a standard basis element $x$ to $0$ if $x$ contains an undotted cup, and otherwise $\eps_a(x) = 1$. 
Extend this to $\eps : H^n_\A\to R_\alpha$ by setting $\eps(x) = 0$ if $x\in \brak{\b{b}a}\{n\}$ with $b\neq a$. 
Note that $\deg(\eps)=-2n$.

\begin{proposition} The trace map $\eps$ turns $H^n_\A$ into a graded symmetric Frobenius $R_{\alpha}$-algebra. 
\end{proposition}

\begin{proof}
    The comultiplication map $\Delta:H^n_\A \lra H^n_\A\otimes_{R_\alpha} H^n_\A$ is the sum of maps 
    \[
    \brak{\b{b}a}\lra \brak{\b{b}c}\otimes_{R_\alpha} \brak{\b{c}a}
    \]
    over all $a,b,c\in B^n$ given by the standard cobordisms that unwrap $2n$ parallel lines annular flat tangle into the composition $c\b{c}$. 
    This map is dual to the multiplication via $\eps$, implying that $\eps$ is nondegenerate. 
    It is clear that $\eps(xy)=\eps(yx)$ for all $x,y\in H^n_\A$, i.e., the trace is symmetric.    
\end{proof}

We list some symmetries of $H^n_\A$ below.

\begin{itemize}

\item Natural isomorphisms of state spaces $\brak{\overline{b}a}\cong \brak{\overline{a}b}$ induce an isomorphism of graded symmetric Frobenius $R_{\alpha}$-algebras $(H^n_\A)^{\mathsf{op}}\cong H^n_{\A}.$ 

\item Graded $R_{\alpha}$-algebra $H^n_\A$ admits an automorphism of order $2n$ for $n>0$ (order $1$ for $n=0$) given by rotating each diagram $a\in B^n$ by $\pi/n$ counterclockwise.

\item Additionally, $H^n_\A$ admits a reflection involution given by picking a pair of opposite boundary points on the circle (other than the $2n$ marked points) and reflecting each diagram in $B^n$ via this involution. 
Jointly with the rotation automorphisms, one obtains an action of the $4n$-element dihedral group on the $R_{\alpha}$-algebra $H^n_\A$. 

\item Let $\tau$ be the involution of $R_\alpha$ which interchanges $\alpha_1$ and $\alpha_2$.  
Define $\tau_n : H^n_\A \to H^n_\A$ to be the ring automorphism induced by the map which interchanges the labels $1$ and $2$ of each anchor point of a cobordism. 
Automorphism $\tau_n$ of $H^n_\A$ intertwines the involution  $\tau$ of $R_{\alpha}$, in the sense that $\tau_n(xy)=\tau(x)\tau_n(y)$ for $x\in R_{\alpha}$ and $y\in H^n_\A$. 

\end{itemize}

\subsection{Generation by degree-one elements }

Our goal in this subsection is to establish Theorem \ref{thm:generated by degree 1}, which says that $H^n_\A$ is generated elements of degree zero and one. 
Note, the planar arc algebras $H^n$ from \cite{KhFunctorValued} do not have this property.  

\begin{definition}
    Given $a\in B^n$ and a component $\gamma$ of $a$, let 
    \[
    d_a(\gamma) = \min\{ \operatorname{int}(r) \cap a \mid r \text{ is an embedded arc joining a point on } \gamma \text{ to } \times \}
    \] 
    be the number of components of $a$ that separate $\gamma$ from $\times$. 
    We say that $\gamma$ is \emph{outermost} if $d_a(\gamma) = 0$. 
\end{definition}

Let $a\in B^n$ and consider a component $\gamma$ of $a$. 
There is a unique interval $I_\gamma \subset \SS^1 \times \{1\}$ such that $\gamma \cup I_\gamma$ bounds a disk in $\A$. 
Note that if $\gamma'$ is another component of $a$, either $I_\gamma \cap I_{\gamma'} = \varnothing$ or one of $I_\gamma, I_{\gamma'}$ is properly contained in the other. 
If $I_\gamma \subset I_{\gamma'}$, then we write $\gamma < \gamma'$. 
This defines a partial order on the components of $a$. We say $\gamma,\gamma'$ are \emph{nested} if $\gamma < \gamma'$ or $\gamma' < \gamma$. Note that $d_a(\gamma)$ is equal to the maximum among all $k$ such that there exists a chain $\gamma < \gamma_1  < \cdots < \gamma_k$; in particular, $\gamma$ is outermost if and only if it is maximal with respect to $<$. 

By a \emph{(surgery) arc} in $a$ we will always mean an embedded closed interval $r\subset \A$ whose interior is disjoint from $a$ and whose endpoints lie on distinct components $\gamma_1, \gamma_2$ of $a$. 
In this case we say that $\gamma_1$ and $\gamma_2$ are \emph{adjacent}. Note that surgery along $r$ results in another element of $B^n$. 
We also note that if $\gamma_1, \gamma_2$ are components of $a$ with $\gamma_1<\gamma_2$, then $\gamma_1$ and $\gamma_2$ are adjacent if and only if there does not exist a component $\gamma$ of $a$ with $\gamma_1<\gamma<\gamma_2$. 

There are two ways that degree $1$ elements can appear in $H^n_\A$. 

\begin{definition}
\label{def:flip}
    Let $a\in B^n$ and suppose $\gamma$  is an outermost component of $a$. The \emph{flip of $a$ along $\gamma$}, denoted $f_\gamma(a)\in B^n$, is the flat annular tangle obtained by sliding $\gamma$ across the puncture and keeping all other components of $a$ unchanged.\footnote{Equivalently, draw a small non-contractible circle $z$ around $\times$ and join $\gamma$ to $z$ via an arc; then $f_\gamma(a)$ is obtained by surgering along this arc.}    
\end{definition} 

This is depicted in \eqref{eq:flip ex}.
\begin{equation}
\label{eq:flip ex}
    \begin{aligned}
        \includestandalone{flip_ex}
    \end{aligned}
\end{equation}

Setting $b=f_\gamma(a)$, we see that $\b{b}a$ has $n-1$ contractible circles and one non-contractible circle. 
Then $\brak{\b{b}a}\{n\}$ has exactly two standard basis elements of degree $1$, which we denote by $y_{a, \gamma}^i$ for $i\in \{1,2\}$, where $i$ corresponds to the label of the unique non-contractible circle. 
If $\gamma' \subset b = f_\gamma(a)$ is the component obtained by moving $\gamma$, then we also have degree $1$ elements $y_{b,\gamma'}^j \in \brak{\b{a} b}\{n\}$, for $j\in \{1,2\}$. Let $x_{a,\gamma} \in \brak{\b{a}a}\{n\}$ denote the standard generator in which the circle containing $\gamma$ is dotted and all other circles are undotted. 
Then \eqref{eq:type B} yields
\begin{equation}
    \label{eq:product of essential degree 1 generators}
     y_{b,\gamma'}^j \cdot y_{a,\gamma}^i = \delta_{ij} (x_{a,\gamma} - \alpha_ i 1_a). 
\end{equation}

\begin{definition}
\label{def:swap}
Let $\gamma_1$ and $\gamma_2$ be two components of $a$, and suppose there is an arc $r$ which joins a point on $\gamma_1$ to a point on $\gamma_2$. 
The \emph{swap of $a$ along $r$}, denoted $s_{r}(a) \in B^n$, is the flat annular tangle obtained by surgery along $r$. 
\end{definition}

This local modification is shown in \eqref{eq:swap ex}.
\begin{equation}
\label{eq:swap ex}
    \begin{aligned}
        \includestandalone{swap_ex}
    \end{aligned}
\end{equation}

\begin{remark}
    If $\gamma_1$ and $\gamma_2$ are adjacent and nested then there is essentially one such arc $r$. 
    Surgery along $r$ results in two components which are not nested. 
    On the other hand, if $\gamma_1$ and $\gamma_2$ are adjacent but not nested, then there are essentially two choices of $r$. 
    Surgery along either choice results in nested components.
\end{remark}

Setting $b=s_r(a)$, we see that $\b{b}a$ has exactly $n-1$ contractible circles and no non-contractible circles. 
Then $\brak{\b{b}a}\{n\}$ has exactly one standard generator of degree $1$ consisting of an undotted cup on each circle; we denote this element by $y_{a,r}$. 
There is a dual surgery arc $r'$ in $b$ which joins the two components resulting from surgery along $r$, so that $s_{r'}(b) = a$. 
As above, let $x_{a,\gamma_i}$ for $i=1,2$ denote the standard generator of $\brak{\b{a} a}\{n\}$ which is a dotted cup on the circle containing $\gamma_i$ and is an undotted cup on all other circles.  
We have 
\begin{equation}
\label{eq:product of contractible degree 1 generators}
    y_{b,r'} \cdot y_{a,r} = x_{a, \gamma_1} + x_{a, \gamma_2}  - (\alpha_1+\alpha_2) 1_a.
\end{equation}

\begin{lemma}
The set of all possible elements of the form $y_{a,\gamma}^i$ and $y_{a,r}$ is a basis for $\left( H^n_\A\right)_1$, the subgroup of $H^n_\A$ consisting of degree-one elements. 
\end{lemma}

\begin{proof}
    Let $x\in \brak{\b{b}a}\{n\}$ be a standard basis element with $\qdeg(x) = 1$. Let $k_0$ and $k_1$ be the number of contractible and non-contractible circles in $\b{b}a$, respectively. If $d\geq 0$ is the number of dotted cups with contractible boundary in $x$, then $1= \qdeg(x) = n-k_0+2d$, which implies $k_0=n-1$ and $d=0$. Since $k_0  + k_1 \leq n$, we also have $k_1 = 0$ or $k_1=1$.
    
    If $k_1 = 1$ then $\b{b}a$ contains $n$ circles, so if two boundary points are joined by a component of $a$ then they are also joined by a component of $b$. For $\b{b}a$ to contain exactly one non-contractible circle, we must have that $b$ is the flip of $a$ along some outermost component $\gamma$ of $a$ and that $x$ is of the form $y^i_{a,\gamma}$ for some $i \in \{1,2\}$.
    
    If $k_0 = 0$ then $\b{b}a$ contains $n-1$ circles, all of which are contractible. Then there are components $\gamma_1, \gamma_2$ of $a$, components $\gamma_1', \gamma_2'$ of $b$, and four boundary points $p_1, p_2, p_3, p_4$ such that $\gamma_1$ joins $p_1$ and $p_2$, $\gamma_2$ joins $p_3$ and $p_4$, $\gamma_1'$ joins $p_1$ and $p_3$, and $\gamma_2'$ joints $p_2$ and $p_4$. Moreover, $\gamma_1$ and $\gamma_2$ are adjacent, $\gamma_1'$ and $\gamma_2'$ are adjacent, and $a$ and $b$ are otherwise identical. Since the circle $\gamma_1 \cup \gamma_2 \cup \gamma_1' \cup \gamma_2'$ is contractible, we see that $b=s_r(a)$ for some arc $r$ joining $\gamma_1$ to $\gamma_2$. 
\end{proof}

\begin{definition}
    Let $\Hdegreeone^n_\A$ denote the $R_\alpha$-subalgebra of $H^n_\A$ generated by the degree-zero and degree-one elements.
\end{definition}

\begin{lemma}
\label{lem:aa is generated by degree 1}
   For every $a\in B^n$, the $R_\alpha$-submodule $\brak{\b{a}a}\{n\} \subset H^n_\A$ is contained in $\Hdegreeone^n_\A$. 
\end{lemma}

\begin{proof}
    Note that $\b{a} a$ consists of $n$ contractible circles and that each circle corresponds uniquely to a component of $a$. 
    For a component $\gamma$ of $a$,  let $x_{a,\gamma} \in \brak{\b{a}a}\{n\}$ denote the standard generator which is a dotted cup on the circle containing $\gamma$ and is an undotted cup on every other circle. 
    It suffices to show that each $x_{a,\gamma}\in \Hdegreeone^n_\A$, since if $x\in \brak{ \b{a} a}\{n\}$ is a standard generator in  which the circles corresponding to the components $\gamma_1, \ldots, \gamma_k$ are dotted\footnote{If no circles in $x$ are dotted then $x=1_a$ is in degree zero and trivially an element of $\Hdegreeone^n_\A$.} and the other circles are undotted, then  $x$ is the product
    \[
        x = x_{a,\gamma_1} \cdots x_{a,\gamma_k}.
    \]

    We show that  $x_{a,\gamma}\in \Hdegreeone^n_\A$, by induction on $d_a(\gamma)$. 
    If $d_a(\gamma) =0$ then $\gamma$ is outermost, and we write $b= f_{\gamma}(a)$ and let $\gamma'\subset b$ be the component obtained by flipping $\gamma$. 
    Then \eqref{eq:product of essential degree 1 generators} yields
    \[
        x_{a,\gamma} = y_{b,\gamma'}^1 \cdot y_{a,\gamma}^1 + \alpha_1 1_a \in \Hdegreeone^n_\A.
    \]

    For the inductive step, suppose $d_a(\gamma) >0$, and let $\gamma'$ be a component of $a$ with $\gamma < \gamma'$ and such $\gamma$ and $\gamma'$ are adjacent. 
    Let $b=s_r(a)$ be the flip of $a$ along an arc $a$ joining $\gamma$ to $\gamma'$ and let $r'$ be the dual arc in $b$. By \eqref{eq:product of contractible degree 1 generators}, 
    \[
        x_{a,\gamma} =  y_{b,r'}\cdot y_{a,r} - x_{a,\gamma'} + (\alpha_1+\alpha_2) 1_a,
    \]
    where $y_{b,r'}, y_{a,r}, 1_a\in \Hdegreeone^n_\A$ by definition and $x_{a,\gamma'}\in \Hdegreeone^n_\A$ by inductive hypothesis. 
    \end{proof}

A circle $C$ in $\b{b}a$ is the union of some components $\gamma_1, \ldots, \gamma_k \subset a$ and $\gamma_1',\ldots, \gamma_k' \subset b$, with $k\geq 1$. We call these components \emph{segments} of $C$.

\begin{corollary}
\label{cor:adding dots}
    Let $a,b\in B^n$ and let $C\subset \b{b}a$ be a contractible circle. 
    Suppose $y\in \brak{\b{b}{a}}\{n\}$ is a standard generator in which $C$ is undotted. Let $y' \in \brak{\b{b}{a}}\{n\}$ denote the standard generator which agrees with $y$ except on $C$, where $y'$ is dotted cup. 
    If $y\in \Hdegreeone^n_\A$ then $y'\in \Hdegreeone^n_\A$ as well. 
\end{corollary}

\begin{proof}
    Let $\gamma \subset a$ be a segment of $C$, and let $z\in \brak{\b{a}a}\{n\}$ be the standard generator in which the circle containing $\gamma$ is dotted and every other circle is undotted. 
    By Lemma \ref{lem:aa is generated by degree 1}, $z\in \Hdegreeone^n_\A$.  
    Then $y' = y \cdot z \in \Hdegreeone^n_\A$. 
\end{proof}

\begin{definition}
    A \emph{reducing arc} for a contractible circle $C$ in $\b{b}a$ is an arc $r$ in either $a$ or in $b$ which joins two segments $\gamma_1, \gamma_2$ of $C$, such that surgery along $r$ splits $C$ into two contractible circles. 
    We refer to $(r,\gamma_1, \gamma_2, C)$ as a reducing system. 
\end{definition}

\begin{example}
Consider $a,b\in B^4$ as shown in \eqref{eq:reducing arc}.
\begin{equation}
\label{eq:reducing arc}
    \begin{aligned}
        \includestandalone{reducing_arc}
    \end{aligned}
\end{equation}
We see that $\b{b}a$ consists of two nested contractible circles. 
Denote them by $C_1$ and $C_2$, with $C_2$ innermost. 
Then $r_1$ is a reducing arc for $C_1$ and $r_2$ is a reducing arc for $C_2$. 
Note that $r_3$ is not a reducing arc, since surgery along $r_3$ splits $C_1$ into two non-contractible circles. 
\end{example}

\begin{lemma}
\label{lem:reducing arcs exist}
    If $\b{b}a$ contains a contractible circle consisting of more than two segments then $\b{b}a$ contains a reducing arc. 
\end{lemma}

\begin{proof}
    Consider all contractible circles in $\b{b}a$ which are comprised of more than two segments, and pick an innermost one $C$ among them. Let $a'$ (resp. $b'$) denote the flat annular tangle obtained by deleting all components of $a$ (resp. of $b$) which are not contained in $C$. 

    Note that $a'$ and $b'$ each have at least two components. Moreover, at least one of $a'$ or $b'$ must contain a pair of nested components, otherwise $\b{b'}a'=C$ cannot consist of a single contractible circle. Without loss of generality, assume $a'$ contains a pair of nested components. We may then find adjacent components $\gamma_1, \gamma_2$ of $a'$ with $\gamma_1< \gamma_2$ and with $\gamma_2$ outermost. Let $r$ be an arc in $a'$ joining $\gamma_1$ to $\gamma_2$. Note that $C$ bounds a disk $D$ in $\A$, and that $\gamma_2$ is outermost implies that $r \subset D$; see Figure \ref{fig:reducing arc exists schematic 1}. It follows that surgery along $r$ splits $C$ into two trivial circles. 
    
    Finally, we show that $\gamma_1$ and $\gamma_2$ are adjacent in $a$, which completes the proof. Suppose that $\gamma_1 < \gamma < \gamma_2$ for some component $\gamma \subset a$. Then $D$ intersects $\gamma$ (see Figure \ref{fig:reducing arc exists schematic 2} for a schematic), so  the circle $C'\subset \b{b}a$ containing $\gamma$ is contractible. 
    Moreover, $C'$ must be nested inside $C$ and cannot consist of only two segments, which contradicts the innermost assumption on $C$. 
\end{proof}

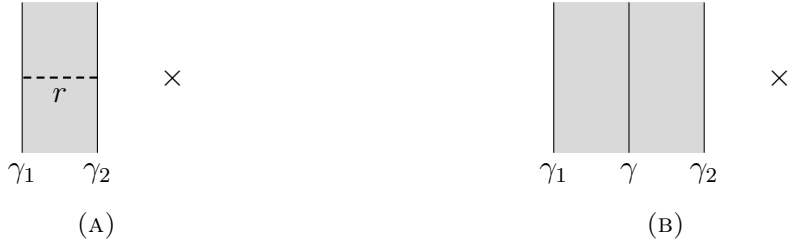
\begin{figure}
\centering 
\subcaptionbox{\label{fig:reducing arc exists schematic 1}}[.45\linewidth]
{
    \begin{tikzpicture}
        \node at (0,0) {$\times$};

\fill[gray,opacity=0.3] (-1,-1) -- (-2,-1) -- (-2,1) -- (-1,1) -- cycle;
        \draw (-1,-1) -- (-1,1);

        \draw (-2,-1) -- (-2,1); 

\draw[thick, densely dashed] (-1,0) -- (-2,0);

        \node[below] at (-1,-1) {$\gamma_2$};
        \node[below] at (-2,-1) {$\gamma_1$};
        \node[below] at (-1.5,0) {$r$};
        
    \end{tikzpicture}
}
\subcaptionbox{\label{fig:reducing arc exists schematic 2}}[.45\linewidth]
{
    \begin{tikzpicture}
        \node at (0,0) {$\times$};
\fill[gray,opacity=0.3] (-1,-1) -- (-3,-1) -- (-3,1) -- (-1,1) -- cycle;
        \draw (-1,-1) -- (-1,1);

        \draw (-2,-1) -- (-2,1); 

        \draw (-3,-1) -- (-3,1);

        \node[below] at (-1,-1) {$\gamma_2$};
        \node[below] at (-2,-1) {$\gamma$}; 
        \node[below] at (-3,-1) {$\gamma_1$};
    \end{tikzpicture}
}
\caption{Schematics for the proof of Lemma \ref{lem:reducing arcs exist}. The shaded regions depict portions of the disk $D$. }\label{fig:reducing arcs exist}
\end{figure}

\begin{lemma} 
\label{lem:reducing arcs reduce}
Suppose $(r,\gamma_1, \gamma_2, C)$ is a reducing system in $\b{b}a$, and let $z\in \brak{\b{b}a}\{n\}$ be a standard generator in which the cup on $C$ is undotted. Then $z = wy$ for some standard generators $w,y\in H^n_\A$ such that at least one of $w$ or $y$ is in degree $1$. 
\end{lemma}

\begin{proof}
 Without loss of generality, suppose $\gamma_1, \gamma_2$ lie in $a$, and let $c= s_r(a)$. Note that $\b{b}a$ is identical to $\b{b}c$ except that $C$ has been split into two contractible circles $C_1, C_2$. Let $w \in \brak{\b{b}c}\{n\}$ be the standard generator which agrees with $z$ on all the common circles and is an undotted cup on $C_1$ and $C_2$. Let $y = y_{a,r} \in \brak{\b{c}a}\{n\}$ be the unique degree $1$ standard generator. Then $z=wy$. 
\end{proof}

\begin{theorem}
\label{thm:generated by degree 1}
The $R_{\alpha}$-algebra $H^n_\A$ is generated by degree-zero and degree-one elements. That is, $H^n_\A = \Hdegreeone^n_\A$. 
\end{theorem}

\begin{proof}
Let $a,b\in B^n$ and let $z \in \brak{\b{b}a}\{n\}$ be a standard generator. By Corollary \ref{cor:adding dots} we may assume every contractible circle in $z$ is undotted. We show that $z\in \Hdegreeone^n_\A$ by induction on $\qdeg(z)$. If $\qdeg(z) \leq 1$ then there is nothing to show. Suppose then that $\qdeg(z)>1$.  

If $\b{b}a$ consists entirely of contractible circles, then necessarily $\b{b}a$ contains a circle comprised of more than two segments. Lemma \ref{lem:reducing arcs exist}, Lemma \ref{lem:reducing arcs reduce}, and the inductive hypothesis imply that $z\in \Hdegreeone^n_\A$. We may then assume that $\b{b}a$ contains at least one non-contractible circle.

We consider two cases. First, suppose there is a  component $\gamma$ of either $a$ or $b$ which is both outermost and contained in a non-contractible circle $C$ in $\b{b}a$. Without loss of generality, assume $\gamma\subset a$. Let $c=f_\gamma(a)$ and let $y^i_{a,\gamma} \in \brak{\b{c}a}\{n\}$ denote the corresponding degree-one generator, where $i\in \{1,2\}$ is the label on $C$ in $z$. Let $\gamma'\subset c$ be the component corresponding to flipping $\gamma$. Note that $\b{b}a$ and $\b{b}c$ are identical except the circle $C$ in $\b{b}a$ has transformed into a contractible circle $C'$ in $\b{b}c$. Let $w\in \brak{\b{b}c}\{n\}$ denote the standard generator which is an undotted cup on $C'$ and agrees with $z$ on all the common circles between $\b{b}a$ and $\b{b}c$. Then $\qdeg(w) = \qdeg(z)-1$, so $w\in \Hdegreeone^n_\A$ by inductive hypothesis. Moreover, $z=w\cdot y^i_{a,\gamma}$, which completes the proof in this case. 

Finally we consider the case where every outermost component in both $a$ and $b$ is contained in a contractible circle in $\b{b}a$. Since $\b{b}a$ contains a non-contractible circle, there exist components $\gamma_1, \gamma_2$ of $a$ such that $\gamma_1 < \gamma_2$, $\gamma_1$ lies on a non-contractible circle, and $\gamma_2$ lies on a contractible circle. Then the circle in $\b{b}a$ containing $\gamma_2$ must be comprised of more than two segments. Applying Lemma \ref{lem:reducing arcs exist}, Lemma \ref{lem:reducing arcs reduce}, and the inductive hypothesis again yields $z\in \Hdegreeone^n_\A$. 
\end{proof}

\subsection{Some relations} 

Below we list some relations, all quadratic, on the generating elements of $H^n_\A$. 
It is a natural question whether these generate all relations and whether $H^n_\A$ is a Koszul graded $R_{\alpha}$-algebra. 

We will use the following notation. If $b\in B^n$ is obtained from $a\in B^n$ by a flip, then we write $a\xrightarrow{i} b$ to denote the standard generator of $\brak{\b{b}a}\{n\}$ that is degree $1$ whose unique non-contractible circle is labeled $i\in \{1,2\}$ (see Definition \ref{def:flip}). 
Similarly, if $b$ is obtained from $a$ by a swap, then we write $a\to b$ to mean the unique degree $1$ standard generator of $\brak{\b{b}{a}}\{n\}$ (see Definition \ref{def:swap}). 
We consider length $2$ paths $a\to b \to c$ (with each arrow possibly decorated by an element of $\{1,2\}$) where $b$ (resp. $c$) is obtained from $a$ (resp. $b$) by a flip or a swap, and such a path corresponds to the product $(b\to c) \cdot (a\to b)$.

\begin{itemize}

\item There are far-commutativity relations of two types. First, suppose $r_1$ and $r_2$ are two disjoint surgery arcs in $a$ with endpoints on four distinct components, and set $b_1 = s_{r_1}(a)$, $b_{1,2}=s_{r_2}(b_1)$,  $b_2= s_{r_2}(a)$, and $b_{2,1} = s_{r_1}(b_2)$. 
Then $a\to b_1 \to b_{1,2} = a \to b_2 \to b_{2,1}$. Second, suppose $r$ is a surgery arc in $a$ and $\gamma\subset a$ is a component disjoint from  $r$ and that can be joined to the puncture via an arc disjoint from $r$ (in particular, $\gamma$ is outermost), and set $b= s_{r}(a)$, $c= f_{\gamma}(b)$, $b' = f_\gamma(a)$, $c' = s_{r}(b') = c$. Then $a\to b \xrightarrow{i} c = a\xrightarrow{i} b'\to c'$ for $i=1,2$. 

\item Consider three components $\gamma_1, \gamma_2, \gamma_3$ of $a\in B^n$ which are all pairwise adjacent. Let $\{i,j,k\} = \{1,2,3\}$. 
Performing surgery along an arc joining $\gamma_i$ to $\gamma_j$ results in two new arcs $\gamma_i'$ and $\gamma_j'$, exactly of which, say $\gamma_j'$, is adjacent to $\gamma_k$. 
Performing surgery along an arc joining $\gamma_j'$ to $\gamma_k$ results in a diagram $b\in B^n$; see below
\begin{equation*}
    \begin{aligned}
        \includestandalone{three_arcs}
    \end{aligned}
\end{equation*}
We can instead first perform surgery along an arc joining $\gamma_i$ and $\gamma_k$, or first perform surgery along an arc joining $\gamma_j$ and $\gamma_k$. All three paths are equal. 

\item Let $\gamma\subset a$ be outermost, set $b=f_\gamma(a)$, and  set $c = f_{\gamma'}(b)$ where $\gamma'\subset b$ is the outermost component corresponding to moving $\gamma$. 
Then $a\xrightarrow{i} b \xrightarrow{j} c = 0$ if $i\neq j$. Note, this is a special case of \eqref{eq:product of essential degree 1 generators}.

\item The four paths below are equal. These modifications are local, applied to adjacent outermost components. 
\begin{center}
    \begin{tikzpicture}[scale=.6]
\node at (0,0) {$\mathsmaller{\times}$};

\draw (-1,1) .. controls (-.25,0) .. (-1,-1);
\draw (1,1) .. controls (.25,0) .. (1,-1);

\node[below] at (0,-1.25) {$a$}; 
\draw[->] (1.5,0) -- (2.5,0);

%%%%%%%%%%%%%%%%%%%%%%%%%%%%%%

\begin{scope}[shift={(4,0)}]

\node at (0,0) {$\mathsmaller{\times}$};

\draw (-1,1) .. controls (0,.5) .. (1,1);
\draw (-1,-1) .. controls (0,.75) .. (1,-1);

\draw[->] (1.5,0) -- (2.5,0)  node[above,midway] {$i$};
\end{scope}

%%%%%%%%%%%%%%%%%%%%%%%%%%%%%%

\begin{scope}[shift={(8,0)}]

\node at (0,0) {$\mathsmaller{\times}$};

\draw (-1,1) .. controls (0,.25) .. (1,1);
\draw (-1,-1) .. controls (0,-.25) .. (1,-1);

\node[below] at (0,-1.25) {$b$}; 
\end{scope}

\end{tikzpicture}
%%%%%%%%%%%%%%%%%%%
\hskip4em
  \begin{tikzpicture}[scale=.6]
\node at (0,0) {$\mathsmaller{\times}$};

\draw (-1,1) .. controls (-.25,0) .. (-1,-1);
\draw (1,1) .. controls (.25,0) .. (1,-1);

\node[below] at (0,-1.25) {$a$}; 
\draw[->] (1.5,0) -- (2.5,0);

%%%%%%%%%%%%%%%%%%%%%%%%%%%%%%

\begin{scope}[shift={(4,0)}]

\node at (0,0) {$\mathsmaller{\times}$};

\draw (-1,1) .. controls (0,-.75) .. (1,1);
\draw (-1,-1) .. controls (0,-.5) .. (1,-1);

\draw[->] (1.5,0) -- (2.5,0)  node[above,midway] {$i$};
\end{scope}

%%%%%%%%%%%%%%%%%%%%%%%%%%%%%%

\begin{scope}[shift={(8,0)}]

\node at (0,0) {$\mathsmaller{\times}$};

\draw (-1,1) .. controls (0,.25) .. (1,1);
\draw (-1,-1) .. controls (0,-.25) .. (1,-1);

\node[below] at (0,-1.25) {$b$}; 
\end{scope}

\end{tikzpicture}

%%%%%%%%%%%%%%%% 
\vskip1em
%%%%%%%%%%%%%%%%

    \begin{tikzpicture}[scale=.6]
\node at (0,0) {$\mathsmaller{\times}$};

\draw (-1,1) .. controls (-.25,0) .. (-1,-1);
\draw (1,1) .. controls (.25,0) .. (1,-1);

\node[below] at (0,-1.25) {$a$}; 
\draw[->] (1.5,0) -- (2.5,0)  node[above,midway] {$i$};

%%%%%%%%%%%%%%%%%%%%%%%%%%%%%%

\begin{scope}[shift={(4,0)}]

\node at (0,0) {$\mathsmaller{\times}$};

\draw (-1,1) .. controls (.75,0) .. (-1,-1);
\draw (1,1) .. controls (.5,0) .. (1,-1);

\draw[->] (1.5,0) -- (2.5,0); 
\end{scope}

%%%%%%%%%%%%%%%%%%%%%%%%%%%%%%

\begin{scope}[shift={(8,0)}]

\node at (0,0) {$\mathsmaller{\times}$};

\draw (-1,1) .. controls (0,.25) .. (1,1);
\draw (-1,-1) .. controls (0,-.25) .. (1,-1);

\node[below] at (0,-1.25) {$b$}; 
\end{scope}

\end{tikzpicture}
%%%%%%%%%%%%%%%%%%%
\hskip4em
  \begin{tikzpicture}[scale=.6]
\node at (0,0) {$\mathsmaller{\times}$};

\draw (-1,1) .. controls (-.25,0) .. (-1,-1);
\draw (1,1) .. controls (.25,0) .. (1,-1);

\node[below] at (0,-1.25) {$a$}; 
\draw[->] (1.5,0) -- (2.5,0) node[above,midway] {$i$};

%%%%%%%%%%%%%%%%%%%%%%%%%%%%%%

\begin{scope}[shift={(4,0)}]

\node at (0,0) {$\mathsmaller{\times}$};

\draw (-1,1) .. controls (-.5,0) .. (-1,-1);
\draw (1,1) .. controls (-.75,0) .. (1,-1);

\draw[->] (1.5,0) -- (2.5,0); 
\end{scope}

%%%%%%%%%%%%%%%%%%%%%%%%%%%%%%

\begin{scope}[shift={(8,0)}]

\node at (0,0) {$\mathsmaller{\times}$};

\draw (-1,1) .. controls (0,.25) .. (1,1);
\draw (-1,-1) .. controls (0,-.25) .. (1,-1);

\node[below] at (0,-1.25) {$b$}; 
\end{scope}

\end{tikzpicture}
\end{center}

\item The sum of the first two below equals the third (the third one can have the surgery arc be on either side of $\times$). This can be seen by using the relations \cite[Equations (18) and (19)]{AK}. 
\begin{center}
    \begin{tikzpicture}[scale=.6]
\node at (0,0) {$\mathsmaller{\times}$};

\draw (-1,1) .. controls (-.25,0) .. (-1,-1);
\draw (1,1) .. controls (.25,0) .. (1,-1);

\node[below] at (0,-1.25) {$a$}; 

\draw[->] (1.5,0) -- (2.5,0) node[above,midway] {$1$};
%%%%%%%%%%%%%%%%%%%%%%%%%%%%%%

\begin{scope}[shift={(4,0)}]

\node at (0,0) {$\mathsmaller{\times}$};

\draw (-1,1) .. controls (.75,0) .. (-1,-1);
\draw (1,1) .. controls (.5,0) .. (1,-1);

\draw[->] (1.5,0) -- (2.5,0) node[above,midway] {$1$};
\end{scope}

%%%%%%%%%%%%%%%%%%%%%%%%%%%%%%

\begin{scope}[shift={(8,0)}]

\node at (0,0) {$\mathsmaller{\times}$};

\draw (-1,1) .. controls (-.25,0) .. (-1,-1);
\draw (1,1) .. controls (.25,0) .. (1,-1);
\node[below] at (0,-1.25) {$a$}; 
\end{scope}

\end{tikzpicture} 
%%%%%%%%%
\hskip4em
\begin{tikzpicture}[scale=.6]
\node at (0,0) {$\mathsmaller{\times}$};

\draw (-1,1) .. controls (-.25,0) .. (-1,-1);
\draw (1,1) .. controls (.25,0) .. (1,-1);

\node[below] at (0,-1.25) {$a$}; 

\draw[->] (1.5,0) -- (2.5,0) node[above,midway] {$2$};

%%%%%%%%%%%%%%%%%%%%%%%%%%%%%%

\begin{scope}[shift={(4,0)}]

\node at (0,0) {$\mathsmaller{\times}$};

\draw (-1,1) .. controls (-.5,0) .. (-1,-1);
\draw (1,1) .. controls (-.75,0) .. (1,-1);

\draw[->] (1.5,0) -- (2.5,0) node[above,midway] {$2$};
\end{scope}

%%%%%%%%%%%%%%%%%%%%%%%%%%%%%%

\begin{scope}[shift={(8,0)}]

\node at (0,0) {$\mathsmaller{\times}$};

\draw (-1,1) .. controls (-.25,0) .. (-1,-1);
\draw (1,1) .. controls (.25,0) .. (1,-1);
\node[below] at (0,-1.25) {$a$}; 
\end{scope}
\end{tikzpicture}  
%%%%%%%%%%%%%%%%%%%%%%%%%%%%%%%
\vskip1em
%%%%%%%%%%%%%%%%%
\begin{tikzpicture}[scale=.6]
\node at (0,0) {$\mathsmaller{\times}$};

\draw (-1,1) .. controls (-.25,0) .. (-1,-1);
\draw (1,1) .. controls (.25,0) .. (1,-1);

\node[below] at (0,-1.25) {$a$}; 

\draw[->] (1.5,0) -- (2.5,0);

%%%%%%%%%%%%%%%%%%%%%%%%%%%%%%

\begin{scope}[shift={(4,0)}]

\node at (0,0) {$\mathsmaller{\times}$};

\draw (-1,1) .. controls (0,.5) .. (1,1);
\draw (-1,-1) .. controls (0,.75) .. (1,-1);

\draw[->] (1.5,0) -- (2.5,0);
\end{scope}

%%%%%%%%%%%%%%%%%%%%%%%%%%%%%%

\begin{scope}[shift={(8,0)}]

\node at (0,0) {$\mathsmaller{\times}$};

\draw (-1,1) .. controls (-.25,0) .. (-1,-1);
\draw (1,1) .. controls (.25,0) .. (1,-1);

\node[below] at (0,-1.25) {$a$}; 
\end{scope}

\end{tikzpicture}
\end{center}

\item The sum of the first two below equals the third. This follows from the relation \cite[Equation (17)]{AK}. 

\begin{center}
 \begin{tikzpicture}[scale=.6]
\node at (0,0) {$\mathsmaller{\times}$};

\draw (-1,1) .. controls (-.5,0) .. (-1,-1);
\draw (1,1) .. controls (-.75,0) .. (1,-1);

\draw[->] (1.5,0) -- (2.5,0) node[above,midway] {$1$};

%%%%%%%%%%%%%%%%%%%%%%%%%%%%%%

\begin{scope}[shift={(4,0)}]

\node at (0,0) {$\mathsmaller{\times}$};

\draw (-1,1) .. controls (-.25,0) .. (-1,-1);
\draw (1,1) .. controls (.25,0) .. (1,-1);

\draw[->] (1.5,0) -- (2.5,0) node[above,midway] {$1$};
\end{scope}

%%%%%%%%%%%%%%%%%%%%%%%%%%%%%%

\begin{scope}[shift={(8,0)}]

\node at (0,0) {$\mathsmaller{\times}$};

\draw (-1,1) .. controls (.75,0) .. (-1,-1);
\draw (1,1) .. controls (.5,0) .. (1,-1);

\end{scope}
\end{tikzpicture}
%%%%%%%%%%%%%%%%%%%%%%%%%%%%%%%%%%%
\hskip4em
 \begin{tikzpicture}[scale=.6]
\node at (0,0) {$\mathsmaller{\times}$};

\draw (-1,1) .. controls (-.5,0) .. (-1,-1);
\draw (1,1) .. controls (-.75,0) .. (1,-1);

\draw[->] (1.5,0) -- (2.5,0) node[above,midway] {$2$};

%%%%%%%%%%%%%%%%%%%%%%%%%%%%%%

\begin{scope}[shift={(4,0)}]

\node at (0,0) {$\mathsmaller{\times}$};

\draw (-1,1) .. controls (-.25,0) .. (-1,-1);
\draw (1,1) .. controls (.25,0) .. (1,-1);

\draw[->] (1.5,0) -- (2.5,0) node[above,midway] {$2$};
\end{scope}

%%%%%%%%%%%%%%%%%%%%%%%%%%%%%%

\begin{scope}[shift={(8,0)}]

\node at (0,0) {$\mathsmaller{\times}$};

\draw (-1,1) .. controls (.75,0) .. (-1,-1);
\draw (1,1) .. controls (.5,0) .. (1,-1);
 
\end{scope}
\end{tikzpicture} \vskip1em
%%%%%%%%%%%%%%%%%%%%%%%%%%%%%%%%%%%

 \begin{tikzpicture}[scale=.6]
\node at (0,0) {$\mathsmaller{\times}$};

\draw (-1,1) .. controls (-.5,0) .. (-1,-1);
\draw (1,1) .. controls (-.75,0) .. (1,-1);

\draw[->] (1.5,0) -- (2.5,0);

%%%%%%%%%%%%%%%%%%%%%%%%%%%%%%

\begin{scope}[shift={(4,0)}]

\node at (0,0) {$\mathsmaller{\times}$};

\draw (-1,1) .. controls (0,.25) .. (1,1);
\draw (-1,-1) .. controls (0,-.25) .. (1,-1);

\draw[->] (1.5,0) -- (2.5,0);
\end{scope}

%%%%%%%%%%%%%%%%%%%%%%%%%%%%%%

\begin{scope}[shift={(8,0)}]

\node at (0,0) {$\mathsmaller{\times}$};

\draw (-1,1) .. controls (.75,0) .. (-1,-1);
\draw (1,1) .. controls (.5,0) .. (1,-1);

\end{scope}
\end{tikzpicture} 

\end{center}
\end{itemize}

\subsection{Annular \texorpdfstring{$SL(2)$}{SL(2)} bimodules} 

For $a\in \Bmatchings^n$, consider the graded left $H^n_\A$-module
\[
    P_a := \bigoplus_{b\in \Bmatchings^n} \brak{\b{b}a}\{n\}.
\]
Viewing $H^n_\A$ a left module over itself, we have the direct sum decomposition $
H^n_\A = \bigoplus_{a \in \Bmatchings^n} P_a$. 
Consequently, each $P_a$ is a projective left $H^n_\A$-module. 
Moreover, $P_a$ is indecomposable since $\End(P_a)_0$, the degree-zero part of its endomorphism ring, is isomorphic to $\Z$ (generated by right multiplication by $1_a$). 
Analogously, we have indecomposable projective graded right $H^n_\A$-modules 
\[
{}_a P = \bigoplus_{b\in \Bmatchings^n} \brak{\b{a} b}\{n\}.
\]

Let $T$ be an  annular flat $(n,m)$-tangle. We define the $(H^n_\A, H^m_\A)$-bimodule $\F(T)$ as follows. As a graded $R_\alpha$-module, 
\[
\F(T) = \bigoplus_{a\in B^m, b\in B^n} \brak{\b{b} T a} \{m\}.
\]
The left $H^n_\A$ action comes from a map 
\[
\brak{\b{c} b} \o_{R_\alpha} \brak{\b{b} T a} \to \brak{\b{c} T a}
\]
induced by the minimial cobordism from $\b{c} b
\b{b} T a$ to $\b{c} T a$. 
The right action is defined similarly via a map $\brak{\b{b} T a} \o_{R_\alpha} \brak{\b{a} d} \to \brak{\b{b}T d}$. 
The grading shifts ensure that $\F(T)$ is a graded $(H^n_\A, H^m_\A$)-bimodule.  
Note that if $T=\id_n$ is the identity tangle, then $\F(\id_n) \cong H^n_\A$ as graded $(H^n_\A, H^n_\A)$-bimodules. 

For $a\in \Bmatchings^m$, consider the $(n,0)$-tangle $Ta$. 
Removing all circles from $Ta$ results in an element $c\in \Bmatchings^n$. 
As a left $H^n_\A$-module, $\F(Ta)$ is isomorphic to $2^k$ copies of $P_c$, with grading shifts, where $k$ is the number of circles in $Ta$. Since 
\[
    \F(T) = \bigoplus_{a\in \Bmatchings^m} \F(Ta)\{m\},
\]
it follows that $\F(T)$ is a projective left $H^n_\A$-module. 
Similarly, for $b\in \Bmatchings^n$,  $\F(\b{b}T)$ is a direct sum of right $H^m_\A$-modules ${}_d P$, with grading shifts, where $d\in \Bmatchings^m$ is obtained by removing all circles from $\b{b}T$. 
We have 
\[
\F(T) = \bigoplus_{b\in \Bmatchings^n} \F(\b{b}T),
\]
so $\F(T)$ is also a projective right $H^m_\A$-module. 

It is evident that if two annular flat $(n,m)$-tangles $T$ and $T'$ are isotopic, then $\F(T) \cong \F(T')$ as graded $(H^n_\A,H^m_\A)$-bimodules. 

\begin{proposition}
\label{prop:composition of tangles}
    Given an annular flat $(n,m)$-tangle $T_1$ and an annular flat $(k,n)$-tangle $T_2$, there is a canonical isomorphism of graded $(H^k_\A, H^m_\A)$-bimodules 
    \[
\F(T_2) \o_{H^n_\A} \F(T_1) \cong \F(T_2 T_1).
    \]
\end{proposition}

\begin{proof}
The argument is similar to the proof of \cite[Theorem 1]{KhFunctorValued}. We will omit grading shifts throughout the proof. For $a\in \Bmatchings^m, b\in \Bmatchings^n$, and $c\in \Bmatchings^k$, consider the $R_\alpha$-linear map
\[
    \brak{\b{c}T_2b} \o_{R_\alpha} \brak{\b{b}T_1a} \to \brak{\b{c} T_2 T_1 a}
\]
induced by the minimal cobordism $b \b{b} \to 1_{2n}$. These assemble into  $\F(T_2) \o_{R_\alpha} \F(T_1) \to \F(T_2 T_1)$, which evidently factors through $\F(T_2) \o_{H^n_\A} \F(T_1)$  due to far-commutativity of the cobordisms involved. Thus we obtain
\[
    \psi: \F(T_2) \o_{H^n_\A} \F(T_1) \to \F(T_2 T_1),
\]
which is a map of $(H^k_\A, H^m_\A)$-bimodules, again due to far-commutativity. It is straightforward to verify that $\psi$ is degree-preserving after introducing the relevant grading shifts. To show that $\psi$ is an isomorphism, it suffices to show that its restriction 
\[
    \F(\b{c}T_2) \o_{H^n_\A} \F(T_1 a) \to \brak{\b{c}T_2 T_1 a}
\]
is an isomorphism, for each $a\in \Bmatchings^m,$ $c\in \Bmatchings^k$.
There are unique elements $d,e\in \Bmatchings^n$ such that $d$ is isotopic to $T_1a$ with all circles removed and $\b{e}$ is isotopic to $\b{c}T_2$ with all circles removed. 
We can then reduce to showing that $\F(\b{e})\o_{H^n_\A} \F(d) \to \brak{\b{e}d}$ is an isomorphism, which follows from multiplying the natural isomorphism $H^n_\A \o_{H^n_\A} H^n_\A \cong H^n_\A$ on the left and right, respectively, by the idempotents $1_e$ and $1_d$. 
\end{proof}

\begin{definition}
\label{def:ATan}
    Define the  $2$-category of annular flat tangles $\ATan$ as follows. Objects of $\ATan$ are numbers $n\in \Z_{\geq 0}$. Recall that we have fixed $2n$ points on $\SS^1$ for each $n\geq 0$; denote this set by $\u{2n}$. The set of $1$-morphisms from $m$ to $n$ is $\Bnoisotopy^n_m$, the space of annular flat $(n,m)$-tangles (see Definition \ref{def:various sets of annular tangles}). 
    
    For $T_0, T_1 \in \Bnoisotopy^n_m$, a $2$-morphism from $T_0$ to $T_1$ is a smoothly and properly embedded orientable surface $S \subset \A\times [0,1]$ whose boundary decomposes into the following four pieces:
\begin{itemize}
    \item $\partial S \cap \left( \A \times \{i\}  \right) = T_i$ for $i=0,1$,
    \item $\partial S \cap \left( \SS^1 \times \{0\} \times [0,1]\right) = \u{2m} \times \{0\} \times [0,1]$,
    \item $\partial S \cap \left( \SS^1 \times \{1\} \times [0,1]\right) = \u{2n} \times \{1\} \times [0,1]$.
\end{itemize}
Surface $S$ has $4(n+m)$ corner points. It may carry dots and will be referred to as a \emph{tangle cobordism} from $T_0$ to $T_1$. We consider $2$-morphisms up to ambient isotopy of $\A\times [0,1]$ which fixes $\partial(\A\times[0,1])$ pointwise.     
\end{definition}

Consider also the $2$-category $\gBimod$ whose objects are $\Z$-graded $R_\alpha$-algebras, $1$-morphisms are graded bimodules, and $2$-morphisms are homogeneous homomorphisms of graded bimodules.

\begin{remark}
    It may be more appropriate to call $\ATan$ and $\gBimod$ \emph{weak} $2$-categories, since composition of 1-morphisms in them is associative and unital only up to canonical 2-isomorphisms.
\end{remark}

Given $T_0, T_1\in \Bnoisotopy^n_m$ and a tangle cobordism $S: T_0 \to T_1$, define the map of bimodules 
\begin{equation}
\label{eq:map of bimodules}
    \F(S) : \F(T_0) \to \F(T_1)
\end{equation}
as follows. 
For $a\in \Bmatchings^m, b\in \Bmatchings^n$, there is an annular cobordism ${}_b S_a : \b{b}T_0 a \to \b{b}T_1 a$ which is obtained by gluing together $S$ and the identity cobordisms $\b{b}\times [0,1]$ and $a\times [0,1]$  along their common  boundary intervals. 
Applying the functor \eqref{eq_funct} yields an $R_\alpha$-linear map 
\[
    \brak{ {}_b S_a } : \brak{\b{b}T_0 a} \to \brak{\b{b}T_1 a},
\]
and \eqref{eq:map of bimodules} is the direct sum of these over all $a$ and $b$. Moreover,
\begin{equation}\label{eq_qdeg_S}
    \qdeg(\F(S)) \ = \ -\chi(S) +2d(S) + n+m, \ \ \adeg(\F(S))\ = \ 0,
\end{equation}
where $d(S)$ is the number of dots on $S$.

Our results and observations imply the following. 

\begin{proposition}
\label{prop:2-functor sl(2)}
    The above construction assembles into a $2$-functor 
\begin{equation}\label{eq_2_func}
    \F : \ATan \to \gBimod
\end{equation}
sending an object $n\geq 0$ to the algebra $H^n_\A$, a $1$-morphism $T\in \Bnoisotopy^n_m$ to the bimodule $\F(T)$, and a $2$-morphism $S: T_0 \to T_1$ to the map $\F(S) : \F(T_0) \to \F(T_1)$.
\end{proposition}

\begin{corollary}
    For $T\in \Bnoisotopy^n_m$ the functors of tensor product with bimodules $\mcF(T)$ and $\mcF(\overline{T})$ are biadjoint, up to overall grading shifts.  
\end{corollary}

\begin{proof}
    This follows from having canonical cobordisms between $T\overline{T}$ and the identity tangle $\id_{n}$ and likewise $\overline{T}T$ and $\id_{m}$ that satisfy isotopy relations matching those in the characterization of biadjoint functors~\cite{KhFunctorValued}. 
\end{proof}

The rotation automorphism of $H^n_\A$ discussed earlier is realized by the bimodule $\mcF(T)$ for the $(n,n)$-tangle $T$ which traces out the rotation of $2n$ points on the circle by $\pi/n$ (for $n\geq 1$). 

Let $\Tw$ denote the positive Dehn twist of $\A$. 
This is an orientation-preserving diffeomorphism of $\A$ which fixes $\partial \A$ pointwise. We view $\Tw$ as acting on the space of annular flat tangles. 

\begin{proposition}
\label{prop:Dehn twist}
    For any $T\in \Bnoisotopy^n_m$, there is an isomorphism of $(H^n_\A, H^m_\A)$-bimodules $\F(\Tw (T)) \cong \F(T)$. 
\end{proposition}

\begin{proof}
    By Proposition \ref{prop:composition of tangles}, we may also reduce to the case where $T= \id_n$, since up to isotopy $\Tw(T)$ can be written as $\Tw(T) = \Tw(\id_n) T$. 

    For any $a,b\in \Bmatchings^n$, there is a bijection between the components of $\b{b}\id_n a$ and $\b{b} \Tw(\id_n) a$ by untwisting. This gives an isomorphism of graded $R_\alpha$-modules $\brak{\b{b}a} \cong \brak{\b{b} \Tw(\id_n) a}$, and assembling these together gives an isomorphism of $R_\alpha$-modules $\F(\id_n)  \cong \F(\Tw(\id_n))$ which is evidently a map of $(H^n_\A, H^n_\A)$-bimodules. 
\end{proof}

Note, the above also holds for the negative Dehn twist $\Tw^{-1}$.

\begin{definition}
    \label{def:annular tangle}
    An \emph{annular $(n,m)$-tangle} is a smoothly and properly embedded compact $1$-manifold $T\subset \A\times [0,1]$ containing exactly $n+m$ interval components, such that $2n$ points of $\partial T$ are the points $\u{2n} \subset \SS^1 \cong \SS^1 \times \{1/2\} \times \{1\}$ and the remaining $2m$ points of $\partial T$ are the points $\u{2m} \subset \SS^1 \cong \SS^1 \times \{1/2\} \times \{0\}$. Taking a generic projection of $T$ onto $\A\times \{1/2\} \cong \A$ yields an annular $(n,m)$-tangle \emph{diagram} $D$.
\end{definition}

We may form the cube of resolutions of an annular tangle diagram $D$ in the usual way, with all smoothings given by annular flat tangles and all edge cobordisms given by tangle cobordisms. Applying the functor \eqref{eq_2_func} to this cube yields a complex $C(D)$ of graded $(H^n_\A, H^m_\A)$-bimodules. 

Two diagrams representing ambiently isotopic annular tangles are related by a sequence of Reidemeister moves supported in disks in the interior of $\A$ \cite{APStangles}. Standard arguments \cite{KhovanovJones,Bar-Natan} imply the following.

\begin{theorem}\label{thm_isotopy_invariance}
    The chain homotopy type of the complex of graded bimodules $C(D)$ depends only on the ambient isotopy class of $T$. 
\end{theorem}
In light of this, we can denote this chain homotopy type by $C(T)$.

\section{Annular \texorpdfstring{$SL(3)$}{SL(3)} web algebras and bimodules}
\label{sec:annular sl3 webs}

\subsection{Annular \texorpdfstring{$SL(3)$}{SL(3)} webs with endpoints on only one boundary component}

\begin{definition}
\label{def:sl3 web}
   An \emph{annular $SL(3)$ web} is a finite oriented graph $w$, which may contain closed oriented loops with no vertices, that is embedded in $\A$ subject to the following conditions. 
    \begin{itemize} 
        \item Every degree-one vertex of $w$ is in $\partial \A$. We call such a point a \emph{boundary point} of $w$. The set of boundary points is denoted $\partial w$. All other points of $w$ lie in the interior of $\A$.
        \item Every other vertex of $w$, called an internal vertex, has degree three and is either a source or a sink, as shown in Figure \ref{fig:trivalent vertex}.
    \end{itemize}
    An \emph{inner face} of $w$ is a connected component of $\A\setminus w$ which is disjoint from $\partial \A$.   
\end{definition}

We will often write \emph{web} in place of annular $SL(3)$ web. The orientation of a web $w$ gives a sign assignment $S:\partial w\to \{+,-\}$ according to Figure \ref{fig:edge orientation at boundary}, where $\partial_i \A := \SS^1\times \{i\}$, for $i=0,1$. 

\begin{figure}
\centering 
\subcaptionbox{The orientations of the edges meeting at an internal vertex.\label{fig:trivalent vertex}}[1\linewidth]
{\includestandalone{images/trivalent_vertex}}\\  \vskip1em
\subcaptionbox{The possible orientations of $w$ near $\partial \A$ and the corresponding sign assignments. \label{fig:edge orientation at boundary}}[1\linewidth]
{\includestandalone{images/orientation_at_boundary}
}
\caption{}\label{fig:orientations of a web}
\end{figure}

\begin{definition}
\label{def:non-elliptic web}
    An inner face of a web $w\subset \A$ is \emph{elliptic} if its boundary is 
    \begin{itemize}
        \item a circle, possibly enclosing the puncture, or
        \item a bigon or a square, not enclosing the puncture. 
    \end{itemize} 
    Say $w$ is \emph{non-elliptic} if it does not contain any elliptic faces. See Figure \ref{fig:reducible regions}.
\end{definition}

\begin{remark}
    A non-elliptic web cannot have any closed components by \cite[Lemma 3.21]{AK}. 
\end{remark}

\begin{figure}
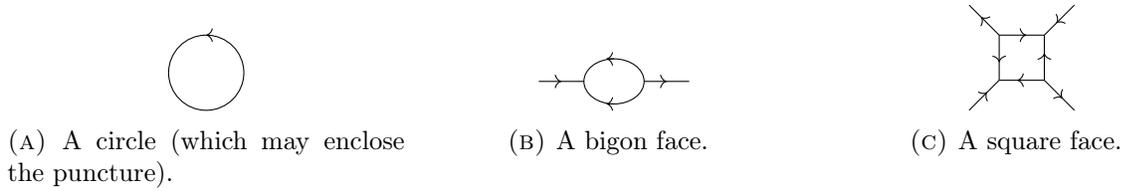

\centering 
\subcaptionbox{A circle (which may enclose the puncture). \label{fig:circle}}[.32\linewidth]
{\includestandalone{images/circle}
}
\subcaptionbox{A bigon face. \label{fig:bigon}}[.32\linewidth]
{\includestandalone{images/bigon}
}
\subcaptionbox{A square face. \label{fig:square}}[.32\linewidth]
{\includestandalone{images/square}}
\caption{ Elliptic faces in a web.}
\label{fig:reducible regions}
\end{figure}

We will first focus on annular $SL(3)$ webs whose boundary lies entirely in $\SS^1\times \{1\}$. In this case, we view $w$ as embedded in $\SS^1\times (0,1]$. It is useful to  model $\SS^1\times (0,1]$ as a punctured disk
\[
    \A_0 := \{ p \in \R^2 \st 0< \vert p \vert \leq 1\} 
\]
with $\times$ denoting the puncture at the origin. We also pick a basepoint $*\in \partial \A_0$, as shown on the left of Figure \ref{fig:annulus-models}.

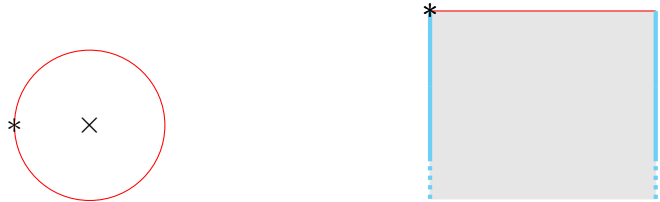
\begin{figure}
    \centering
    %%%%%%%
    \begin{tikzpicture}
        \draw[\boundarycolor] (0,0) circle (1cm);
        \node at (0,0) {$\times$};
        \node at (-1,0) {$*$};
    \end{tikzpicture}
    \hspace{3cm}
    \begin{tikzpicture}[xscale=3]
        % shading
        \filldraw[black!10!white] (0,0) rectangle (1,-2.5);
        % lines
        \draw[\boundarycolor] (0,0) -- (1,0);
        \foreach \x in {0,1}{
            \draw[ultra thick, \sidecolor] (\x,-2) -- (\x,-1);
            \draw[ultra thick, \sidecolor] (\x,-1) -- (\x,0);
            \draw[dotted, \sidecolor, ultra thick] (\x, -2.5) -- (\x, -2);
        % basepoint
        \node at (0,0) {$*$};
        }
    \end{tikzpicture}
    %%%%%%%
    \caption{Left: The punctured disk model $\A_0$ for the annulus. Right: The infinite cylinder model $\A_0'$ for the annulus.}
    \label{fig:annulus-models}
\end{figure}

There is a homeomorphism from $\A_0$ to the  infinite cylinder
\[
    \A_0' :=  \{ (x,y) \in \R^2 \st 0 \leq x \leq 1,\  y \leq 0\} / \sim, 
    \qquad
    \text{where}
    \qquad
    (0,y) \sim (1,y)
    \quad \text{for all $y$},
\]
depicted on the right of Figure \ref{fig:annulus-models}. 

Notice that in this latter model, there is an implicit `basepoint' $*$ given by $(0,0) \sim (1,0)$ in the boundary $\partial \A_0'$.
We will use this model in our adaptation of the growth algorithm from \cite{KhKup}.
Under a homeomorphism $f: \A_0' \to \A_0$, $\lim_{y \to -\infty} f(x,y) = \times$, when the homeomorphism is extended to one-point compactifications, and moreover $f$ identifies the two basepoints.  For a web $w\subset \A_0$ (equivalently $w\subset \A_0'$), we always assume that $\partial w$ does not contain the basepoint. 

Let us begin by collecting some properties of webs. First, if $w$ is an annular web (potentially with elliptic faces), then every edge in $w$ has distinct regions on each of its two sides. To see this, consider the reflection $\b{w}$ of $w$ through $\SS^1\times \{1/2\}$. Gluing $\b{w}$ to $w$ along their common endpoints results in a closed $SL(3)$ web, which is in particular a bipartite trivalent graph and hence bridgeless.

A \emph{face} of $w$ is a connected component of $\A \setminus w$. We say a face $F$ is inner (resp. outer) if $F\cap \partial \A = \varnothing$ (resp. $F\cap \partial \A \neq \varnothing$).  Let $v_i$ and $v_b$ be the number of inner and boundary vertices of $w$, respectively, and let $e_i$  and $e_b$ be the number of inner and boundary edges of $w$, respectively. Inner and boundary vertices of $w$ are trivalent and univalent, respectively. Note that $v_b = e_b$. We have 
   \begin{equation}
       \label{eq:v and e identity}
       3v_i = 2 e_i + e_b.
   \end{equation}
   
Let $f_i$ be the number of inner faces and let $f_b$ be the number of outer faces.  If $w$ is connected, then every outer face of $w$ has exactly one of its sides on $\partial \A$, so in particular $f_b = v_b$.  An Euler characteristic computation for the $2$-disk $\DD^2$ gives  
   \[
       \chi(\DD^2) = 1 = v_i + v_b - (e_i + e_b + f_b) + f_i + f_b = v_i - e_i + f_i
    \]
which implies 
    \begin{equation}
    \label{eq:euler char result}
        e_i = v_i + f_i - 1.
    \end{equation}
Combining \eqref{eq:euler char result} and \eqref{eq:v and e identity} implies 
    \begin{equation}
    \label{eq:inner vertices}
        v_i = 2f_i + f_b -2.
    \end{equation}
     
   Let $F$ be the set of faces of $w$, and for $f\in F$ let $s(f)$ denote the number of sides of $f$ (including those that lie inside $\partial \A$). We have 
   \begin{equation}
   \label{eq:sum of sides}
        \sum_{f\in F} s(f) = 2(e_i + e_b) + f_b = 2 e_i + 3 f_b.
    \end{equation}
Note that every inner face has an even number of sides.

We will sometimes consider webs in the disk $\DD^2$. There are clear analogous notions of inner and outer vertices, edges, and faces. 
As in the annular case, by non-elliptic we mean that no inner face is a circle, a bigon, or a square. 
The above properties, in particular  equations \eqref{eq:inner vertices} and \eqref{eq:sum of sides} hold  for a connected web $w\subset \DD^2$. 

An outer face which does not contain the puncture and which has $2$, $3$, or $4$ sides will be called a $U$, a $Y$, or an  $H$, in that order. 
See \eqref{eq:U Y and H faces}.
 \begin{equation}
 \label{eq:U Y and H faces}
 \begin{aligned}
        \begin{tikzpicture}
            \draw[\boundarycolor] (0,0) -- (1.5,0);
            \draw (1.25,0) arc (0:-180:.5 and .5);
            %%%%
            \begin{scope}[shift={(4,0)}]
                   \draw[\boundarycolor] (0,0) -- (1.5,0);
            \draw (1.25,0) -- (.75,-.5);
            \draw (.25,0) -- (.75,-.5);
            \draw (.75,-.5) -- (.75,-.75);
            \end{scope}
            %%%%%%
            \begin{scope}[shift={(8,0)}]
                \draw[\boundarycolor] (0,0) -- (1.5,0);
            \draw (1.25,0) -- (1.25,-.75);
            \draw (.25,0) -- (.25,-.75);
            \draw (.25,-.5) -- (1.25,-.5);
            \end{scope}
        \end{tikzpicture}
       \end{aligned}     
    \end{equation}
The following two results will be useful later. 

\begin{lemma}
\label{lem:at least 3 small regions} 
    Let $w\subset \DD^2$ be a connected non-elliptic web with at least one inner vertex. Then $w$ has at least three outer faces such that each one is either a $Y$ or an $H$. 
\end{lemma}

\begin{proof}
    Note the only connected web with no inner vertices is an arc connecting two points on $\partial \DD^2$. Suppose that at most two outer faces of $w$ are $Y$'s or $H$'s. Then \eqref{eq:sum of sides} gives 
    \[
        2e_i + 3 f_b \geq 6 + 5(f_b - 2) + 6 f_i
    \]
    which can be rewritten as $ e_i  \geq 3 f_i +  f_b -2. $
   Substituting \eqref{eq:euler char result} gives 
   \[
        v_i \geq 2f_i + f_b -1 
   \]
  which contradicts \eqref{eq:inner vertices}.
\end{proof}

\begin{lemma}
\label{lem:U Y or H}
    Let $w$ be a non-elliptic annular web. Then at least one outer face of $w$ is a $U$, a $Y$, or an $H$.
\end{lemma}
\begin{proof}
  
Assume first that $w$ is connected. 
We may also assume that the puncture is contained in an inner face, since otherwise $w$ is contained in a disk and  the statement of the present lemma follows from \cite[Proposition 1]{KhKup}. 
Now, assume to the contrary that every outer face of $w$ has at least $5$ sides. 
Adding up the sides of the faces for the face with the puncture, inner (but not punctured) faces, and outer faces respectively, \eqref{eq:sum of sides} implies
   \[
       2 e_i + 3 f_b \geq  2 + 6(f_i-1) + 5 f_b 
   \] 
which can be rewritten as $ e_i  \geq 3 f_i +  f_b -2. $
Substituting \eqref{eq:euler char result} gives 
   \[
        v_i \geq 2f_i + f_b -1 
   \]
which contradicts \eqref{eq:inner vertices}. 
  
Finally, we consider the case where $w$ is not necessarily connected. Let $w'$ be a component of $w$. By the above, some outer face $F$ of $w'$ is a $U$, a $Y$, or an $H$. If $F$ is also a face of $w$ then we are done. Otherwise, let $w''$ be the union of all components of $w$ that are contained in the interior of $F$. We may view $w''$ as a non-elliptic web in the lower half-plane, and as a consequence of \cite[Proposition 1]{KhKup} $w''$ contains an outer $U$, $Y$, or $H$ face, which is then also a face of $w$. 
\end{proof}

\subsection{Sign strings and state strings}

Each of our basis webs will be in bijection with admissible (sign string, state string) pairs, which we define below. 

Given $n$  points in $\partial \A \setminus \{*\}$, our convention is to order them $\{1,2,\ldots, n\}$ by reading clockwise from the basepoint $*$. Similarly, given $n$ points in $\partial \A'\setminus \{*\} \cong (0,1)$, we order them according to the standard ordering of $(0,1)$ (from left to right). 

\begin{definition}
A \emph{sign string of length $n$} is a sequence $S: \{1,2,\ldots, n\} \to \{+,-\}$.
We typically denote $S$ by $S=(s_1,\ldots, s_n)$, where each $s_i\in \{+,-\}$. 
The number of $\pm$ entries in $S$ is denoted $n_\pm(S)$. If $n_+(S) \equiv n_-(S) \mod 3$ then we say $S$ is \emph{admissible}. 
\end{definition}

Once again, the orientation of a web $w$ in $\A_0'$  gives a sign string. It is immediate to verify that a sign string $S$ is the boundary of a web if and only if $S$ is admissible. 

\begin{definition}
    Let $B^S$ denote a set of representatives for ambient isotopy classes of annular webs with boundary $S$. 
\end{definition}

We now turn to state strings. Let $\Lambda$ denote the weight lattice of the Lie algebra $\sl_3$. It has a basis given by the two fundamental dominant weights $\mu^+$ and $\mu^-$. Set 
\[
    \Lambda^+ = \{\mu^+, \mu^- - \mu^+, -\mu^-\}, \ \ \Lambda^- = \{\mu^-, \mu^+ - \mu^-, -\mu^+\}, \ \ \Lambda^\pm = \Lambda^+ \cup \Lambda^-. 
\]
Let $V^+$ denote the fundamental representation of $\sl_3$ and let $V^- = (V^+)^*$ denote its dual. Then $\Lambda^+$ and $\Lambda^-$ are the weights of $V^+$ and $V^-$, respectively. This is summarized in Figure \ref{fig:weight lattice}. 

\begin{figure}
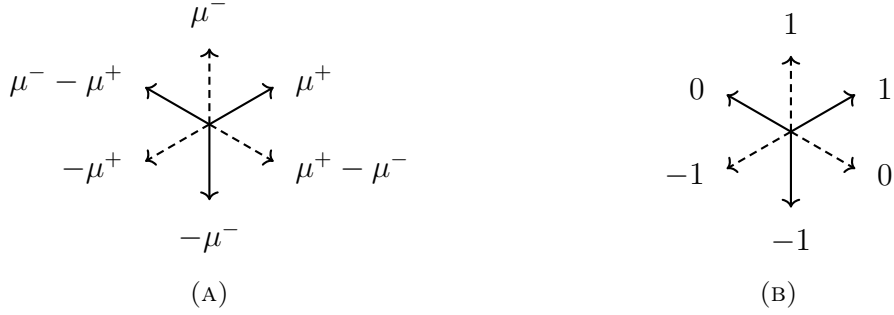

\centering 
\subcaptionbox{ \label{fig:weight lattice}}[.45\linewidth]
{\includestandalone{images/weight_lattice}
}
\subcaptionbox{ \label{fig:weight states}}[.45\linewidth]
{\includestandalone{images/weight_state}}
\caption{Left: the weights $\Lambda^+$ of the fundamental representation $V^+$ (drawn solid) and the weights $\Lambda^-$ of its dual $V^-$ (drawn dashed). Right: the assignment of a state to each element of $\Lambda^\pm$.}\label{fig:lattice and states}
\end{figure}

\begin{definition}
    Given a sign string $S=(s_1,\ldots, s_n)$, an \emph{$S$-path} is a sequence of lattice points $\pi_0 = 0, \pi_1,\ldots, \pi_n \in \Lambda$ such that $\pi_i - \pi_{i-1} \in \Lambda^{s_i}$ for all $1\leq i \leq n$; we say it is \emph{closed} if $\pi_n = 0$.
\end{definition}

In other words, a closed $S$-path consists of a sequence $v_1\ldots, v_n \in \Lambda^{\pm}$ such that $v_i \in \Lambda^{s_i}$ and $\sum_{i=1}^n v_i = 0$. In the notation of the previous paragraph, $v_i = \pi_i - \pi_{i-1}$.

A \emph{state string of length $n$} is a sequence $J: \{1,2,\ldots, n\} \to \{ -1, 0, 1\}$. We write $J$ as $(j_1,\ldots, j_n)$. 
We fix a bijection between the set of $S$-paths and the set $\{-1,0,1\}^n$ of state strings of length $n$ where $j_i$ is given by identifying 
$\pi_i - \pi_{i-1}$ 
with $-1$, $0$, or $1$ according  to the diagrams in Figure \ref{fig:lattice and states}. 

\begin{definition}
    Given a sign string $S$, a (sign, state) string pair $(S,J)$ is \emph{admissible} if the corresponding path in the weight lattice is closed.
\end{definition}

\begin{remark} 
We record some comments. 
\begin{itemize}
    \item If $(S,J)$ is admissible, then so is $S$. 
    \item  If $(S,J)$ is admissible, then $\sum_{i} j_i = 0$, but the converse is not true. Consider for instance $S = (+,+,+)$ and $J=(0,0,0)$. 
    \item The planar case \cite{KhKup} restricts to \emph{dominant paths}, which must lie in a chosen Weyl chamber. We do not impose this restriction. 
\item    The choice of basepoint $*$ fixes a linear ordering of the boundary  and is largely for convenience. We could instead consider $n$ points $P_n \subset \partial \A_0$ (or $P_n\subset \partial \A_0'$) and define sign and state strings as functions $S: P_n \to \{+,-\}$ and $J:P_n \to \{-1, 0,1\}$, respectively. Without choosing a basepoint, $(S,J)$ does not naturally correspond to a lattice path. However, whether or not $(S,J)$ is admissible is independent of the choice of basepoint.
\end{itemize}
\end{remark}

\subsection{Growth algorithm}

Let $(S,J)$ be an admissible  pair. 
The following `growth algorithm' $\growalg$ builds a web from this boundary condition. 
We follow \cite{KhKup}, where the convention is for webs to grow downward (see also \cite{sl4-webs} for a related construction in the case of planar $SL(4)$ webs). 
In our case, webs naturally grow downward in $\A_0'$. 
The only significant difference between \cite{KhKup} and our annular setting is that the growth rules apply to states that are cyclically ordered.

The web can be iteratively propagated downward using the following set of rules. 

\begin{itemize}
\item 
If the signs (i.e.\ orientations) differ at the top:
\renewcommand{\arraystretch}{2}
\begin{center}
    \begin{tabular}{c c c c c}
    \centering
        \growthwebH{1}{0}{0}{1} 
        &
        \growthwebH{0}{0}{-1}{1}
        &
        \growthwebH{0}{-1}{-1}{0} 
        &
        \phantom{tomato}
        &
        % U web needs baseline fix
        \begin{tikzpicture}[scale=.5, baseline=-10ex]
        \draw (0,0) node[label={90:$1$}] {} 
            arc (180:360:1cm) 
            node[label={90:$-1$}] {};
    \end{tikzpicture}
    \\
        $H_1$ & $H_0$ & $H_{-1}$ & & $U_0$
    \end{tabular}
\end{center}
We refer to the first three as \emph{$H$-type growth webs} and the fourth as the \emph{$U$-type growth web}.

\item 
If the signs (i.e.\ orientations) agree at the top:
\begin{center}
    \begin{tabular}{c c c}
        \growthwebY{1}{0}{1}
        &
        \growthwebY{0}{-1}{-1}
        &
        \growthwebY{1}{-1}{0}  \\
        $Y_1$ & $Y_{-1}$ & $Y_{0}$
    \end{tabular}
\end{center}
We refer to these three as \emph{$Y$-type growth webs}. 
\end{itemize}
The names for these growth webs are inspired by the sum of the state labels at the top (or bottom).

\begin{example}
    \label{ex:length 2}
Let us consider $S= (+,-)$. In this case, the admissible state strings are $(-1,1)$, $(0,0)$, and $(1,-1)$. The corresponding webs are shown in \eqref{eq:length two ex}, using both the model $\A_0'$ (top row) and $\A_0$ (bottom row). 

\begin{equation}
    \begin{aligned}
    \label{eq:length two ex}
        \includestandalone{length_two_ex}
    \end{aligned}
\end{equation}

\end{example}

\begin{figure}
    \centering
    \input{images/growth-rules-dictionary}
    \caption{The dictionary between growth rules and modifications of lattice paths.}
    \label{fig:growth dictionary}
\end{figure}

Before proceeding to discuss the growth algorithm in more detail, let us show that we have found all the elements of $B^{(+,-)}$. 

\begin{lemma}
    \label{lem:2 boundary points}
    When $S = (+, -)$, the set $B^S$ consists of the three webs in Example \ref{ex:length 2}.
\end{lemma}

\begin{proof}
    Let $w\in B^{(+,-)}$. Lemma \ref{lem:U Y or H} says that $w$ has a $U$ or an $H$ at its boundary (note that a $Y$ is excluded for orientation reasons). If $w$ has a boundary $U$ then we are done. Suppose then that $w$ has a boundary $H$. Removing this boundary $H$ yields a web $w'\in B^{(-,+)}$. If $w'$ has a boundary $U$ then since $w$ is non-elliptic it must be the web shown in the middle of \eqref{eq:length two ex}. Otherwise, $w'$ has a boundary $H$. Repeating this, we see that $w$ consists of some finite number of $H$'s stacked on top of each other, necessarily in an alternating manner as shown in \eqref{eq:stacked Hs}, followed by a $U$. However, there is no way to close up two or more $H$'s with a $U$ to form a non-elliptic web. 
    \begin{equation}
    \label{eq:stacked Hs}
    \begin{aligned}
        \begin{tikzpicture}
        \draw[very thick,\sidecolor] (-.5,0) -- (-0.5,-1.7);
        \draw[very thick,\sidecolor] (1.5,0) -- (1.5,-1.7);
        \draw[\boundarycolor] (-.5,0) -- (1.5,0);
        \draw (0,0) -- (0,-1);
        \draw (1,0) -- (1,-1);
        \draw (0,-.25) -- (1,-.25);
        \draw(0,-.5) -- (-.5,-.5);
        \draw (1,-.5) -- (1.5,-.5);
        \draw (0,-.75) -- (1,-.75);
        \node at (.5,-1.25)  {$\vdots$};
        \end{tikzpicture}
\end{aligned}
    \end{equation}
\end{proof}

The remainder of this subsection is dedicated to showing that the growth algorithm is well-defined; that is, it assigns an element of $B^S$ to every admissible pair $(S,J)$. 

\begin{lemma}
Let $(S,J)$ be admissible. 
If $(S',J')$ is obtained from $(S,J)$ by one of the moves from the growth algorithm, then $(S',J')$ is also admissible. 
\end{lemma}
\begin{proof}
    Each growth rule modifies the corresponding lattice path according to the dictionary in Figure \ref{fig:growth dictionary}. It follows that if $(S,J)$ corresponds to a closed path, so does $(S',J')$. 
\end{proof}

We will make a small modification to Lemma 3 of \cite{KhKup}, as our situation is in a sense more lax.

\begin{lemma}
\label{lem:growth alg terminates}
    If $(S,J)$ is an admissible (sign, state) string pair, then the growth algorithm terminates only when the strings have been reduced to length 0. 
\end{lemma}

\begin{proof}
There is a total ordering $\prec$ on state strings as follows (\cite[proof of Lemma 1]{KhKup}). Let $K, K'$ be two state strings. 
\begin{itemize}
    \item If $\len(K) < \len(K')$, then $K \prec K'$.
    \item If $\len(K) = \len(K')$, then use the lexicographic ordering where $-1 \prec 0 \prec 1$. 
\end{itemize}

We proceed by induction on the length $n$ of $S$. The base case $n=2$ was covered in Example \ref{ex:length 2}.  
For the inductive step, if $J=(j_1,\ldots, j_n)$ with some $j_i > j_{i+1}$, then we may apply one of the $Y$ or $U_0$ moves to either reduce the length, or we may apply an $H$ move to arrive at $(S', J')$ of the same length and with $(S', J')\prec (S, J)$. Thus we may assume that $J$ is non-decreasing, so it is of the form
\[
    J= (\underbrace{-1,\ldots, -1}_a, \underbrace{0,\ldots, 0}_{b}, \underbrace{1, \ldots, 1}_c)
\]
with $a,b,c\geq 0$ and $a+b+c=n$. Since $J$ is admissible, either $a,c>0$ or  $a=c=0$. In the case $a,c>0$, we may apply either an annular $U_0$ move or an annular $Y$ move between the first and last boundary points to decrease the length. In the case $a=c=0$, we must also have that $n_-(S) = n_+(S)$ for $J$ to correspond to a closed path. 
Therefore a subsequence $+ -$ occurs in $S$, and by cyclic reordering we may assume $S = (+,-, s_3, \ldots, s_n)$. 
We apply $H_0$ at the first two points to arrive at $S'' = (-, +, s_3, \ldots, s_n)$ and $J'' = (-1,1,j_3,\ldots, j_n)$. If $s_3 = +$ then we may apply a $Y_1$ move to decrease the length, or otherwise if $s_3=-$ then we apply an $H_1$ move:
\begin{center}
    \begin{tikzpicture}
        \node[above] at (0,0) {$+$};
        \node[above] at (1,0) {$-$};
        \node[above] at (2,0) {$+$};
        
        \node[left] at (0,0) {$0$};
        \node[right] at (1,0) {$0$};
        \node[right] at (2,0) {$0$};

        \node[left] at (0,-1) {$-1$};
        \node[left] at (1,-1) {$1$};
        \node[right] at (2,-1) {$0$};
       \draw (1,0) -- (1,-1);
       \draw (2,0) -- (2,-1);
       \draw(0,-.5) -- (1,-.5);
       \draw (0,0) -- (0,-1);

       \draw(1,-1) -- (1.5,-1.5);
       \draw (2,-1) -- (1.5,-1.5);
       \draw (1.5,-1.5) -- (1.5,-2);
       \draw (0,-1) -- (0,-2);

%%%%%%%%%%%%%%%%%%%%%%%%%%%%%
\begin{scope}[shift = {(5,0)}]
     \node[above] at (0,0) {$+$};
        \node[above] at (1,0) {$-$};
        \node[above] at (2,0) {$-$};
        
        \node[left] at (0,0) {$0$};
        \node[right] at (1,0) {$0$};
        \node[right] at (2,0) {$0$};

        \node[left] at (0,-1) {$-1$};
        \node[left] at (1,-1) {$1$};
        \node[right] at (2,-1) {$0$};

         \node[left] at (1,-2) {$0$};
        \node[right] at (2,-2) {$1$};
        
          \draw (0,0) -- (0,-2);
       \draw (1,0) -- (1,-2);
       \draw (2,0) -- (2,-2);
       \draw(0,-.5) -- (1,-.5);
       \draw(1,-1.5) -- (2,-1.5);
\end{scope}
\node at (3.5,-1) {or};
    \end{tikzpicture} .
\end{center}
Continuing this by examining $s_4$, we will either eventually reduce the length or arrive at the state string $(-1,0,\ldots, 0,1)$, which can be reduced by an annular $Y_0$ or $U_0$ move.
\end{proof}

If $(S,J)$ is not admissible then the growth algorithm may not terminate at an empty pair; for instance, consider $S= (+,-)$ and $J = (1,1)$, where no growth rule can be applied. Moreover, the growth algorithm may continue indefinitely: if $S= (+,-)$ and $J=(1,0)$, then $\growalg$ would produce an infinite sequence of $H$'s as in \eqref{eq:stacked Hs}. 
The following is our analogue of \cite[Lemma 1]{KhKup}, and we include the admissibility condition only to rule out this latter situation. 

\begin{lemma}
\label{lem:growth-algorithm-order-independence}
Let $(S,J)$ be an admissible (sign, state) string pair.
The web resulting from applying the growth algorithm does not depend on the choice of order in which we apply the growth rules.
\end{lemma} 

\begin{proof}
Suppose there were an admissible (sign, state) string pair $(S,J)$, and that there exist two sequences of moves (i.e.\ applications of the growth rules) that result in nonequivalent webs.

Without loss of generality, we may assume $(S,J)$ is a minimal counterexample. That is, by applying two inequivalent rules (`moves') $M_1$ and $M_2$ to $(S,J)$, we obtain new admissible (sign, state) string pairs $(S_1, J_1)$ and $(S_2, J_2)$, respectively, and both of these are \emph{not} counterexamples, i.e.\ the growth algorithm determines a unique web $w_i$ for the pair $(S_i, J_i)$, where $i = 1,2$. 
As $(S,J)$ is a counterexample, we are assuming here that $w_1$ and $w_2$ are nonequivalent webs. 

If $M_1$ and $M_2$ are disjoint from each other (i.e.\ they operate on a total of four distinct endpoints), then $w_1$ and $w_2$ are obviously isotopic, simply by exchanging the heights at which the moves $M_1$ and $M_2$ are applied. 
Therefore we need only consider situations where the the moves $M_1, M_2$ are not disjoint, and then compare $(S_1,J_1)$ and $(S_2, J_2)$. To show that $(S,J)$ cannot be a counterexample, we show that, when there is a choice to be made, any choice will lead to an isotopic web. 

First, let us categorize the possible moves (growth rules) by how they can be combined with the other moves. 
We will consider admissible pairs of length-3 (sign, state) substrings where two rules can be applied consecutively. 
Then, we can consider whether there are any other possible sequences of moves we can execute.

The chart on the left organizes the seven growth webs by their right-most label on the bottom row. 
This label will then become the top left label of a composable growth web. 
The chart on the right organizes the seven growth webs by the left-most label on the top row. 

    \begin{center}
    \renewcommand{\arraystretch}{1.5}

    \begin{tabular}{|r l|}
        \hline
        bottom right label = $1$: & $H_1, H_0, Y_1$ \\
        bottom right label = $0$: & $H_{-1}, Y_0$ \\
        bottom right label = $-1$: & $Y_{-1}$ \\ 
        \hline
    \end{tabular}
    \hspace{1cm}
    \begin{tabular}{|r l|}
        \hline
        top left label = $1$: & $H_1, U_0, Y_1, Y_0$ \\
        top left label = $0$: & $H_0, H_{-1}, Y_{-1}$ \\
        top left label = $-1$: & none \\ 
        \hline
    \end{tabular}
    \end{center}

    A priori, we need to check the 18 pairs 
    \[
        \{H_1, H_0, Y_1 \} \times  \{H_1, U_0, Y_1, Y_0\} 
        \quad
        \cup 
        \quad
        \{H_{-1}, Y_0\} \times \{ H_0, H_{-1}, Y_{-1}\}.
    \]

    \textit{When there is no choice to be made.}
    For some pairs, there is no choice to be made. 
    For example, for $H_1$ followed by $H_1$, the top length-3 (sign, state) substrings must be one of the following:
    \[
        \text{either} 
        \qquad
        \begin{tabular}{c c c}
            1 & 0 & 0 \\
            $+$ & $-$ & $-$ 
        \end{tabular}
        \qquad
        \text{or}
        \qquad
        \begin{tabular}{c c c}
            1 & 0 & 0 \\
            $-$ & $+$ & $+$ 
        \end{tabular}.
    \]

    In either case, one must apply $H_1$ on the left first. 
    The resulting propagated web looks like the web in Figure \ref{fig:no choice ex}.

    Table \eqref{table:no choice} lists the pairs of moves where there is no choice for which move to apply first. 
    The first column shows the pair of moves, written in the form (move on left two endpoints, move on remaining right two remaining endpoints); 
    the second column shows the starting (sign, state) substring, so that the reader may verify the lack of choice. 
    (Note that for the second column, one may also reverse all the signs in the sign substring.)

    \begin{equation}    
    \label{table:no choice}
    \renewcommand{\arraystretch}{1.5}
    \begin{tabular}{|c|c|}
        \hline
        moves & (sign, state) substrings \\
        \hline
        $H_1, H_1$ & $+-+, \quad (1,0,0)$ \\
        $H_0, H_1$ & $+--, \quad (0,0,0)$ \\
        $Y_1, H_1$ & $+++, \quad (1,0,0)$ \\
        $H_{-1}, H_0$ & $+--, \quad (0,-1,0)$ \\
        $H_{-1}, H_{-1}$ & $+--, \quad (0,-1,-1)$ \\
        $H_{-1}, Y_{-1}$ & $+-+, \quad (0,-1,-1)$ \\
        $Y_0, H_0$ & $+++, \quad (1,-1,0)$ \\
        $Y_0, H_{-1}$ & $+++, \quad (1,-1,-1)$ \\
        $Y_0, Y_{-1}$ & $++-, \quad (1,-1,-1)$ \\
        \hline
    \end{tabular}
    \end{equation}

    \textit{Additional check for annular case.}
    Since we are working in the annulus $\A_0'$, there is one more thing to check. 
    If the substring being considered is indeed the entire boundary of the web, then there is a priori a third possible move to consider first: a rule that is active on the first and third endpoints.
    By admissibility of the \textit{sign} string, 
    the only cases to consider are those where all signs are the same, namely the cases $Y_1, H_1$ and $Y_0, H_{-1}$. 
    But in either of these cases, the third state label is strictly less than the first state label, so indeed there is no rule we can apply. 

    \textit{When a choice can be made.}
    There are nine remaining cases, shown in \eqref{table:choice can be made}. In all of these, choosing to perform the `other' move first (i.e.\ operate on the right two endpoints first) yields an equivalent web.
    In the chart below, we record the remaining nine cases to check in the planar setting. 
    For example, the first row
    \begin{center}
    \begin{tabular}{|c|c|c|}
        \hline
        $H_1, U_0$ & $Y_{-1}, Y_0$ 
        & 
        $+--, \quad (1,0,-1)$
            \\
        \hline
    \end{tabular}
    \end{center}
    indicates there is an equivalence of webs
    \[
        (\id \sqcup U_0) \circ (H_1 \sqcup \id)  \cong 
        Y_0 \circ (\id \sqcup Y_{-1}),
    \]
    as shown in the Figure \ref{fig:HU-YY-example-homotopy}.

    \begin{figure}
        \centering
        
          \subcaptionbox{ \label{fig:no choice ex}}[.35\linewidth]
{
 \begin{center}
    \begin{tikzpicture}[scale=.75,
        decoration={
        markings,
        mark=at position 0.5 with {\arrow{>}}}
        ] 
        \draw (0,0)  -- (1,0);
        \draw (0,0) -- +(120:1) node[label={90:$1$}] {};
        \draw (0,0) to[out=240, in=90] (-.5, -1.732) node[label=-90:$0$] {};
        \draw (1,0) -- +(60:1) node[label={90:$0$}] {};
        \draw (1,0) -- +(-60:1) 
            -- +(-90:1.732) node[label={-90:$0$}] {};
        \draw (1.5, -0.866) -- +(1,0);
        \draw (2.5, -0.866) to[out=60, in=-90] (3, .866) node[label={90:$0$}] {};
        \draw (2.5, -0.866) -- + (-60:1) node[label={-90:$1$}] {};
    \end{tikzpicture}
    \end{center}
}
        \subcaptionbox{ \label{fig:HU-YY-example-homotopy}}[.6\linewidth]
{\includestandalone{images/HU-YY-example-homotopy}
}
        \caption{(A): An example of when there is no choice to be made. At the middle level, the state labels are $(0,1,0)$. (B): The equivalent webs $(\id \sqcup U_0) \circ (H_1 \sqcup \id)  \cong 
        Y_0 \circ (\id \sqcup Y_{-1})$.}  
    \end{figure}
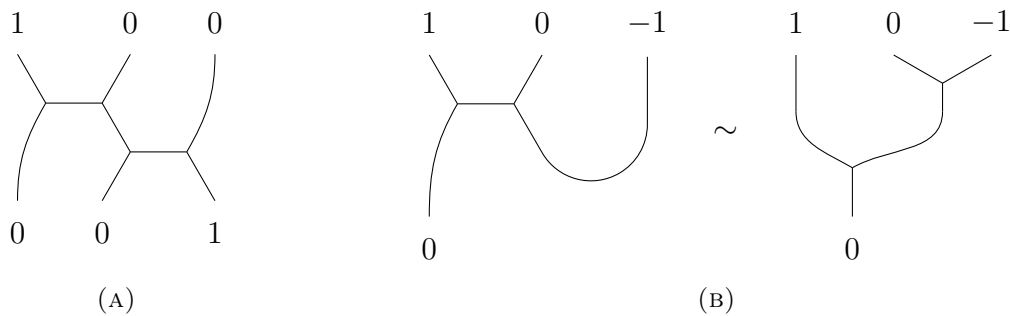

    \begin{equation}
    \label{table:choice can be made}
        \renewcommand{\arraystretch}{1.5}
        \begin{tabular}{|c|c|c|}
            \hline
            \makecell{start with left \\ two endpoints} & \makecell{start with right \\ two endpoints} & (sign, state) substrings \\
            \hline
            $H_1, U_0$ & $Y_{-1}, Y_0$      & 
            $+--, \quad (1,0,-1)$
            \\
            $H_1, Y_1$ & $H_0, Y_0$         & 
            $+-+, \quad (1,0,0)$
            \\
            $H_1, Y_0$ & $H_{-1}, Y_0$      &
            $+-+, \quad (1,0,-1)$
            \\
            $H_0, U_0$ & $Y_{-1}, Y_{-1}$   & 
            $+--, \quad (0,0,-1)$
            \\
            $H_0, Y_1$ & $H_0, Y_{-1}$      & 
            $+-+, \quad (0,0,0)$
            \\
            $H_0, Y_0$ & $H_{-1}, Y_{-1}$   & 
            $+-+, \quad (0,0,-1)$
            \\
            $Y_1, U_0$ & $Y_{-1}, U_0$      & 
            $+++, \quad (1,0,-1)$
            \\
            $Y_1, Y_1$ & $H_0, U_0$         & 
            $++-, \quad (1,0,0)$
            \\
            $Y_1, Y_0$ & $H_{-1}, U_0$      & 
            $++-, \quad (1,0,-1)$
            \\
            \hline
        \end{tabular}
    \end{equation}

    \textit{Additional check for annular case.}
    Once again, we need to check the case $Y_1, U_0$, the only case where the sign substring is an admissible sign string. Again, $-1 < 1$ so no rule can be applied to the third and first endpoints. 
    
\end{proof}

\begin{lemma}
    Any annular SL(3) web produced by the growth algorithm is non-elliptic.
\end{lemma}

\begin{proof}
    This is analogous to \cite[Lemma 2]{KhKup} but requires considering additional cases. Clearly a circle can never be produced. A bigon can be produced only if an $H$-web is closed up by a $U$-web. By examining the state string at the bottom boundary of $H$, this can only occur if the last applied move is the following $U$-web
    \begin{center}
        \begin{tikzpicture}
            \draw[very thick,\sidecolor] (-.5,0) -- (-0.5,-2);
            \draw[very thick,\sidecolor] (1.5,0) -- (1.5,-2);
            \draw[\boundarycolor] (-.5,0) -- (1.5,0);
            
            \draw (0,0) -- (0,-1);
            \draw (1,0) -- (1,-1);
            \draw (0,-.5) -- (1,.-.5);
            \draw (0,-1) arc (0:-90:.5 and .5);
            \draw (1,-1) arc (180:270:.5 and .5);
        \end{tikzpicture}
    \end{center}
    which does create a bigon, but the bigon contains the puncture.

    If a square face is produced then it must have been closed off by a $U$-web. We organize the check into four cases:
    \begin{equation*}
        \begin{aligned}
            \includestandalone{images/square_face_cases}.
        \end{aligned}
    \end{equation*}
    For cases $(a)$, $(b)$, and $(c)$, the horizontal edges may be part of $H$- or $Y$-webs. Case $(a)$ has the following four sub-cases:
     \begin{equation*}
        \begin{aligned}
            \includestandalone{images/square_face_case1}.
        \end{aligned}
    \end{equation*}
    If the state at the top of the $U$-web is $1,-1$, read from left to right (which must hold for the first three of these sub-cases), then working backwards, we see that the first $H$-web cannot occur, since its bottom boundary points must have state $0,0$:
     \begin{equation*}
        \begin{aligned}
            \includestandalone{images/square_face_case1_states}.
        \end{aligned}
    \end{equation*}
    The remaining check in case $(a)$ is when the state at the top of the $U$-web is $-1,1$, which occurs in only the fourth sub-case and will produce the web
   \begin{center}
        \begin{tikzpicture}
            \draw[very thick,\sidecolor] (-1.5,0) -- (-1.5,-2.5);
            \draw[very thick,\sidecolor] (2.5,0) -- (2.5,-2.5);
            \draw[\boundarycolor] (-1.5,0) -- (2.5,0);
            
                          \draw (0,0) -- (0,-1); 
            \draw (1,0) -- (1,-1);
            \draw (0,-.5) -- (1,.-.5);
            \draw (0,-1) .. controls ( 0,-2) .. (-1.5,-2);
            \draw (1,-1) .. controls ( 1,-2) .. (2.5,-2);
            
            \draw (2,0) .. controls (2,-1) .. (1,-1); 
            \draw (-1,0) .. controls (-1,-1) .. (0,-1);
            
            \node[below left] at (0,-1) {$\mathsmaller{-1}$};
            
            \node[below right] at (1,-1) {$\mathsmaller{1}$};
        \end{tikzpicture},
    \end{center}
    which does contain a square face, but the square face contains the puncture.

    Case (b) has the following two sub-cases:
     \begin{equation*}
        \begin{aligned}
            \includestandalone{images/square_face_case2}.
        \end{aligned}
    \end{equation*}
    If the state at the top of the $U$-web is $1,-1$, then working backwards we again see that the $H$-web is not possible:
     \begin{equation*}
        \begin{aligned}
        \includestandalone{images/square_face_case2_states}.
        \end{aligned}
    \end{equation*}
    If the state at the top of the $U$-web is $-1,1$, which can only occur in the first of the two sub-cases, then the web produced is 
    \begin{center}
        \begin{tikzpicture}
            \draw[very thick,\sidecolor] (-2.5,0) -- (-2.5,-2.5);
            \draw[very thick,\sidecolor] (2,0) -- (2,-2.5);
            \draw[\boundarycolor] (-2.5,0) -- (2,0);
            
                        \draw (0,0) -- (0,-1.5);
            \draw (1,0) -- (1,-1.5);
            \draw (0,-.5) -- (1,-.5);
            \draw (-1,0) .. controls (-1,-1) .. (0,-1);
            \draw (-2,0) .. controls  (-2,-1.5) .. (0,-1.5);
            
            \draw (0,-1.5) .. controls (0,-2.) .. (-2.5,-2); 
            \draw (1,-1.5) .. controls (1,-2) .. (2,-2);
            
            \node[below left] at (0,-1.5) {$\mathsmaller{-1}$};

            \node[below right] at (1,-1.5) {$\mathsmaller{1}$};
        \end{tikzpicture},
    \end{center}
    which contains a square face with a puncture.

    Case (c) is similar and is left to the reader. Case (d) contains two $U$-webs. The state at the boundary of the top $U$-web must be $1,-1$. If the state at the bottom $U$-web is also $1,-1$, then this is readily ruled out. Otherwise, the state at the bottom $U$-web is $-1,1$, which produces the web 
     \begin{center}
        \begin{tikzpicture}
            \draw[very thick,\sidecolor] (-.5,0) -- (-.5,-2.5);
            \draw[very thick,\sidecolor] (3.5,0) -- (3.5,-2.5);
            \draw[\boundarycolor] (-.5,0) -- (3.5,0);
            
            \draw (0,0) -- (0,-1);
            \draw (1,0) -- (1,-1);
            \draw (2,0) -- (2,-1);
            \draw (3,0) -- (3,-1);
            
            \draw (0,-.5) -- (1,-.5);
            \draw (2,-.5) -- (3,-.5);
            
            \draw (2,-1) arc (0:-180:.5 and .5);
            
            \draw (0,-1) .. controls (0,-2) .. (-.5,-2);
            
            \draw (3,-1) .. controls (3,-2) .. (3.5,-2);
            
            \node[below right] at (0,-1.5) {$\mathsmaller{-1}$};
            
            \node[below left] at (3,-1.5) {$\mathsmaller{1}$};
        \end{tikzpicture},
    \end{center}
    which contains a square face with a puncture. 
\end{proof}

\subsection{Minimal cut paths}

We will extend the minimal cut path algorithm from \cite{KhKup} to the annular setting.
We will see in Theorem \ref{thm:MCP-GROW-inverses} that this algorithm, denoted by $\mcpalg$, provides an inverse to the growth algorithm $\growalg$. 
It may be helpful to refer to Example \ref{eg:example-correspondence} while reading through this section and the next. Here is a summary of the two algorithms.
\renewcommand{\arraystretch}{1.25}
\begin{center}
\begin{tabular}{|c | c | c |}
\hline
    & $\growalg$ 
    & $\mcpalg$
    \\
\hline
    input: & $(S,J)$ an admissible (sign, state) pair 
    & a nonelliptic $SL(3)$ web $w$ in $\A_0'$
    \\
\hline 
    output: & a nonelliptic $SL(3)$ web $w^S_J$ in $\A_0'$ 
    &  a state string $J^w$, the sign string $S^w$
    \\
\hline
\end{tabular}
\end{center}
\renewcommand{\arraystretch}{1}

Here are some definitions specific to the situation we are working with (webs in $\A_0'$). 
Fix a non-elliptic web $w$ in $\A_0'$.
Let $n$ be the number of endpoints of $w$, and label them from left to right as $B_1, B_2, \ldots, B_n$. 
Let $S^w$ denote the sign string representing the sequence of $\pm$ boundary orientations at the endpoints $(B_k)_{1\leq k \leq n}$.

Let $P$ denote a point in the infinite region of $\A_0'$, below the web $w$. 
Along $\partial \A_0'$, place points $Q_1, Q_2, \ldots, Q_n$ such that $Q_k$ is on the interval between the boundary point indexed $k-1$ and the boundary point indexed $k$, with indices interpreted modulo $n$ so that $B_0 = B_n$.

\begin{definition}
A \emph{cut path from $P$ to $Q_k$} is an oriented simple path $\gamma$ from $P$ to $Q_k$ that intersects $w$ transversely. 

\begin{itemize}
    \item Let $I^+(\gamma,w)$ (\resp $I^-(\gamma,w)$) be the number of positive (\resp negative) intersections between $\gamma$ and $w$. 

    \begin{center}
        \begin{tikzpicture}[scale=.7]
            \draw[<-] (0,0) -- (2,0);
            \draw[->, densely dashed] (1,-1) -- (1,1);
            \node[right] at (1,1) {$\gamma$};
            \node[above] at (0,0) {$w$};
            \node at (1,-1.5) {\textsmaller{positive}};

            \begin{scope}[shift={(5,0)}]
                   \draw[->] (0,0) -- (2,0);
            \draw[->, densely dashed] (1,-1) -- (1,1);
            \node[right] at (1,1) {$\gamma$};
            \node[above] at (2,0) {$w$};
            \node at (1,-1.5) {\textsmaller{negative}};
            
            \end{scope}
        \end{tikzpicture}
    \end{center}
    \item The \emph{weight} of $\gamma$ is defined as 
    \[
        \wt(\gamma) := I^+(\gamma, w) \mu^+ 
        + I^-(\gamma,w) \mu^- \in \Lambda.
    \]
    \item Let 
    \[
        I(\gamma,w) := I^+(\gamma,w) + I^-(\gamma,w) 
    \]
    be the total number of intersections.
\end{itemize}

A cut path $\gamma: P \leadsto Q_k$ is \emph{minimal} if it
intersects $w$ at a minimal number of points. 

\end{definition}

When using the model $\A_0$, we will view $\A_0$ as a disk with a marked point $\times$ which will play the role of $P$ above, in that minimal cut paths start at $\times$. 

The following modification of \cite[Lemma 6.5]{Kup-spiders} describes how to relate minimal cut paths with the same endpoint. 

\begin{lemma}
\label{lem:min-cut-weight-invariance}
    Let $w$ be a non-elliptic web and let $\gamma$ and $\gamma'$ be minimal cut paths from $P$ to $Q_k$. If $\gamma$ and $\gamma'$ are disjoint except at their endpoints, then they are related by a sequence of $\Imove$-moves:\footnote{Note, this is called an $\mathrm{H}$-move in \cite{Kup-spiders,KhKup}.}
\begin{equation}
\label{eq:H move}
\begin{aligned}
    \begin{tikzpicture}[scale=.5]
        \draw (0,0) -- (0,-1);
        \draw (-1,1) -- (0,0);
        \draw (1,1) -- (0,0);
        \draw (-1,-2) -- (0,-1);
        \draw (1,-2) -- (0,-1);
        \draw[densely dashed] (0,1) .. controls (-1,-.5) ..  (0,-2);

\draw[<->] (1.5,-.5) -- (2.5,-.5);
        \begin{scope}[shift={(4,0)}]
                    \draw (0,0) -- (0,-1);
        \draw (-1,1) -- (0,0);
        \draw (1,1) -- (0,0);
        \draw (-1,-2) -- (0,-1);
        \draw (1,-2) -- (0,-1);
        \draw[densely dashed] (0,1) .. controls (1,-.5) ..  (0,-2);
        \end{scope}
    \end{tikzpicture}
    \end{aligned}
\end{equation}    
More generally, $\gamma$ and $\gamma'$ are related by $\Imove$-moves and the following three moves.
  \begin{center}
    \begin{tikzpicture}
        \draw[densely dashed] (0,0) .. controls (-.5,-.5) .. (0,-1);
        \draw[densely dashed] (-.5,0) .. controls (0,-.5) .. (-.5,-1);
        \draw[<->] (.5,-.5) -- (1,-.5);
        \draw[densely dashed]  (1.5,0) -- (1.5,-1);
        \draw[densely dashed] (2,0) -- (2,-1);
        \node[above] at (0,0) {$\gamma'$};
        \node[above] at (-.5,0) {$\gamma$};
        \node[above] at (2,0) {$\gamma'$};
        \node[above] at (1.5,0) {$\gamma$};
    \end{tikzpicture}
    \hskip4em
    \begin{tikzpicture}
        \draw (0,0) -- (1,0);
        \draw[densely dashed] (.5,.5) -- (.5,-.5);
        \draw[densely dashed] (1,.5) .. controls  (.25,.25) .. (0,-.5);
        \node[above] at (.5,.5) {$\gamma$};
        \node[above] at (1,.5) {$\gamma'$};
         \draw[<->] (1.5,0) -- (2,0);
         \begin{scope}[shift={(2.5,0)}]
           \draw (0,0) -- (1,0);
           \node[above] at (.5,.5) {$\gamma$};
        \node[above] at (1,.5) {$\gamma'$};
        \draw[densely dashed] (.5,.5) -- (.5,-.5);
        \draw[densely dashed] (1,.5) .. controls  (.75,-.25) .. (0,-.5);
        \end{scope}
    \end{tikzpicture}
     \hskip4em
    \begin{tikzpicture}
         \draw[densely dashed] (0,0) .. controls (.75,-.5) .. (.5,-1);
        \draw[densely dashed] (1,0) .. controls (.25,-.5) .. (.5,-1);
        \node[above] at (0,0) {$\gamma$};
        \node[above] at (1,0) {$\gamma'$};
        \node at (.5,-1) {$\times$};
        \draw[<->] (1.5,-.5) -- (2,-.5);
        \begin{scope}[shift={(2.5,0)}]
        \draw[densely dashed]  (0,0) -- (.5,-1);
        \draw[densely dashed] (1,0) -- (.5,-1);
         \node[above] at (0,0) {$\gamma$};
        \node[above] at (1,0) {$\gamma'$};
        \node at (.5,-1) {$\times$};
        \end{scope}
    \end{tikzpicture}
\end{center} 
In particular, the weight $\wt(\gamma)$ is independent of the choice of minimal cut path $P\leadsto Q_k$. 
\end{lemma}

\begin{proof}
For this proof, we replace $\A_0'$ by $\A_0$ and assume that $P$ is the puncture $\times$ so that $D^2:=\A_0\sqcup \{\times\}$ is a disk (with the puncture $\times$). 

We apply induction on the number of inner intersection points of $\gamma$ and $\gamma'$. Inner intersection points do not include the two boundary points $P$ (the marked point $\times$) and $Q_k$, which belong to both $\gamma$ and $\gamma'$.  

If there are no inner intersection points, then arcs $\gamma$ and $\gamma'$ separate the marked disk $D^2$ into a bigon $B$ and its complement, which has three sides. Consider the part $w\cap B$ of $w$ lying inside $B$. It is a non-elliptic web in a disk and $\gamma,\gamma'$, slightly pushed into $B$, are minimal cut paths for $w\cap B$ (if, say, $\gamma$ is not minimal for $w\cap B$ then this contradicts $\gamma$ being minimal for $w$). Next,  \cite[Lemma 6.5]{Kup-spiders}
implies that $\gamma$ and $\gamma'$  are related by a sequence of $\Imove$-moves
inside $B$, which completes the proof in this case. 

In the case of a nonempty set of inner intersection points, we want to find a bigon configuration inside the intersection system for $\gamma$ and $\gamma'$ so that these two curves each bound one side of the bigon and they are disjoint from the interior of the bigon. 

For that, deform $\gamma,\gamma'$ and $w$ without changing the combinatorics of their planar intersections so that $\gamma$ is the interval $[0,1]\in D^2$,  $P=\{0\}$  and $Q_k=\{1\}$. Pick a parametrization of $\gamma'$ and  consider $\gamma'$ as a function $(0,1]\lra D^2\setminus P\cong \SS^1\times (0,1]$. Note that $\gamma'(0)=P$, and indeed this is the only value $x$, $0\le x\le 1$, such that $\gamma'(x)=P$, since both $\gamma$ and $\gamma'$ are simple paths.  

Compose the function $\gamma'$ above with the projection to $\SS^1$ to get a function $f:(0,1]\lra \SS^1$. Pick a lift of this function to a function $\widetilde{f}:(0,1]\lra \R$, where we identify $\SS^1\cong \R/\Z$ and assume that $\widetilde{f}(1)=0$. Our isotopy can be set up so that $\widetilde{f}$ extends continuously to $0$, and necessarily $\widetilde{f}(0)\in \Z$. 

The path $\gamma'$ may travel away from $\gamma$ and then come back from the same side of the angular coordinate in $\SS^1$. This gives a non-minimal bigon, which may contain smaller bigons inside; take a minimal one. Otherwise, $\gamma'$ wraps around $\gamma$, always intersecting in the same direction. In this case we look at the bigon formed by the parts of $\gamma$ and $\gamma'$ from $P$ to their first intersection point.

More formally, 
consider all points $0<x_1<\dots <x_n<1$ such that $f(x_i)=0$ (equivalently, $\widetilde{f}(x_i)\in \Z$). Notice that $|\widetilde{f}(x_{i+1})-\widetilde{f}(x_i)|\le 1$  for all $i$ and if $\widetilde{f}(x_{i+1})=\widetilde{f}(x_i)$ then the corresponding intervals of $\gamma$ and $\gamma'$ form a minimal bigon.   

The only other possibility is for these numbers to always increase, either by $1$ or $-1$. Then the intervals of $\gamma$ and $\gamma'$ from $0$ to their first intersection point (common for both) form a bigon, which is shaded in Figure~\ref{fig:winding path}. 

\begin{figure}
    \centering
    \includestandalone{winding_path}
    \caption{}
    \label{fig:winding path}
\end{figure}

In either case, a bigon can be found and arcs of $\gamma$ and $\gamma'$ bounding it are minimal cut paths for this bigon, necessarily related by a sequence of $\Imove$-moves as argued in the first part of the proof. 
\end{proof}

For any non-elliptic annular web $w$, we can always find a system of minimal cut paths $\{\gamma_k\}$ which are pairwise disjoint. For instance, it can be built inductively by starting with the region containing the puncture and emanating outwards from it. In light of this, we adopt the following.
\begin{convention}
    Whenever considering minimal cut paths for a non-elliptic web, we will assume that they are pairwise disjoint. 
\end{convention}

As noted in the proof of \cite[Theorem 2]{KhKup}, the planar analogue of the following lemma follows from a curvature argument in \cite{Kup-spiders}. 

\begin{lemma}
\label{lem:adjacent cut paths}
    Let $w$ be a non-elliptic annular web. Consider minimal cut paths $\gamma_k$ and $\gamma_{k+1}$ from $P$ to adjacent points $Q_k$ and $Q_{k+1}$, respectively. Then after applying $\Imove$-moves (see \eqref{eq:H move}) the portion of the web between them is one of the three types shown in \eqref{eq:three options between adj paths}.
    \begin{equation}
    \label{eq:three options between adj paths}
    \begin{aligned}
        \begin{tikzpicture}[xscale=1.3]
            \draw[densely dashed] (0,0) -- (.4,-1.75);
            \draw[densely dashed] (1,0) -- (.6,-1.75);
            \draw (1,-.7) arc (60:120: 1 and .75);
             \draw (1,-.9) arc (60:120: 1 and .75);
             \node at (.5,-1.1) {$\vdots$};
           \node[above] at (0,0) {$\mathsmaller{Q_k}$};
             \node[above] at (1.2,0) {$\mathsmaller{Q_{k+1}}$};
             \node[above] at (.5,0) {$\mathsmaller{B_k}$};
             \draw (.5,0) arc (180:245:1 and .5);
        \end{tikzpicture}
        \hskip3em
            \begin{tikzpicture}[xscale=1.3]
            \draw[densely dashed] (0,0) -- (.4,-1.75);
            \draw[densely dashed] (1,0) -- (.6,-1.75);
            \draw (1,-.7) arc (60:120: 1 and .75);
             \draw (1,-.9) arc (60:120: 1 and .75);
             \node at (.5,-1.1) {$\vdots$};
               \node[above] at (0,0) {$\mathsmaller{Q_k}$};
             \node[above] at (1.2,0) {$\mathsmaller{Q_{k+1}}$};
             \node[above] at (.5,0) {$\mathsmaller{B_k}$};
             \draw (.5,0) -- (.5,-.25);
             \draw (.5,-.25) -- (-.1,-.45);
             \draw (.5,-.25) -- (1.1,-.45);
        \end{tikzpicture}
        \hskip3em
            \begin{tikzpicture}[xscale=1.3]
            \draw[densely dashed] (0,0) -- (.4,-1.75);
            \draw[densely dashed] (1,0) -- (.6,-1.75);
            \draw (1,-.7) arc (60:120: 1 and .75);
             \draw (1,-.9) arc (60:120: 1 and .75);
             \node at (.5,-1.1) {$\vdots$};
      \node[above] at (0,0) {$\mathsmaller{Q_k}$};
             \node[above] at (1.2,0) {$\mathsmaller{Q_{k+1}}$};
             \node[above] at (.5,0) {$\mathsmaller{B_k}$};
             \draw (.5,0) arc (0:-65:1 and .5);
        \end{tikzpicture}
    \end{aligned}
    \end{equation}
\end{lemma}
\begin{proof}

Let $\DD$ be the component (topologically a disk) of $\A \setminus (\gamma_k \cup \gamma_{k+1})$ that contains $B_k$. Set $w_\DD = w \cap \DD$. We view $\gamma_k$ and $\gamma_{k+1}$ as arcs in $\partial \DD$ which intersect at $\times$. Call an outer face $F$ of $w_\DD$ \emph{active} if $F\cap \partial \DD$ is contained entirely in $\gamma_k$ or entirely in $\gamma_{k+1}$. Note, if $w_\DD$ is connected then this is the same as saying both boundary vertices of $F$ are contained in $\gamma_k$ or they are both contained in $\gamma_{k+1}$. 

We first show that by using $\Imove$-moves we can arrange it so that every component of $w_\DD$ that does not contain $B_k$ is an arc joining a point on $\gamma_k$ to a point on $\gamma_{k+1}$.
Note, any component of $w_\DD$ that does not contain $B_k$ must have boundary points on both $\gamma_k$ and on $\gamma_{k+1}$, otherwise one of the two cut paths is not minimal. Take one such component $w_\DD'$ that has at least one internal vertex. Lemma \ref{lem:at least 3 small regions} gives three outer faces of $w_\DD'$, each of which is a $Y$ or an $H$. At least one of these faces $F$ is an active face of $w_\DD'$. Note that $F$ is also a face of $w_\DD$, otherwise $F$ would contain a component of $w_\DD'$ in its interior, contradicting minimality of the cut paths. Therefore $F$ is also an active face of $w_\DD$, so we may apply an $\Imove$-move to remove it. We continue this until every component not containing $B_k$ is an arc, as desired. 
 
Next, let $w_\DD^*$ denote the component of $w_\DD$ that contains $B_k$. 
As above, every outer face of $w_\DD^*$ is also an outer face of $w_\DD$. 
We claim that $w_\DD^*$ is either a single arc with no interval vertices, a tripod with exactly one internal vertex, or $w_\DD^*$ has an active $H$. This would complete the proof: the first two cases correspond to the three figures in \eqref{eq:three options between adj paths}, and in the third case we can apply an $\Imove$-move to simplify the web. 

Let $F_1, F_2$, and $F_3$ denote the (outer) faces of $w_\DD^*$ which contain $Q_k, Q_{k+1}$, and $\times$, respectively. Every other outer face of $w_\DD^*$ is active and is not a $Y$ by minimality. Note that $s(F_3) \geq 3$. If $w_\DD^*$ is not an arc then $s(F_1) , s(F_2) \geq 3$, so 
\[
s(F_1) + s(F_2) +s(F_3) \geq 9.
\]
Equality holds exactly when each $s(F_i) = 3$, which corresponds to the tripod. We may then assume $s(F_1) + s(F_2) +s(F_3) \geq 10$. Suppose that $w_\DD^*$ does not contain an active $H$. 
Equation \eqref{eq:sum of sides} implies 
\[
2 e_i + 3 f_b \geq 10 + 5(f_b - 3) + 6 f_i
\]
which can be rewritten as $2e_i \geq 2 f_b + 6f_i - 5$. Equation \eqref{eq:euler char result} gives 
\[
2v_i \geq 2 f_b + 4f_i - 3.
\]
However, $2v_i = 2f_b + 4f_i - 4$ by \eqref{eq:inner vertices}, which is a contradiction. 
\end{proof}

We will describe in more detail what the region between two minimal cut paths with a common endpoint can look like. Note that such a region may contain internal faces: 
\begin{equation}
\label{eq:balanced hexagon with paths}
\begin{aligned}
    \includestandalone{balanced_hexagon_paths}
\end{aligned}
\end{equation}
As we shall see in Lemma \ref{lem:region between cut paths}, only certain hexagons can appear. 

\begin{definition}
    \label{def:depth}
    Let $w$ be a non-elliptic annular web. For each face $F$ of $w$, take a simple path $\gamma$ from $\times$ to a point in $F$ such that $\gamma$ intersects $w$ transversely in the interior of its edges and $\vert \gamma \cap w \vert $ is minimal among all such paths. Define the \emph{depth} of $F$ to be $d(F) = \vert \gamma \cap w \vert $. 
    
\end{definition}

\begin{definition}
    \label{def:balanced}
     An edge of $w$ is \emph{balanced} if the two faces on each side have the same depth. 
     An internal vertex is \emph{balanced} if it abuts exactly 1 balanced and 2 unbalanced edges. 
     An internal hexagon in a web is balanced if, up to rotation, the following bold edges are balanced and all other depicted edges are unbalanced: 

    \begin{center}
        \includestandalone{balanced_hexagon}
    \end{center}
\end{definition}

For example, if a minimal cut path $\gamma'$ is obtained from a minimal cut path $\gamma$ by an $\Imove$-move across an edge $e$, then $e$ is balanced
The two vertices of $e$ are also balanced, because depth strictly increases along minimal cut paths. 
The hexagon in \eqref{eq:balanced hexagon with paths} is balanced. 

Note that if $F_1$ and $F_2$ are adjacent faces, then $d(F_1)$ and $d(F_2)$ differ by at most $1$. It follows that for every trivalent vertex, either one or three of the edges incident to it is balanced.

\begin{lemma}
\label{lem:one-balanced-two-unbalanced}
In a non-elliptic web, every (internal) vertex is balanced. 
\begin{proof}
Pick an internal vertex $v$ in $w$ as well as a set of minimal cut paths $\{\gamma_j\}$. Suppose $v$ lies in the region $R$ between $\gamma_i$ and $\gamma_{i+1}$. 
By Lemma \ref{lem:adjacent cut paths}, we may replace $\gamma_i, \gamma_{i+1}$ with $\gamma_i', \gamma_{i+1}' \subset R$, respectively, such that for $j \in \{i, i+1\}$,  $\gamma_j'$ is obtained from $\gamma_j$  by $\Imove$-moves. 
If $v$ is a vertex in the region between $\gamma_j$ and $\gamma_j'$ for $j \in \{i,i+1\}$, then the $\Imove$-move across $v$  verifies that $v$ is balanced. 
If $v$ is in the region between $\gamma_i'$ and $\gamma_{i+1}'$, then by Lemma \ref{lem:adjacent cut paths}, $v$ is the vertex in shown in the middle figure of \eqref{eq:three options between adj paths}, and is clearly balanced. 
\end{proof}
\end{lemma}

Note that Lemma \ref{lem:one-balanced-two-unbalanced} is not true if we allow square faces. See Figure \ref{fig:elliptic-counterexample} for an example.
\begin{figure}
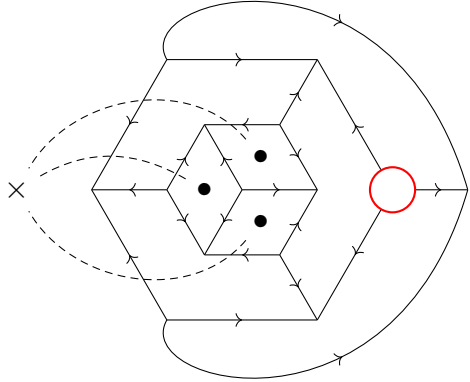

    \centering
    \includestandalone{elliptic-counterexample}
    \caption{An elliptic web containing a vertex adjacent to three balanced edges. The red circle is the outer boundary of the annulus, and the central vertex abuts three faces with the same depth.}
    \label{fig:elliptic-counterexample}
\end{figure}

\begin{lemma}
    \label{lem:region between cut paths}
    Let $\gamma$ and $\gamma'$ be disjoint minimal cut paths in a non-elliptic web $w$ ending at the same point $Q_k$. Let $B \subset \A_0$ be the bigon formed by $\gamma$ and $\gamma'$, and consider the non-elliptic web $w_B = w \cap B$. Then each internal face of $w_B$ has exactly six sides, and moreover each such internal hexagon is balanced.  
\end{lemma}

\begin{proof}
We assume that $w_B$ has internal vertices; otherwise there is nothing to show.
 It follows from\footnote{This can also be shown using Lemma \ref{lem:at least 3 small regions} by an approach similar to the proof of Lemma \ref{lem:adjacent cut paths}.}  \cite[Lemma 6.5]{Kup-spiders} that $w_B$ contains an outer $H$ face such that both of its boundary vertices are contained in $\gamma$ or they are both contained in $\gamma'$, as shown in Figure \ref{fig:active H}. 
Then $w_B$ can be formed by taking  $m$ horizontal lines $L_1, \ldots, L_m$ (Figure \ref{fig:horizontal ladder}) and placing some number of vertical intervals joining adjacent horizontal lines. Note that $m= I(\gamma, w)$. Every face of $w_B$ is then positioned between some pair of adjacent horizontal lines. If a face $F$ of $w_B$ lies between $L_i$ and $L_{i+1}$, then $d(F) = i$. All vertical edges are balanced, and all horizontal edges (formed by subdividing the horizontal lines by adding vertical edges) are unbalanced. 

Say a vertical edge is in level $i$ if it joins $L_i$ and $L_{i+1}$. Two vertical edges in level $i$ are adjacent if there is an embedded path joining them and which is otherwise disjoint from $w_B$.
If there are two adjacent vertical edges in level $i$, there must be a vertical edge in level $i-1$ or $i+1$ separating them, otherwise $w_B$ would contain a square face. In fact, there must be at least two such vertical edges since every inner face has an even number of sides, and moreover the two vertical edges are in different levels since otherwise we would obtain an infinite sequence of adjacent pairs. This rules out the possibility of inner faces with more than 6 sides and shows that the only hexagons are balanced, as shown in Figure \ref{fig:ladder hexagon}. 
\end{proof}

\begin{figure}
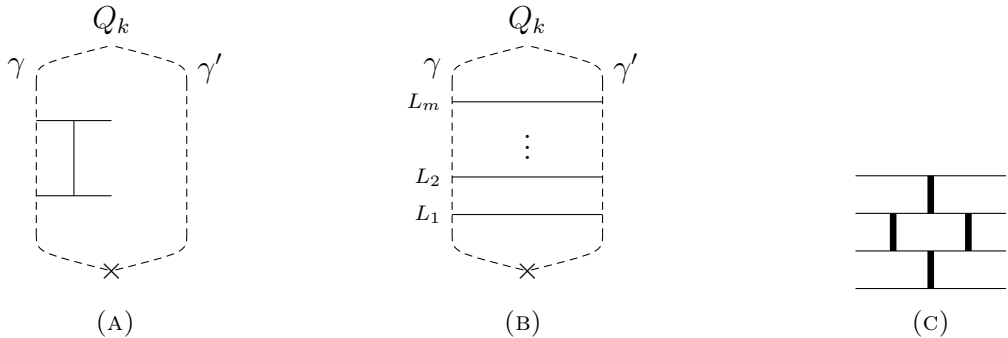

\centering 
\subcaptionbox{ \label{fig:active H}}[.32\linewidth]
{
\includestandalone{images/active_H}
}
\subcaptionbox{  \label{fig:horizontal ladder}}[.32\linewidth]
{\includestandalone{images/horizontal_ladder}
}
\subcaptionbox{ \label{fig:ladder hexagon}}[.32\linewidth]
{\includestandalone{images/ladder_hexagon}}
\caption{Figures in the proof of Lemma \ref{lem:region between cut paths}. }
\end{figure}

We are now ready to describe the minimal cut path algorithm $\mcpalg$.

\begin{definition}
Given a non-elliptic annular web $w$ with $n$ boundary points, define $\mcpalg(w) = (S^w, J^w)$ as follows. The sign sequence $S^w$ is read off from the orientations of the boundary points of $w$. Next, choose a sequence of minimal cut paths $\{\gamma_k: P \leadsto Q_k\}_{1 \leq k \leq n}$.
Define the length-$n$ state string $J^w = (j_1, j_2, \ldots, j_n)$ as follows. Lemma \ref{lem:adjacent cut paths} implies that
\[
\wt(\gamma_{k+1}) - \wt(\gamma_{k}) \in \Lambda^{S_k^w},
\]
and $j_k\in \{-1,0,1\}$ is given by $\wt(\gamma_{k+1}) - \wt(\gamma_{k})$ using the dictionary in Figure \ref{fig:lattice and states}. Indices are interpreted modulo $n$, so that $\gamma_{n+1} = \gamma_1$. 
\end{definition}

Lemma \ref{lem:min-cut-weight-invariance} implies that $J^w$ does not depend on the choice of minimal cut paths. Note that $J^w$ corresponds to a closed lattice path by construction. Moreover, Lemma \ref{lem:adjacent cut paths} implies that 
\[
j_k = I(\gamma_{k+1},w) - I(\gamma_k,w),
\]
which we will use frequently.

\subsection{Bijection between \texorpdfstring{$SL(3)$}{SL(3)} webs and closed lattice paths}
\label{sec:bijection between webs and paths}

Our goal in this section is to prove the following extension of \cite[Proposition 1]{KhKup} to the annular setting. 

\begin{theorem}
\label{thm:MCP-GROW-inverses}
The algorithms $\growalg$ and $\mcpalg$ are inverses. 
In particular, the set $B^S$ is finite for any sign string $S$. 
\end{theorem}

We will show in Proposition \ref{prop:mcp-circ-grow} and Proposition \ref{prop:grow-circ-mcp} that $\mcpalg \circ \growalg = \id$ and $\growalg \circ \mcpalg = \id$, respectively. We begin 
by addressing $\mcpalg \circ \growalg$, for which we will use the relationship between minimal cuts and maximal flows, following \cite{KhKup}.

\begin{definition}
    A \emph{flow} in a web $w$ is a subgraph $\Phi$ of $w$ which contains exactly two out of the three edges at each internal vertex of $w$. Components of $\Phi$ are called \emph{flow lines}. Each flow line  (either an arc or a loop) must be oriented, but the orientation need not agree with that of $w$.

\end{definition}

We imagine a flow line as the path of a particle traveling around the annulus $\A_0'$. 
For a web $w$ produced by $\growalg$ applied to the admissible pair $(S,J)$, we build a flow and a system of cut paths by gluing together the local models in Figure \ref{fig:flow-growth-cuts}. 

\begin{remark}
    Any non-elliptic web $w$ has a \emph{canonical flow} $\Phi_w$, as follows. Lemma \ref{lem:one-balanced-two-unbalanced} implies that deleting all balanced edges in $w$ results in an unoriented flow. We orient the resulting loops and arcs clockwise. It will follow from Lemma \ref{lem:adjacent cut paths}, Lemma \ref{lem:local cut paths are minimal}, and Proposition \ref{prop:mcp-circ-grow} that the flow built from Figure \ref{fig:flow-growth-cuts} is precisely $\Phi_w$. While this is not needed for any of our arguments, we nonetheless refer to the flow  built from Figure \ref{fig:flow-growth-cuts} as $\Phi_w$.
\end{remark}

The flow built from Figure \ref{fig:flow-growth-cuts} is \emph{monotone rightward}, in the sense that the particle always travels to the right or vertically, but never to the left. Moreover, particles move downward (\resp upward) along edges labeled $+1$ (\resp $-1$). In particular, particles enter the web traveling downward (where $j_k = +1$), and leave the web traveling upward (where $j_k = -1$). At boundary points labeled $j_k =0$, no particle enters or exits along that edge.

\begin{figure}
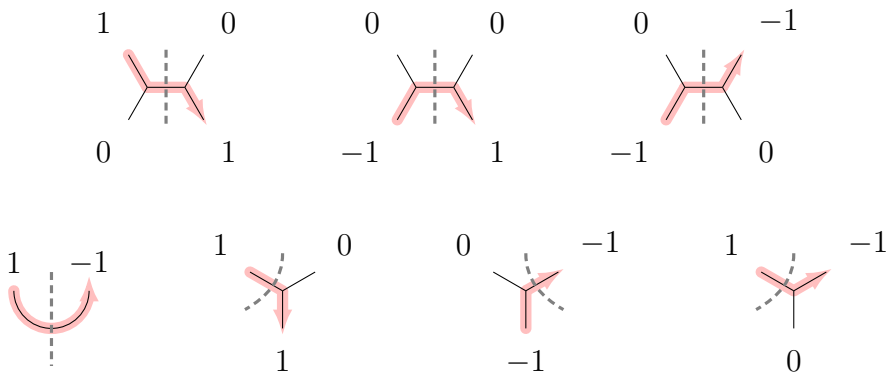

    \centering
    \includestandalone{images/flow-growth-cuts}
    \caption{The local models for the monotone rightward flow (highlighted pink) and cut paths (dashed gray) from the proof of \cite[Proposition 1]{KhKup}.
    }
    \label{fig:flow-growth-cuts}
\end{figure}

\begin{remark}
In all but one of the local models in Figure \ref{fig:flow-growth-cuts}, the choice of cut path is determined by the flow. 
However, for the web $Y_0$, a choice was made to send the cut path to the left of the $0$-labeled edge.
This is inconsequential, as the choice just tells us that the cut path that used to be between the $1$- and $-1$-labeled edges is now parallel to the cut path originally to the left of the $1$-labeled edge. 
Lemma \ref{lem:min-cut-weight-invariance} ensures that when we use $\mcpalg$ to compute $J^w$, this choice does not matter. 
\end{remark}

As in the proof of \cite[Proposition 1]{KhKup}, the dashed cut path pieces in Figure \ref{fig:flow-growth-cuts} may be replaced with any number of parallel cuts. That is, if there are $t$ parallel cut paths entering the top region of the local picture, then replace the dotted line with $t$ parallel copies of dotted lines.
We define the set of cut paths $\{\gamma_k\}_{1 \leq k \leq n}$ by gluing together these local dotted paths within regions of $\A_0'$ cut out by the web $w$.

\begin{lemma}
\label{lem:local cut paths are minimal}
    The cut paths $\gamma_k$ built from the models in Figure \ref{fig:flow-growth-cuts} are minimal. 
\end{lemma}
\begin{proof}
Observe from Figure \ref{fig:flow-growth-cuts} that $\Phi_w$ is actually a collection of arcs with endpoints at the boundary, along with closed loops, each of which generate $\pi_1(\A)$ since the flow is monotone rightward. 
We know that the minimum number of possible intersections between a cut path $P \leadsto Q_k$ is bounded below by the number of components of $\Phi_w$ that separate $Q_k$ from $P$.
But by construction of $\gamma_k$, we precisely have $I(\gamma_k, w) = I(\gamma_k, \Phi_w)$.
Therefore $\gamma_k$ is a minimal cut path $P \leadsto Q_k$ in $\Phi_w$, which in turn implies that $\gamma_k$ is also a minimal cut path in $w$, a graph with more edges. 
\end{proof}

\begin{proposition}
\label{prop:mcp-circ-grow}
For any admissible pair $(S,J)$, we have $(\mcpalg \circ \growalg)(S,J) = (S,J)$. 
\end{proposition} 
\begin{proof}
Now let $R_k$ denote the region in $\A_0'$ bounded by $\gamma_k$ on the left and $\gamma_{k+1}$ on the right
(the region $R_n$ is infinite).
By construction, we ensured that every intersection of $\gamma_k$ with $w$ corresponds to a particle traveling from $R_k$ to $R_{k+1}$ by crossing $\gamma_k$. 

By conservation of flow, the difference
\[  
    I(\gamma_{k+1}, w) -  I(\gamma_k, w) 
\]
is equal to the signed number of particles entering ($-1$) or exiting ($+1$) the web at $\partial \A_0'$ along the interval $[Q_k, Q_{k+1}]$ (with indices considered cyclically). 
Now by induction on the number of steps of the growth algorithm, we see that the label $j_k$ on $B_k$ is precisely $I(\gamma_{k+1}, w) -  I(\gamma_k, w)$. 
(The base case is the $U$-web.)
\end{proof}

The following example illustrates the above arguments. 

\begin{example}
\label{eg:example-correspondence}
Let $S = (+,+,+)$ and $J = (-1,0,1)$. 
The corresponding path in the weight lattice is shown in Figure \ref{fig:eg-lattice-path} as the orange loop, beginning at the origin. 

%%%%%%%%%%%%%
\begin{figure}
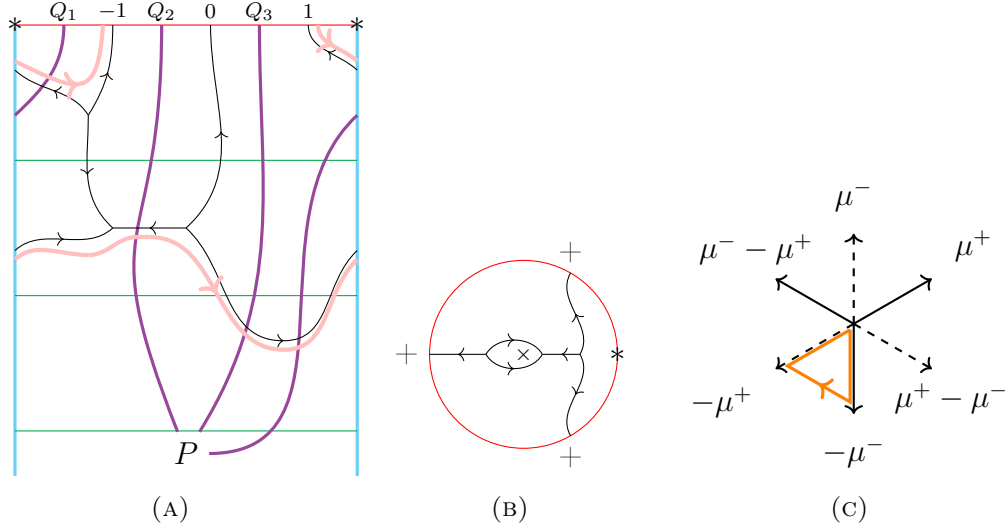

\centering 
\subcaptionbox{\label{fig:flat-annulus-example}}[.28\linewidth]
{\includestandalone{images/flat-annulus-growth-example}
}
\subcaptionbox{\label{fig:disk-like-example}}[.25\linewidth]
{\includestandalone{images/disk-annulus-growth-example}}
\subcaptionbox{\label{fig:eg-lattice-path}}[.28\linewidth]
{\includestandalone{images/eg-lattice-path}}
\caption{
Left: The web corresponding to $(S,J) = ((+,+,+), (-1,0,1))$ drawn on $\A_0'$, along with minimal cut paths and flow lines discussed in Example \ref{eg:example-correspondence}.
Middle: The web corresponding to $(S,J) = ((+,+,+), (-1,0,1))$, drawn on $\A_0$.
Right: The corresponding lattice path. 
}
\label{fig:annulus-example}
\end{figure}

%%%%%%%%%%%%%
In Figure \ref{fig:flat-annulus-example}, we obtain the web $w$ drawn in black by beginning with the data at the very top, including the state string and the boundary orientations (i.e.\ sign string), and applying the growth algorithm. The three steps are separated by {\color{Green} green} horizontal lines in the figure: first we apply a type $Y$ move, then type $H$, then type $U$. Finally, draw $P$ at the bottom. The corresponding web on the disk-like annulus $\A_0$ is shown in Figure \ref{fig:disk-like-example}.

Conversely, we may begin with the web $w$ and consider a set of minimal cut paths, drawn in {\color{\cutcolor} purple}.
The rightward flow is depicted in {\color{\flowcolor} pink} and drawn slightly offset from the web $W$ for clarity. 
The reader may verify that the number of particles leaving the web at each of the boundary points exactly recovers the state string $J = (-1,0,1)$. 

Here is the computation of the weights of the min-cuts as well as their differences, which determine the state string $J^w = (j_1, j_2, j_3)$. 

\renewcommand{\arraystretch}{1.5}
\begin{center}
\begin{tabular}{| c | l | l | c |}
    \hline
    $k$ & $\pi_k$ 
        & $\pi_{k+1} - \pi_{k}$ 
        &  $j_k$ \\ 
    \hline
    $1$ & $\pi_1 = \mu^+ + \mu^-$ 
        & $\pi_2 - \pi_1 = -\mu^-$ 
        & $-1$ \\
    $2$ & $\pi_2 = \mu^+$ 
        & $\pi_3 - \pi_2 =  \mu^- - \mu^+$
        & $0$ \\
    $3$ & $\pi_3 = \mu^-$ 
        & $\pi_1 - \pi_3 = \mu^+$
        & $1$ \\
    \hline
\end{tabular}
\end{center}
\end{example}

We now begin to address $\growalg \circ \mcpalg$. To that end, we have the following lemmas. In what follows, a label (typically $k, k-1$, or $k+1$) above the endpoint of a cut path means the number of intersection points between the path and the web.

\begin{lemma}
    \label{lem:no 1 -1 H}
    For any  non-elliptic annular web $w$, minimal cut paths never have the following intersection configuration at any boundary $H$ of $w$:
  \begin{center}
    \begin{tikzpicture}
        \draw (0,0) -- (0,-1);
        \draw (1,0) -- (1,-1); 
        \draw (0,-.5) -- (1,-.5);
        \draw[densely dashed] (-.5,0) -- (-.5,-1);
        \draw[densely dashed] (.5,0) -- (.5,-.25);
        \draw[densely dashed] (1.5,0) -- (1.5,-1);
        \node[above] at (-.5,0) {$\mathsmaller{k}$};
        \node[above] at (.5,0) {$\mathsmaller{k+1}$};
        \node[above] at (1.5,0) {$\mathsmaller{k}$};
    \end{tikzpicture}
\end{center}
\end{lemma}

\begin{proof}
    Suppose that this intersection configuration does occur at a boundary $H$ of $w$. Take any minimal cut path $\gamma$ that ends outside of this $H$. By minimality, $\gamma$ does not pass inside the $H$:
    \begin{center}
        \begin{tikzpicture}
                 \draw (0,0) -- (0,-1);
        \draw (1,0) -- (1,-1); 
        \draw (0,-.5) -- (1,-.5);
          \draw[densely dashed] (.5,-1.25) .. controls (.5,-.5).. (-.5,0);
        \end{tikzpicture}
    \end{center}
    Cutting $w$ along $\gamma$ results in a non-elliptic web $w'$ in the disk (with, in general, more boundary points than $w$). Note that this boundary $H$ in $w$ is also present in $w'$. However, we see that $\mcpalg(w')$ gives $(1,-1)$ at this boundary $H$ in $w'$. Applying $\growalg$ would produce a $U$, contradicting the fact that $(\growalg \circ \mcpalg)(w') = w'$ from \cite[Proposition 1]{KhKup}.
\end{proof}

\begin{lemma}
    \label{lem:no far apart cut paths} Let $F$ be an outer face which does not contain the puncture. Let $\gamma$ be a minimal cut path ending in $F$, and let $e(\gamma)$ be the unique edge of $F$ which intersects $\gamma$. If $\gamma'$ is another minimal cut path ending in $F$, then $e(\gamma')$ and $e(\gamma)$ share a vertex. 
\end{lemma}

\begin{remark}
 For instance, the above lemma says that the configuration of minimal cut paths shown in \eqref{eq:far apart cut paths example} is impossible:
  \begin{equation}
  \label{eq:far apart cut paths example}
  \begin{aligned}
        \begin{tikzpicture}
        \draw[\boundarycolor] (-.25,0) -- (1.25,0);
                 \draw (0,0) -- (0,-1);
        \draw (1,0) -- (1,-1); 
        \draw (0,-.5) -- (1,-.5);
          \draw[densely dashed] (.25,0) .. controls (.15,-.25) .. (-.25,-.3);
          \draw[densely dashed] (.75,0) .. controls (.85,-.25) .. (1.25,-.3);
        \end{tikzpicture}
    \end{aligned}
    \end{equation}
    Note that this also implies Lemma \ref{lem:no 1 -1 H}. 
\end{remark}

\begin{proof}[Proof of Lemma \ref{lem:no far apart cut paths}]
    Let us begin by establishing some notation; see Figure \ref{fig:no far apart cut paths figure 1} for a summary. Order the vertices of $F$ as $v_1, \ldots, v_n$ so that $v_1, v_n \in \partial \A$, $v_1$ is connected to $v_n$ by an interval in $\partial \A$, and $v_i$ is connected to $v_{i+1}$ by an edge of $w$ for $1\leq i \leq n-1$. Suppose $\gamma, \gamma'$ exist where $e(\gamma)$ and $e(\gamma')$ do not share a vertex. Then $e(\gamma)$ connects $v_i, v_{i+1}$, $e(\gamma')$ connects $v_{j}, v_{j+1}$, and without loss of generality $v_{i+1} < v_j$. Pick any $i+1 \leq k < j$, and consider the face $F'$ below the edge connecting $v_k$ and $v_{k+1}$. 

    First, minimality implies that $F'$ does not contain the puncture. Next we argue that $F'$ is an inner face. Suppose instead that $F'$ is outer. Take a properly embedded arc $\alpha \subset \A_0$ as indicated in Figure \ref{fig:no far apart cut paths figure 2}. Exactly one of the two components of $\A\setminus \alpha$ contains the puncture, and without loss of generality assume that it is also the component containing $v_i$. Then $\gamma'$ must eventually pass through $F'$, intersecting exactly two of its boundary sides. We can then find a cut path that intersects $w$ fewer times than $\gamma'$ does, which is a contradiction. 

    Therefore $F'$ is an inner face. Since $\gamma$ and $\gamma'$ are related by a sequence of $\Imove$-moves, we see that the two edges of $F'$ that are incident to $v_k$ and $v_{k+1}$ are balanced, as indicated in Figure \ref{fig:no far apart cut paths figure 3}. But this is impossible by Lemma \ref{lem:region between cut paths}. 
\end{proof}

\begin{figure}
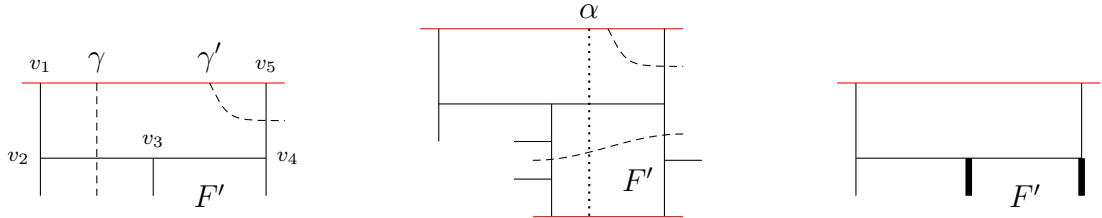

\centering 
\subcaptionbox{Here, $n=5$, $i=2$, $j=4$, and we take $k=3$.  \label{fig:no far apart cut paths figure 1}}[.32\linewidth]
{\includestandalone{images/no_far_apart_cut_paths_figure1}
}
\subcaptionbox{The arc $\alpha$ is dotted, and the dashed segments are parts of $\gamma'$. \label{fig:no far apart cut paths figure 2}}[.32\linewidth]
{\includestandalone{images/no_far_apart_cut_paths_figure2}}
\subcaptionbox{ The two bold edges of $F'$ must be balanced. \label{fig:no far apart cut paths figure 3}}[.32\linewidth]
{
\includestandalone{images/no_far_apart_cut_paths_figure3}
}
\caption{A schematic of the notation in the proof of Lemma \ref{lem:no far apart cut paths}. }\label{fig:no far apart cut paths}
\end{figure}

The following and Proposition \ref{prop:mcp-circ-grow} complete the proof of Theorem \ref{thm:MCP-GROW-inverses}. 

\begin{proposition}
\label{prop:grow-circ-mcp}
    For any non-elliptic annular $SL(3)$ web $w$, we have $(\growalg \circ \mcpalg)(w) = w$. 
\end{proposition}

\begin{proof}
    We proceed by induction on the number of vertices $N$ (including boundary vertices) of $w$. The base case is $N=2$, and Lemma \ref{lem:2 boundary points} implies there are two options for $w$, for which it is straightforward to check that the statement holds. 

    Next, Lemma \ref{lem:U Y or H} implies that $w$ has an outer face $F$ which is a $U$, a $Y$, or an $H$. Let $w'=w\setminus F$ be the non-elliptic web obtained by deleting $F$. We have $(\growalg\circ \mcpalg)(w') = w'$ by inductive hypothesis. We consider the three possibilities for $F$.
    \vskip1em
    \noindent
    \underline{$F=U$.} Minimal cut paths in $w'$ extend naturally to minimal cut paths of $w$: 
    \begin{center}
    \begin{tikzpicture}
        \draw (0,0) arc (0:-180: .5 and .5);
        \draw[densely dashed] (-.7,-1) -- (-1.5,0);
        \draw[densely dashed] (-.5,-1) -- (-.5,0);
        \draw[densely dashed] (-.3,-1) -- (.5,0);
        \node[above] at (-1.5,0) {$\mathsmaller{k}$};
        \node[above] at (-.5,0) {$\mathsmaller{k+1}$};
        \node[above] at (.5,0) {$\mathsmaller{k}$};
    \end{tikzpicture}
\end{center}
 We see that $\mcpalg(w)$ produces states $(1,-1)$ at $F$. Applying one step of $\growalg$ produces $F$ followed by the state $\mcpalg(w')$. Continuing to apply $\growalg$ produces $w'$ by inductive hypothesis, which concludes the proof in this case. 
        \vskip1em
    \noindent
    \underline{$F=Y$.} Minimal cut paths in $w'$ extend again to minimal cut paths in $w$, except there are three cases depending on the state at the boundary point of $w'$ corresponding to the bottom point of $F$:
    \begin{center}
    \begin{tikzpicture}
        \draw (0,0) -- (0,-1);
        \draw[densely dashed] (-.75,0) -- (-.75,-1);
        \draw[densely dashed] (.75,0) -- (.75,-1);
        \node[above] at (-.75,0) {$\mathsmaller{k}$};
        \node[above] at (.75,0) {$\mathsmaller{k-1}$};
        \node[above] at (0,0) {$\mathsmaller{-1}$};
        \node[] at (0,-1.25) {$w'$};
\node at (1.5,-.5) {$\Rightarrow$};
 \begin{scope}[shift={(2.75,0)}]
           \draw (0,0) -- (.5,-.5);
        \draw (1,0) -- (.5,-.5);
        \draw (.5,-.5) -- (.5,-1);
        \draw[densely dashed] (-.5,0) -- (-.5,-1);
        \draw[densely dashed] (.5,0)-- (.85,-1);
        \draw[densely dashed] (1.5,0) -- (1.5,-1);
        \node[above] at (-.5,0) {$\mathsmaller{k}$};
        \node[above] at (.5,0) {$\mathsmaller{k}$};
        \node[above] at (1.5,0) {$\mathsmaller{k-1}$};
        \node[] at (.5,-1.27) {$w$};
 \end{scope}
    \end{tikzpicture}\vskip1em
    %%%%%%%%%%%%%%%%%%%%%%%%%
      \begin{tikzpicture}
        \draw (0,0) -- (0,-1);
        \draw[densely dashed] (-.75,0) -- (-.75,-1);
        \draw[densely dashed] (.75,0) -- (.75,-1);
        \node[above] at (-.75,0) {$\mathsmaller{k}$};
        \node[above] at (.75,0) {$\mathsmaller{k}$};
        \node[above] at (0,0) {$\mathsmaller{0}$};
        \node[] at (0,-1.25) {$w'$};
\node at (1.5,-.5) {$\Rightarrow$};
 \begin{scope}[shift={(2.75,0)}]
        \draw (0,0) -- (.5,-.5);
        \draw (1,0) -- (.5,-.5);
        \draw (.5,-.5) -- (.5,-1);
        \draw[densely dashed] (-.5,0) -- (-.5,-1);
        \draw[densely dashed] (.5,0)-- (.85,-1);
        \draw[densely dashed] (1.5,0) -- (1.5,-1);
        \node[above] at (-.5,0) {$\mathsmaller{k}$};
        \node[above] at (.5,0) {$\mathsmaller{k+1}$};
        \node[above] at (1.5,0) {$\mathsmaller{k}$};
        \node[] at (.5,-1.27) {$w$};
        \node at (2.5,-.5) {or};
 \end{scope}
 \begin{scope}[shift={(6.5,0)}]
           \draw (0,0) -- (.5,-.5);
        \draw (1,0) -- (.5,-.5);
        \draw (.5,-.5) -- (.5,-1);
        \draw[densely dashed] (-.5,0) -- (-.5,-1);
        \draw[densely dashed] (.5,0)-- (.15,-1);
        \draw[densely dashed] (1.5,0) -- (1.5,-1);
        \node[above] at (-.5,0) {$\mathsmaller{k}$};
        \node[above] at (.5,0) {$\mathsmaller{k+1}$};
        \node[above] at (1.5,0) {$\mathsmaller{k}$};
        \node[] at (.5,-1.27) {$w$};
 \end{scope}
    \end{tikzpicture}\vskip1em
    %%%%%%%%%%%%%%%%%%%%%%%%%%%%%%%%%%%
      \begin{tikzpicture}
        \draw (0,0) -- (0,-1);
        \draw[densely dashed] (-.75,0) -- (-.75,-1);
        \draw[densely dashed] (.75,0) -- (.75,-1);
        \node[above] at (-.75,0) {$\mathsmaller{k}$};
        \node[above] at (.75,0) {$\mathsmaller{k+1}$};
        \node[above] at (0,0) {$\mathsmaller{1}$};
         \node[] at (0,-1.25) {$w'$};
\node at (1.5,-.5) {$\Rightarrow$};
 \begin{scope}[shift={(2.75,0)}]
           \draw (0,0) -- (.5,-.5);
        \draw (1,0) -- (.5,-.5);
        \draw (.5,-.5) -- (.5,-1);
        \draw[densely dashed] (-.5,0) -- (-.5,-1);
        \draw[densely dashed] (.5,0)-- (.15,-1);
        \draw[densely dashed] (1.5,0) -- (1.5,-1);
        \node[above] at (-.5,0) {$\mathsmaller{k}$};
        \node[above] at (.5,0) {$\mathsmaller{k+1}$};
        \node[above] at (1.5,0) {$\mathsmaller{k+1}$};
         \node[] at (.5,-1.27) {$w$};
    \end{scope}
    \end{tikzpicture}
    \end{center}
    In the second case, either of the two extensions give minimal cut paths for $w$. By examining the states at $F$ given by $\mcpalg(w)$ and comparing with the $Y$-type growth webs, we see that in all three cases we can apply $\growalg$ once to $\mcpalg(w)$ to produce $F$, and the resulting state string is precisely $\mcpalg(w')$. Continuing to apply $\growalg$ produces $w'$ by inductive hypothesis, finishing the proof in this case. 
        \vskip1em
    \noindent
    \underline{$F=H$.} Note that minimal cut paths in $w'$ do not necessarily extend to $w$. Since $\growalg$ produces $w'$, we must have that $w'$ has a $U$, a $Y$ or a \emph{good} $H$ at its boundary, where a good $H$ has states $(1,0)$, $(0,-1)$, or $(0,0)$ according to $\mcpalg(w')$ at its top boundary points. Let $F'$ denote one of these $U$'s, $Y$'s, or good $H$'s in $w'$. Denote by $p_1, p_2$ the two bottom points of $F=H$ in $w$, so that $p_1, p_2$ are boundary points of $w'$. Let $q_1, q_2$ denote the boundary points of $F'$.

    Suppose first that $\{p_1, p_2\} \cap \{q_1 , q_2\} =\varnothing$. In this case, we may view $F'$ as an outer face of $w$. If $F' = U$ or $F' = Y$ then this reduces to an earlier case. Suppose that $F'$ is a good $H$ in $w'$. 
    Minimal cut paths in $w'$ near $F'$ have one of the three intersection configurations shown in \eqref{eq:if good H is far away}. 
    \begin{equation}
    \label{eq:if good H is far away}
    \begin{aligned}
        \begin{tikzpicture}
            \draw(0,0) -- (0,-1);
            \draw (1,0) -- (1,-1);
            \draw (0,-.5) -- (1,-.5);
            \draw[densely dashed] (-.5,0) -- (-.5,-1);
             \draw[densely dashed] (.5,0) -- (.5,-1);
              \draw[densely dashed] (1.5,0) -- (1.5,-1);
              \node[above] at (-.5,0) {$\mathsmaller{k-1}$};
              \node[above] at (.5,0) {$\mathsmaller{k}$};
              \node[above] at (1.5,0) {$\mathsmaller{k}$};
        \end{tikzpicture}
        \hskip3em
                \begin{tikzpicture}
            \draw(0,0) -- (0,-1);
            \draw (1,0) -- (1,-1);
            \draw (0,-.5) -- (1,-.5);
            \draw[densely dashed] (-.5,0) -- (-.5,-1);
             \draw[densely dashed] (.5,0) -- (.5,-1);
              \draw[densely dashed] (1.5,0) -- (1.5,-1);
              \node[above] at (-.5,0) {$\mathsmaller{k}$};
              \node[above] at (.5,0) {$\mathsmaller{k}$};
              \node[above] at (1.5,0) {$\mathsmaller{k-1}$};
        \end{tikzpicture}
        \hskip3em        \begin{tikzpicture}
            \draw(0,0) -- (0,-1);
            \draw (1,0) -- (1,-1);
            \draw (0,-.5) -- (1,-.5);
            \draw[densely dashed] (-.5,0) -- (-.5,-1);
             \draw[densely dashed] (.5,0) -- (.5,-1);
              \draw[densely dashed] (1.5,0) -- (1.5,-1);
              \node[above] at (-.5,0) {$\mathsmaller{k}$};
              \node[above] at (.5,0) {$\mathsmaller{k}$};
              \node[above] at (1.5,0) {$\mathsmaller{k}$};
        \end{tikzpicture}
        \end{aligned}
    \end{equation}
    Lemma \ref{lem:adjacent cut paths} implies that a minimal cut path for the middle region can always be chosen to intersect the horizontal rung of $F'$. Moreover, since $\{p_1, p_2\} \cap \{q_1 , q_2\} =\varnothing$, we can take the same minimal cut paths near $F'$ in $w$.  Setting $w''  = w\setminus F'$, we have the following intersection configurations in $w''$, corresponding to the three possibilities in \eqref{eq:if good H is far away}.
    \begin{center}
    \begin{tikzpicture}
        \draw (0,0) -- (0,-.5); 
        \draw (1,0) -- (1,-.5);
        \draw[densely dashed] (-.5,0) -- (-.5,-.5);
        \draw[densely dashed] (.5,0) -- (.5,-.5);
        \draw[densely dashed] (1.5,0) -- (1.5,-.5);
        \node[above] at (-.5,0) {$\mathsmaller{k-1}$};
        \node[above] at (.5,0) {$\mathsmaller{k-1}$};
        \node[above] at (1.5,0) {$\mathsmaller{k}$};
        \end{tikzpicture}
        \hskip3em
           \begin{tikzpicture}
        \draw (0,0) -- (0,-.5); 
        \draw (1,0) -- (1,-.5);
        \draw[densely dashed] (-.5,0) -- (-.5,-.5);
        \draw[densely dashed] (.5,0) -- (.5,-.5);
        \draw[densely dashed] (1.5,0) -- (1.5,-.5);
        \node[above] at (-.5,0) {$\mathsmaller{k}$};
        \node[above] at (.5,0) {$\mathsmaller{k-1}$};
        \node[above] at (1.5,0) {$\mathsmaller{k-1}$};
        \end{tikzpicture}
        \hskip3em
           \begin{tikzpicture}
        \draw (0,0) -- (0,-.5); 
        \draw (1,0) -- (1,-.5);
        \draw[densely dashed] (-.5,0) -- (-.5,-.5);
        \draw[densely dashed] (.5,0) -- (.5,-.5);
        \draw[densely dashed] (1.5,0) -- (1.5,-.5);
        \node[above] at (-.5,0) {$\mathsmaller{k}$};
        \node[above] at (.5,0) {$\mathsmaller{k-1}$};
        \node[above] at (1.5,0) {$\mathsmaller{k}$};
        \end{tikzpicture}
    \end{center}
Comparing each one with the $H$-type growth web and using that $(\growalg\circ \mcpalg)(w'') = w''$ by induction completes the proof in this case.

Assume now that $\{p_1, p_2\} \cap \{q_1 , q_2\}\neq \varnothing$.  We may then also assume that $\{p_1, p_2\} \cap \{q_1 , q_2\}$ consists of exactly one point. Otherwise we have $\{p_1, p_2\} = \{q_1 , q_2\}$. This is impossible if $F'=Y$ for orientation reasons. If $F' = U$ then $w$ must be the web shown in the middle of \eqref{eq:length two ex}, for which the statement is easily verified. Finally, arguing as in the proof of Lemma \ref{lem:2 boundary points} shows that $F'=H$ is impossible. 

We now have six further cases depending on $F'$. Three are shown in \eqref{eq:three cases for F'}, and the other three are reflections of these through a vertical line. 
\begin{equation}
\label{eq:three cases for F'}
\begin{aligned}
    \begin{tikzpicture}
         \draw (0,0) -- (0,-1.5);
        \draw (1,0) -- (1,-1); 
        \draw (0,-.5) -- (1,-.5);
        \draw (2,-1) arc (0:-180:.5 and .5);
        \draw (2,0) -- (2,-1);
        \node at (1,-2) {$F'=U$};
    \end{tikzpicture} \hskip4em
    %%%%%%%%
        \begin{tikzpicture}
         \draw (0,0) -- (0,-1.5);
        \draw (1,0) -- (1,-.5); 
        \draw (0,-.5) -- (1,-.5);
        \draw (1,-.5) -- (1.5,-1);
        \draw (2,0) -- (1.5,-1);
        \draw (1.5,-1) -- (1.5,-1.5);
        \node at (1,-2) {$F'=Y$};
    \end{tikzpicture}\hskip4em
    %%%%%%%%%
        \begin{tikzpicture}
         \draw (0,0) -- (0,-1.5);
        \draw (1,0) -- (1,-1.5); 
        \draw (0,-.5) -- (1,-.5);
        \draw (2,0) -- (2,-1.5);
        \draw (1,-1) -- (2,-1);
        \node at (1,-2) {$F'=H$};
    \end{tikzpicture}
    \end{aligned}
\end{equation}
We will only address the three depicted cases; the arguments for their reflections are analogous and are left to the reader. 

In the first case $F'=U$, we see that $w$ has a boundary $Y$, which reduces to a previous case. Suppose then $F'=Y$. Considering a minimal cut path $\gamma'$ in $w'$ shown below on the left, which intersects $w'$ in $k$ points, we see that there are nine possibilities for labeling the nearby regions. 
\begin{center}
    \begin{tikzpicture}
    \draw (-1,0) -- (-1,-1);
         \draw (0,0) -- (.5,-.5);
        \draw (1,0) -- (.5,-.5);
        \draw (.5,-.5) -- (.5,-1);
        \draw[densely dashed] (-.5,0) -- (-.5,-1);
        \node[above] at (-.5,0) {$\mathsmaller{k}$};
    \end{tikzpicture} \hskip4em
%%%%%%%%%%%%% 
    \begin{tikzpicture}
    \draw (-1,0) -- (-1,-1);
         \draw (0,0) -- (.5,-.5);
        \draw (1,0) -- (.5,-.5);
        \draw (.5,-.5) -- (.5,-1);
        \draw[densely dashed] (-.5,0) -- (-.5,-1);
        \node[above] at (-.5,0) {$\mathsmaller{k}$};
           \node[above] at (-1.5,0) {$\mathsmaller{a}$};
              \node[above] at (.5,0) {$\mathsmaller{b}$};
                 \node[above] at (1.5,0) {$\mathsmaller{c}$};
    \end{tikzpicture}
\end{center}
The tuple $(a,b,c)$ indicated on the right (we do not draw minimal cut paths to avoid clutter) must be one of 
\begin{align}
\begin{aligned}
\label{eq:nine cases for Y}
    (k+1, k+1, k+1), & & (k,k+1, k+1), &  & (k-1, k+1,k+1), \\
    (k+1, k, k-1), &  & (k,k, k-1), & & (k-1, k ,k-1), \\
    (k+1, k+1, k), & & (k,k+1, k), & & (k-1, k+1 ,k).
    \end{aligned}
\end{align}
For each of these we can read off the states given by $\mcpalg(w')$; for instance, $(k+1, k, k-1)$ corresponds to
\begin{equation}
\begin{aligned}
\label{eq:ex case for Y}
    \begin{tikzpicture}
    \draw (-1,0) -- (-1,-1);
         \draw (0,0) -- (.5,-.5);
        \draw (1,0) -- (.5,-.5);
        \draw (.5,-.5) -- (.5,-1);
        \node[above] at (-1,0) {$\mathsmaller{-1}$};
         \node[above] at (0,0) {$\mathsmaller{0}$};
          \node[above] at (1,0) {$\mathsmaller{-1}$};
    \end{tikzpicture} .
    \end{aligned}
    \end{equation}
The three minimal cut paths in $w'$ corresponding to $a,b,c$ extend to $w$. Assume first that $\gamma'$ extends across the horizontal rung of $F=H$ in $w$ to a minimal cut path in $w$:
\begin{center}
     \begin{tikzpicture}
         \draw (0,0) -- (0,-1.5);
        \draw (1,0) -- (1,-.5); 
        \draw (0,-.5) -- (1,-.5);
        \draw (1,-.5) -- (1.5,-1);
        \draw (2,0) -- (1.5,-1);
        \draw (1.5,-1) -- (1.5,-1.5);
    \draw[densely dashed] (.5,0) -- (.5,-1.5);
    \node[above] at (.5,0) {$\mathsmaller{k+1}$};
     \node[above] at (-.75,0) {$\mathsmaller{a}$};
              \node[above] at (1.5,0) {$\mathsmaller{b}$};
                 \node[above] at (2.25,0) {$\mathsmaller{c}$};
    \end{tikzpicture}
\end{center}
Of the nine options in \eqref{eq:nine cases for Y}, the three in the last column (when $a=k-1$) are then impossible. Moreover, $(k,k,k-1)$ is also impossible by Lemma \ref{lem:no 1 -1 H}. This leaves five cases, and for each of them one checks that we can apply the growth algorithm once to produce $F=H$ and such that the resulting state string matches $\mcpalg(w')$. For instance, $(k+1, k, k-1)$ gives 
\begin{center}
    \begin{tikzpicture}
        \node at (0,0) {$\pm$};
        \node at (1,0) {$\mp$};
        \node at (2,0) {$\pm$};
        \node[above] at (0,.2) {$\mathsmaller{0}$};
        \node[above] at (1,.2) {$\mathsmaller{-1}$};
        \node[above] at (2,.2) {$\mathsmaller{-1}$};

        \draw[->] (3,0) -- (4,0) node[above,midway] {$\growalg$};
        \begin{scope}[shift={(5,.5)}]
 \draw (0,0) -- (0,-1);
 \draw (1,0) -- (1,-1);
 \draw (0,-.5) -- (1,-.5);
 \draw (2,0) -- (2,-1);
         \node[above] at (0,0) {$\mathsmaller{0}$};
        \node[above] at (1,0) {$\mathsmaller{-1}$};
        \node[above] at (2,0) {$\mathsmaller{-1}$};
 \node[below] at (0,-1) {$\mathsmaller{-1}$};
 \node[below] at (1,-1) {$\mathsmaller{0}$};
 \node[below] at (2,-1) {$\mathsmaller{-1}$};
        \end{scope}
    \end{tikzpicture}
\end{center}
where the bottom states match \eqref{eq:ex case for Y}. We leave the other four cases to the reader. 

Next, assume $\gamma'$ does not extend. Then in $w$ we have one of two options:
\begin{center}
 \begin{tikzpicture}
         \draw (0,0) -- (0,-1.5);
        \draw (1,0) -- (1,-.5); 
        \draw (0,-.5) -- (1,-.5);
        \draw (1,-.5) -- (1.5,-1);
        \draw (2,0) -- (1.5,-1);
        \draw (1.5,-1) -- (1.5,-1.5);
    \draw[densely dashed] (.5,0) -- (2.25,-.5);
    \node[above] at (.5,0) {$\mathsmaller{k}$};
     \node[above] at (-.75,0) {$\mathsmaller{a}$};
              \node[above] at (1.5,0) {$\mathsmaller{b}$};
                 \node[above] at (2.25,0) {$\mathsmaller{c}$};
                \node at (4,-.75) {or};
    \end{tikzpicture}\hskip3em
         \begin{tikzpicture}
         \draw (0,0) -- (0,-1.5);
        \draw (1,0) -- (1,-.5); 
        \draw (0,-.5) -- (1,-.5);
        \draw (1,-.5) -- (1.5,-1);
        \draw (2,0) -- (1.5,-1);
        \draw (1.5,-1) -- (1.5,-1.5);
    \draw[densely dashed] (.5,0) -- (-.5,-.5);
    \node[above] at (.5,0) {$\mathsmaller{k}$};
     \node[above] at (-.75,0) {$\mathsmaller{a}$};
              \node[above] at (1.5,0) {$\mathsmaller{b}$};
                 \node[above] at (2.25,0) {$\mathsmaller{c}$};
    \end{tikzpicture} 

\end{center}
The first option forces $(b,c) = (k-1, k-2)$, which is impossible, and the second option forces $a=k-1$, which gives three possible cases. However, note that $(b,c) = (k+1, k)$ contradicts Lemma \ref{lem:no 1 -1 H}, and $(b,c) = (k,k-1)$ implies that $\gamma'$ does extend by Lemma \ref{lem:adjacent cut paths}. 

This leaves only $(a,b,c)=(k-1,k+1, k+1)$. Let $w''$ denote $w'\setminus F'$. Note that $(\growalg \circ \mcpalg)(w'') = w''$ by inductive hypothesis. Applying one step of $\growalg$ to $\mcpalg(w')$ gives 
\begin{center}
    \begin{tikzpicture}
        \node at (0,0) {$\mp$};
        \node at (1,0) {$\pm$};
        \node at (2,0) {$\pm$};
        \node[above] at (0,.2) {$\mathsmaller{1}$};
        \node[above] at (1,.2) {$\mathsmaller{1}$};
        \node[above] at (2,.2) {$\mathsmaller{0}$};

        \draw[->] (3,0) -- (4,0) node[above,midway] {$\growalg$};
        \begin{scope}[shift={(5,.5)}]
        \draw (0,0) -- (0,-1);
        \draw (1,0) -- (1.5,-.5);
        \draw (2,0) -- (1.5,-.5);
        \draw (1.5,-.5) -- (1.5,-1);
             \node[above] at (0,0) {$\mathsmaller{1}$};
            \node[above] at (1,0) {$\mathsmaller{1}$};
            \node[above] at (2,0) {$\mathsmaller{0}$};
        \node[below] at (0,-1) {$\mathsmaller{1}$};
        \node[below] at (1.5,-1) {$\mathsmaller{1}$};
        \end{scope}
    \end{tikzpicture} .
\end{center}
  
In this case, applying two steps of the growth algorithm to $\mcpalg(w)$ yields
\begin{center}
    \begin{tikzpicture}[yscale=.6]
        \node at (0,0) {$\pm$};
        \node at (1,0) {$\mp$};
        \node at (2,0) {$\pm$};
        \node[above] at (0,0.2) {$\mathsmaller{1}$};
        \node[above] at (1,0.2) {$\mathsmaller{1}$};
        \node[above] at (2,0.2) {$\mathsmaller{0}$};

        \draw[->] (3,0) -- (4,0) node[above,midway] {$\growalg$};
        \begin{scope}[shift={(5,.5)}]
            \draw (0,0) -- (0,-1);
            \draw (1,0) -- (1,-1);
            \draw (1,-.5) -- (2,-.5);
            \draw (2,0) -- (2,-2);
            \node[above] at (0,0) {$\mathsmaller{1}$};
            \node[above] at (1,0) {$\mathsmaller{1}$};
            \node[above] at (2,0) {$\mathsmaller{0}$};
            \node[below,left] at (0,-1) {$\mathsmaller{1}$};
            \node[below,right] at (1,-1) {$\mathsmaller{0}$};
            \node[below] at (2,-2) {$\mathsmaller{1}$};
            \draw (0,-1) -- (.5,-1.5);
            \draw (1,-1) -- (.5,-1.5);
            \draw (.5,-1.5) -- (.5,-2);
            \node[below] at (.5,-2) {$\mathsmaller{1}$};
        \end{scope}
    \end{tikzpicture} .
\end{center}
We see that we arrive at a state string matching $\mcpalg(w'')$, allowing us to complete the proof by induction. 

Finally, we consider when $F'$ is a good $H$. Take a minimal cut path in $w'$ as shown below, and consider possibilities in the surrounding regions. 

\begin{center}
    \begin{tikzpicture}
    \draw (-1,0) -- (-1,-1);
        \draw(0,0) -- (0,-1);
        \draw (1,0) -- (1,-1);
        \draw (0,-.5) -- (1,-.5); 
        \draw[densely dashed] (-.5,0) -- (-.5,-1);
        \node[above] at (-.5,0) {$k$};
        \node[above] at (-1.5,0) {$a$};
        \node[above] at (.5,0) {$b$};
        \node[above] at (1.5,0) {$c$};
    \end{tikzpicture}
\end{center}
Since $F'$ is good, we have nine possibilities for $(a,b,c)$:
\begin{align}
\begin{aligned}
\label{eq:nine cases for H}
    (k+1, k+1, k+1), & & (k,k+1, k+1), &  & (k-1, k+1,k+1), \\
    (k+1, k, k-1), &  & (k,k, k-1), & & (k-1, k ,k-1), \\
    (k+1, k, k), & & (k,k, k), & & (k-1, k ,k).
    \end{aligned}
\end{align}
Assume first that the depicted cut path extends to $w$:
\begin{center}
       \begin{tikzpicture}
         \draw (0,0) -- (0,-1.5);
        \draw (1,0) -- (1,-1.5); 
        \draw (0,-.5) -- (1,-.5);
        \draw (2,0) -- (2,-1.5);
        \draw (1,-1) -- (2,-1);
         \draw[densely dashed] (.5,0) -- (.5,-1.5);
        \node[above] at (.5,0) {$k+1$};
    \end{tikzpicture}
\end{center}
Then $a=k-1$ is not possible. Moreover, $(a,b) = (k,k)$ contradicts Lemma \ref{lem:no 1 -1 H}, leaving four options. For each of these we see that $F$ can be produced by $\growalg$ and that the resulting states match those in $w'$. For instance, $(k, k+1 ,k+1)$ gives 
\begin{center}
    \begin{tikzpicture}
        \node at (0,0) {$\pm$};
        \node at (1,0) {$\mp$};
        \node at (2,0) {$\mp$};
        \node[above] at (0,.2) {$\mathsmaller{1}$};
        \node[above] at (1,.2) {$\mathsmaller{0}$};
        \node[above] at (2,.2) {$\mathsmaller{0}$};

        \draw[->] (3,0) -- (4,0) node[above,midway] {$\growalg$};
        \begin{scope}[shift={(5,.5)}]
            \draw (0,0) -- (0,-1);
            \draw (1,0) -- (1,-1);
            \draw (0,-.5) -- (1,-.5);
            \draw (2,0) -- (2,-1);
            \node[above] at (0,0) {$\mathsmaller{1}$};
            \node[above] at (1,0) {$\mathsmaller{0}$};
            \node[above] at (2,0) {$\mathsmaller{0}$};
            \node[below] at (0,-1) {$\mathsmaller{0}$};
            \node[below] at (1,-1) {$\mathsmaller{1}$};
            \node[below] at (2,-1) {$\mathsmaller{0}$};
        \end{scope}
    \end{tikzpicture} .
\end{center}
Suppose then that the cut path does not extend to $w$. In $w$ we have two possibilities:
\begin{center}
\begin{tikzpicture}
    \draw (0,0) -- (0,-1.5);
    \draw (1,0) -- (1,-1.5); 
    \draw (0,-.5) -- (1,-.5);
    \draw (2,0) -- (2,-1.5);
    \draw (1,-1) -- (2,-1);
    \draw[densely dashed] (.5,0) -- (1.5,-.5);
    \node[above] at (.5,0) {$k$};
    \node at (4,-.75) {or};
\end{tikzpicture}\hskip3em
\begin{tikzpicture}
    \draw (0,0) -- (0,-1.5);
    \draw (1,0) -- (1,-1.5); 
    \draw (0,-.5) -- (1,-.5);
    \draw (2,0) -- (2,-1.5);
    \draw (1,-1) -- (2,-1);
    \draw[densely dashed] (.5,0) -- (-.5,-.5);
    \node[above] at (.5,0) {$k$};
\end{tikzpicture} 
\end{center}
The first option forces $b=k-1$ which is not possible, and the second forces $a=k-1$, which gives three options. Of these three, the two cases where $b=k$ imply that the path does extend by Lemma \ref{lem:adjacent cut paths}, leaving only $(a,b,c) = (k-1,k+1, k+1)$. This means we can arrange two minimal cut paths in $w$ to look like 
\begin{center}
\begin{tikzpicture}
    \draw (0,0) -- (0,-1.5);
    \draw (1,0) -- (1,-1.5); 
    \draw (0,-.5) -- (1,-.5);
    \draw (2,0) -- (2,-1.5);
    \draw (1,-1) -- (2,-1);
    \draw[densely dashed] (1.25,0) .. controls (1.25, -.25) .. (-.5,-.25);
    \draw[densely dashed] (1.75,0) -- (1.75,-1.5);
\end{tikzpicture} .
\end{center}
However, Lemma \ref{lem:no far apart cut paths} rules out this possibility. 
\end{proof}

%%%%%%%%%%%%%%%%%%%%%%%%%%%%%%
\subsection{Counting paths}

Given a sign string $S$, let $M^S$ be the free $\Z[q,q^{-1}]$-module generated by all  annular $SL(3)$ webs with boundary $S$, modulo ambient planar isotopy and the local relations show in Figure \ref{fig:sl3 relations}. Recall that $B^S$ denotes a set of representatives for planar isotopy classes of non-elliptic annular $SL(3)$ webs with boundary $S$.  Any annular web $w$ with boundary $S$ may be reduced to a $\Z[q,q^{-1}]$-linear combination of elements of $B^S$ by applying the local relations. The argument in \cite{KuperbergQuantumG2} showing that the resulting linear combination is well-defined (independent of choices of simplification) applies here without changes. Therefore $B^S$ forms a basis for $M^S$.

\begin{figure}
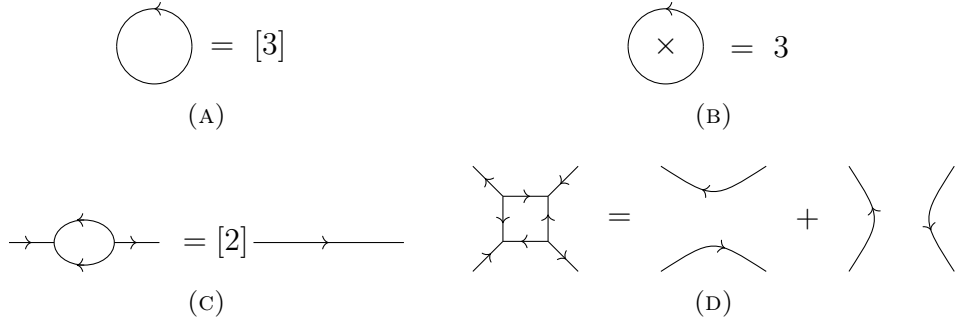

\centering 
\subcaptionbox{ \label{fig:circle relation}}[.4\linewidth]
{\includestandalone{images/circle_relation}
}
\subcaptionbox{\label{fig:essential circle relation}}[.4\linewidth]
{\includestandalone{images/essential_circle_relation}
}
\vskip1em
\subcaptionbox{\label{fig:bigon relation}}[.4\linewidth]
{\includestandalone{images/bigon_relation}
}
\subcaptionbox{\label{fig:square relation}}[.4\linewidth]
{\includestandalone{images/square_relation}}
\caption{The local $SL(3)$ web relations. Each relation also holds if all orientations are reversed. Here, $[3] = q^2 + 1 + q^{-2}$ and $[2] = q+q^{-1}$ are quantum integers.}\label{fig:sl3 relations}
\end{figure}

If $-S$ is the sign sequence obtained by negating all entries of $S$, then there is an evident bijection $B^S \cong B^{-S}$ given by reversing the orientation of each web.

\begin{lemma}
\label{lem:switching signs}
    Suppose that two consecutive entries in $S$ are of opposite sign, and let $S'$ be the sign string obtained from $S$ by switching these two signs. That is, 
    \begin{align*}
        S &= (s_1, \ldots, s_{i-1}, \pm, \mp, s_{i+2}, \ldots, s_n),\\
        S' &=  (s_1, \ldots, s_{i-1}, \mp, \pm, s_{i+2}, \ldots, s_n).
    \end{align*}
    Then there is a bijection $B^S \cong B^{S'}$. 
\end{lemma}

\begin{proof}
 There is an evident bijection between the closed $S$-paths and the closed $S'$-paths, so this follows from Theorem \ref{thm:MCP-GROW-inverses}. 
\end{proof}

\begin{remark}
 Alternatively, suppose $S'$ is obtained from $S$ by switching consecutive opposite signs $+-$ to $-+$. Define $\beta_{+-}: M^S \to M^{S'}$ and $\beta_{-+} : M^{S'} \to M^S$ as in Figure \ref{fig:braiding}. It is straightforward to check using the relations in Figure \ref{fig:sl3 relations} that $\beta_{+-}$ and $\beta_{-+}$ are mutually inverse isomorphisms. In fact, $\beta_{+-}$ and $\beta_{-+}$ are obtained by resolving the crossings 
 \begin{center}
     \begin{tikzpicture}[decoration={
    markings,
    mark=at position 0.85 with {\arrow{Stealth}}}]
        \draw[postaction={decorate}] (0,0) -- (1,1); 
        \draw[postaction={decorate},preaction={draw, line width=12pt, white}]  (0,1) -- (1,0);
        \node at (2,.5) {and};
        \draw[postaction={decorate}] (4,0) -- (3,1);
        \draw[postaction={decorate},preaction={draw, line width=12pt, white}]  (4,1) -- (3,0);
     \end{tikzpicture}
 \end{center}
according to \cite[Figure 10]{sl3-link-homology}. That $\beta_{+-}$ and $\beta_{-+}$ are mutually inverse isomorphisms is then analogous to the proof of invariance of the $SL(3)$ link polynomial under Reidemeister II moves. Since  $M^S$ (resp. $M^{S'}$) is free over $\Z[q,q^{-1}]$ with basis $B^S$ (resp. $B^{S'}$), it follows that there is a bijection $B^S \cong B^{S'}$. 
\end{remark}

\begin{figure}
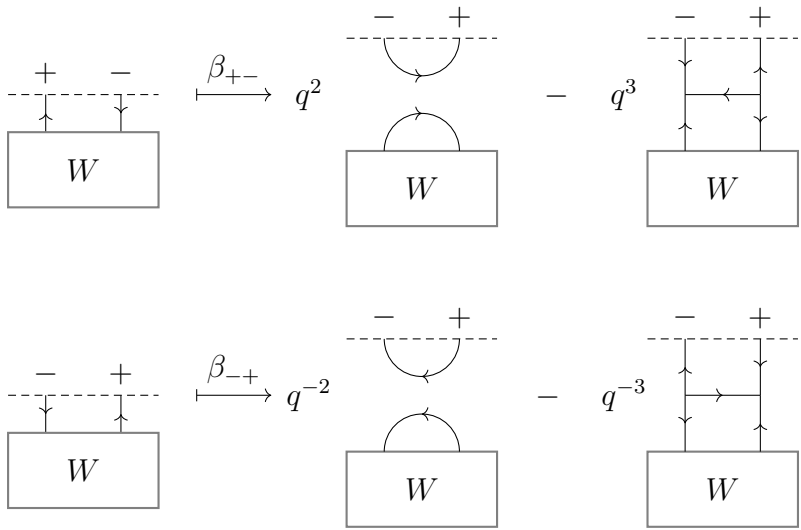

     \centering
     \includestandalone{images/braiding}
     \caption{Mutually inverse isomorphisms $\beta_{+-}: M^S \to M^{S'}$ and $\beta_{-+} : M^{S'} \to M^S$.}
     \label{fig:braiding}
\end{figure}

We are now ready to count the number of admissible pairs $(S,J)$ for fixed $S$, by counting paths in the weight lattice. Recall that 
\[
    \binom{\ell}{i,j,k} = \frac{\ell!}{i! j! k!}
\]
denotes the trinomial coefficient.

\begin{proposition}
\label{prop:count-paths}
    Given a fixed sign sequence $S$ with $n_+$ pluses and $n_-$ minuses, the number of admissible pairs $(S,J)$ is 
    \begin{multline}
    \label{eq:count-SJ-fixed-S}
        \sum_{(a_1,a_2) \in \Hull \{(0,0), (n_+,0), (0, n_+)\}}
        \binom{n_+ }{{a_1, \  a_2, \  n_+ - a_1 -a_2}}
        \\
        \cdot 
        \binom{n_-
        }{
        a_1 - \frac{1}{3} (n_+-n_-), \  
        a_2 - \frac{1}{3} (n_+-n_-), \  
        n_+ - a_1 -a_2 - \frac{1}{3} (n_+ - n_-)
        }.
    \end{multline}
\end{proposition}

\begin{proof}
By Lemma \ref{lem:switching signs},  we may assume that $S$ is of the form
\[
    (\underbrace{+,+\ldots,+}_{n_+}, \underbrace{-,-,\ldots,-}_{n_-}).
\] 
We will use more convenient notation for the weights $\Lambda^+$ and $\Lambda^-$ of $V^+$ and $V^-$ (compare with Figure \ref{fig:weight lattice}):

\begin{center}
\begin{tikzpicture}[scale=.8]
    % fund rep
    \draw[thick, ->] (0,0) -- +(30:2cm) 
        node[label={30:$\mu_1$}]  {};
    \draw[thick, ->] (0,0) -- +(150:2cm)
        node[label={150:$\mu_2$}]  {};
    \draw[thick, ->] (0,0) -- +(270:2cm)
        node[label={270:$\mu_3 := -\mu_1-\mu_2$}]  {};
    % dual rep
    \draw[thick, ->, dashed] (0,0) -- +(90:2cm) 
        node[label={90:$\nu_3 := -\nu_1-\nu_2$}]  {};
    \draw[thick, ->, dashed] (0,0) -- +(210:2cm)
        node[label={210:$\nu_1$}]  {};
    \draw[thick, ->, dashed] (0,0) -- +(330:2cm)
        node[label={330:$\nu_2$}]  {};
\end{tikzpicture}
\end{center}
Given an $S$-path in $\Lambda$, we may write it as a word in these six letters. Precisely, $\gamma = \gamma_\mu \gamma_\nu$ where $\gamma_\mu \in \{\mu_1, \mu_2, \mu_3\}^{n_+}$ and  
$\gamma_\nu \in \{\nu_1, \nu_2, \nu_3\}^{n_-}$.

Let $\origin$ denote the origin of $\Lambda$.
We first focus on $\gamma_\mu$, viewed as a path from $\origin$ to $P$, where $P\in \Lambda$ is determined as follows. For $1\leq i \leq 3$, let $a_i$ denote the number of $\mu_i$'s in $\gamma_\mu$. Then 
\[
    P=a_1 \mu_1 + a_2 \mu_2 + a_3 \mu_3
    = (a_1 - a_3) \mu_1 + (a_2 - a_3) \mu_2
\]
is the terminus of the path $\gamma_\mu$. 
There is a two-parameter family of triples $(a_1, a_2, a_3 = n_+ - a_1 - a_2)$ such that $P$ is reachable by a path of length $n_+$ given by words in the letters $\{\mu_1, \mu_2, \mu_3\}$. 
The constraints are that $0 \leq a_1, a_2$ and $a_1+a_2 \leq n_+$. 
Stated differently, 
$(a_1,a_2) \in \Hull \{(0,0), (n_+,0), (0, n_+)\} \subset \Z^2$.
Given such a point $P$, the number of paths $\gamma_\mu$ with terminus $P$ is counted by the trinomial coefficient 
\[
    \binom{n_+}{a_1,\  a_2,\  a_3} 
    = 
    \binom{n_+}{a_1, \  a_2, \  n_+ - a_1 -a_2}.
\]

We now count the paths $\gamma_\nu$ from  $P$ to $\origin$ of length $n_-$ corresponding to words in $\{\nu_1, \nu_2, \nu_3\}$. 
Given such a path, let $b_i$ denote the number of $\nu_i$'s in $\gamma_\nu$. 
We have the following constraints:
\begin{enumerate}
    \item $a_3 = n_+ - a_1 - a_2$
    \item $b_3 = n_- - b_1 -b_2$
    \item $a_1 - a_3 = b_1 - b_3$
    \item $a_2 - a_3 = b_2 - b_3$. 
\end{enumerate}
The last two constraints come from first writing $P$ in terms of the basis $\{\mu_1, \mu_2\}$ and then recalling that $\nu_i = -\mu_i$ for each $i \in \{1,2,3\}$. By substitution, these constraints reduce to the following:
\begin{enumerate}
    \item[(3')] $2a_1 - a_2 - n_+ = 2b_1 -b_2 - n_-$
    \item[(4')] $2a_2 - a_1 - n_+ = 2b_2 -b_1 - n_-$
\end{enumerate}
Solving this system, we find that
\[
    b_i = a_i - \frac{1}{3}(n_+ - n_-)   
    \qquad 
    \text{for } i = 1,2,3.
\]
In summary, the total number of admissible $(S,J)$ pairs is equal to 
\begin{equation}
    \label{eq:intuitive-count-SJ-fixed-S}
        \sum_{(a_1,a_2) \in \Hull \{(0,0), (n_+,0), (0, n_+)\}}
        \binom{n_+}{a_1, \  a_2, \  a_3}
        \\
        \cdot 
        \binom{n_-}{b_1, \  b_2, \  b_3}
\end{equation}
with the values $a_3, b_1, b_2, b_3$ determined by $a_1, a_2$ as above.
\end{proof}

\begin{example}
    As a first family of examples we may consider sign sequences $S_{3m}^+$ consisting of only $n_+ = 3m$ pluses ($n_- = 0$), and consider the number of admissible pairs $(S_{3m}^+,J)$ as a function of $m$.
The second trinomial coefficient in  \eqref{eq:count-SJ-fixed-S} 
reduces to 
\[
    \binom{0}{
        a_1 - m,\ 
        a_2 - m,\ 
        a_3 - m},
\]
indicating that the word in the $\mu_i$ must return to the origin (i.e.\ $P = \origin$, so $a_1 = a_2 = a_3 = m$). 
Therefore for $S = (\underbrace{+,+\ldots, +}_{3m})$, the number of admissible pairs $(S,J)$ is given by 
\[
    \binom{3m}{m,\ m,\ m},
\]
the only surviving term in \eqref{eq:count-SJ-fixed-S}.
\end{example}

\begin{example}
    When $S$ has $n_+  = n_-=n$,  the number of paths is given by the integer sequence \url{https://oeis.org/A002893} and \eqref{eq_found_seq}, as follows. Using Lemma \ref{lem:switching signs} we may assume that signs alternate in $S$, so that $S = (+,-,+,\ldots, +,-)$. In this case a closed $S$-path is the same as a path  of length $2n$ in the honeycomb (hexagonal) lattice starting and ending at the origin. The number of such paths is given by 
\begin{equation}\label{eq_found_seq}
    \sum_{k=0}^n \binom{2k}{k} \binom{n}{k}^2.
\end{equation}
See, for instance, \cite[Lemma 2.1]{Honeycomb}. The above expression is easily seen to agree with Proposition \ref{prop:count-paths} by writing \eqref{eq:count-SJ-fixed-S} as a sum over diagonals, 
\[
    \sum_{k=0}^n \sum_{a = 0}^k \binom{n}{a,k-a,n-k}^2,
\]
and observing that 
\[
    \sum_{a = 0}^k \binom{n}{a,k-a,n-k}^2 =  \binom{n}{k}^2 \sum_{a=0}^k \binom{k}{a}^2 =  \binom{n}{k}^2 \binom{2k}{k}.
\]
\end{example}

\subsection{Annular \texorpdfstring{$SL(3)$}{SL(3)} web algebras}

Given sign strings $R$ and $S$, an \emph{$(R,S)$-web} is an $SL(3)$ web $w\subset \A$ such that the orientation of $w$ at the boundary component $\partial_1 \A = \SS^1\times \{1\}$ (resp. $\partial_0\A = \SS^1\times \{0\}$) is given by $R$ (resp. $S$), according to the convention in Figure \ref{fig:edge orientation at boundary}. Two $(R,S)$ webs $w_1$ and $w_2$ are isotopic if there is an orientation-preserving diffeomorphism of $\A$ which fixes $\partial \A$ pointwise, is isotopic
to the identity, and takes $w_1$ to $w_2$.

For an $(R,S)$-web $w$, let $\b{w}$ denote the $(S,R)$-web given by reflecting $w$ across $\SS^1\times \{1/2\}$ and reversing all orientations. A \emph{closed} web is a $(\varnothing, \varnothing)$-web. If $Q$ is another sign string, $v$ is a $(Q,R)$-web, and $w$ is an $(R,S)$-web, then we let $vw$ denote the  $(Q,S)$-web obtained by stacking $u$ on top of $w$. 

The construction in \cite{AK}  gives a functorial assignment of modules to closed webs, which we now describe. In this section, let 
\[
R_\alpha = \Z[\alpha_1, \alpha_2, \alpha_3]
\]
be the graded ring of polynomials in three variables with $\deg(\alpha_i) = 2$ for $i=1,2,3$. Consider also the abelian group 
\[
\Lambda = \Z c_1 \oplus \Z c_2 \oplus \Z c_2 /(c_1 + c_2 + c_3)
\]
with three generators and one relation. Note that $\Lambda$ is a free abelian group of rank two. Let $R_\alpha \gggmod$ be the category of $\Z\oplus \Lambda$-graded $R_\alpha$-modules. The $\Z$-grading is the quantum grading $\qdeg$ and the $\Lambda$-grading is the \emph{annular} grading $\adeg$. Variables $\alpha_i$ increase the quantum grading by two and preserve the annular grading. 

Below we give a brief summary of anchored $SL(3)$ foams. We refer the reader to \cite[Definition 4.1, Definition 4.9]{AK} for more details.
\begin{definition}
    An \emph{anchored $SL(3)$ foam} is a 2-dimensional CW complex $F$ that is PL embedded into $\R^3$. Every point of $F$ has a neighborhood homeomorphic to either an open disk or $T\times (0,1)$, where $T$ is the \emph{tripod}, obtained by gluing three copies of $[0,1)$ along $\{0\}$. Points of the second type are called \emph{seam points}, and the complement of the seam points is a disjoint union of orientable surfaces called \emph{facets}. Facets of $F$ may be decorated by some finite number of dots. Moreover, $F$ intersects the anchor line $\line$ transversely in the interior of its facets, and each such intersection point carries a label in $\{1,2,3\}$.  We also consider anchored $SL(3)$ foams \emph{with boundary}, viewed as cobordisms between annular webs. 
\end{definition}

Let $\AFoam$ be the category whose objects are closed annular $SL(3)$ webs and whose morphisms are linear combinations of anchored $SL(3)$ foam cobordisms, viewed up to ambient isotopy fixing $\line$ setwise.  Morphism spaces in $\AFoam$ are $\Z\oplus \Lambda$-graded (see \cite[Equations (89) and (91)]{AK}). 
The construction in \cite[Section 4]{AK} gives a functor
\[
    \brak{-} : \AFoam \to R_\alpha \gggmod.
\]
By \cite[Theorem 4.24]{AK}, every closed annular web $w$ is assigned a free finitely-generated $\Z\oplus \Lambda$-graded $R_\alpha$-module $\brak{w}$. The functor $\brak{-}$ satisfies categorified versions of the relations in Figure \ref{fig:sl3 relations}, so its graded rank $\rk_q(\brak{w})$ (with respect to the quantum grading) can be recursively computed by using these relations. 

For a sign string $S$, recall that $B^S$ denotes a set of representatives for planar isotopy classes of non-elliptic annular $SL(3)$ webs $w$ whose boundary signs are equal to $S$. Let $\ell(S)$ denote the length of $S$. We  view elements of $B^S$ as $(S, \varnothing)$-webs. If $v,w \in B^S$, then $\b{v} w$ is a closed web, so we may consider its associated $R_\alpha$-module $\brak{\b{v} w}$. 

\begin{definition}
    Given a sign string $S$, the \emph{annular $SL(3)$ web algebra of $S$} is 
    \[
    H^S_\A = \displaystyle \bigoplus_{v,w\in B^S} \brak{\b{v} w} \{\ell(S)\}.
    \]
    Multiplication on the direct summand $\brak{\b{v} w} \o_{R_\alpha} \brak{\b{x} y} $ of $H^S_\A \o_{R_\alpha} H^S_\A$ is defined to be $0$ unless $w= x$, in which case it is induced by the  foam that couples $\b{w}$ to $w$ through the standard cobordism to the identity web $\id_S$.
    
    In analogy with the $SL(2)$ setting, there are idempotents $1_w \in \brak{\b{w} w} \{n\}$ for each   $w\in B^S$, and the unit in $H^S_\A$ is given by $\sum_{w\in B^S} 1_w$. 
\end{definition} 

Multiplication in $H^S_\A$ is degree-preserving with respect to both the quantum and annular gradings. For the quantum grading, this follows from computing degrees of the relevant cobordisms (see, for instance \cite[Proposition 3.6]{MPTwebalgebra}) and the degree shift $\{\ell(S)\}$.

Theorem \ref{thm:MCP-GROW-inverses} implies that $B^S$ is finite, with cardinality equal to the number of closed $S$-paths. It follows that $H^S_\A$ is a free, finitely generated $\Z \oplus \Lambda$-graded $R_\alpha$-module. 

\subsection{Annular \texorpdfstring{$SL(3)$}{SL(3)} bimodules} 

Given $w\in B^S$, we have left and right $H^S_\A$-modules 
\[
    P_w := \bigoplus_{v\in B^S} \brak{\b{v} w} \{\ell(S)\} \hskip2em \text{ and } \hskip2em _w P := \bigoplus_{v\in B^s} \brak{\b{w} v} \{\ell(S)\} 
\]
respectively, which are each projective since $H^S_\A = \bigoplus_{w\in B^S} P_w$ as a left $H^S_\A$-module and $H^S_\A = \bigoplus_{w \in B^S} {_w P}$ as a right $H^S_\A$-module.

In contrast with the $SL(2)$ setting, the web module $P_w$ may be decomposable. 
For planar web algebras this behavior was first observed by Morrison-Nieh \cite{MorrisonNieh}.
This was further studied extensively by Robert \cite{Robert-thesis} who characterized precisely which webs yield indecomposable modules. 

\begin{lemma}
\label{lem:indecomposability condition}
    Let $w\in B^S$. If $\rk_q(\brak{\b{w}w})$ is monic of degree $\ell(S)$, then $P_w$ is indecomposable. 
\end{lemma}
\begin{proof}
    The same argument as in the proof of \cite[Lemma 2.2.22]{Robert-thesis} applies.
\end{proof}

In the planar setting, the smallest web that yields a decomposable module (over the planar $SL(3)$ algebra as defined in \cite{Robert-thesis,MPTwebalgebra}) has 12 boundary points; see \cite[Theorem 5.3]{MorrisonNieh} and \cite[Proposition 3.1.1]{Robert-thesis}. In the annular setting we can find a much smaller example. 

\begin{proposition}
\label{prop:decomposable web}
    Consider the webs $w_0$ and $w$ shown below.
    \begin{center}
        \begin{tikzpicture}[decoration={markings,mark=at position .8 with
    {\arrow[scale=1,>=latex']{>}}}]
\node at (0,0) {$\times$};
\draw (-1.4,-.25) arc (-90:90:.25 and .25);
        \draw (1.4,-.25) arc (-90:-270:.25 and .25);

\draw[->,>=latex',dashed,line width=.4pt]
(-1.15,0) --++ (0,.05);
\draw[->,>=latex',dashed,line width=.4pt]
(1.15, 0) --++ (0,.05);

%%%%%%%%%%%%%%%%%%%%%%
\begin{scope}[shift={(5,0)}]
    \node at (0,0) {$\times$};
    \draw (0,0) circle[radius=.25];
 
    \draw[->,>=latex',dashed,line width=.4pt]
    (-.25,0) --++ (0,-.05);
    \draw[->,>=latex',dashed,line width=.4pt]
    (.25,0) --++ (0,-.05);
    
    \draw[postaction=decorate] (-1.2,.25) -- (-1.4,.25);
    \draw[postaction=decorate] (-1.4,-.25) -- (-1.2,.-.25);
    \draw[postaction=decorate] (-1.2,.25) -- (-1.2,-.25);
    
    \draw[postaction=decorate] (1.2,.25) -- (1.4,.25);
    \draw[postaction=decorate] (1.4,-.25) -- (1.2,-.25);
    \draw[postaction=decorate] (1.2,.25) -- (1.2,-.25);

    \draw[postaction=decorate] (-1.2,.25) -- (0,.75);
    \draw[postaction=decorate] (0,-.75) -- (-1.2,-.25);

    \draw[postaction=decorate] (1.2,.25) -- (0,.75);
    \draw[postaction=decorate] (0,-.75) -- (1.2,-.25);

    \draw[postaction=decorate] (0,.25) -- (0,.75);
    \draw[postaction=decorate] (0,-.75) -- (0,-.25);
\end{scope}
\node at (0,-1) {$w_0$};
\node at (5,-1) {$w$};
\end{tikzpicture}
\end{center}
After extending from $\Z$ to $\Q$ (or to any commutative domain where $2$ is invertible), the projective $H^{(+,-,-,+)}_\A$ module $P_w$ is decomposable and has $P_{w_0}$ as a direct summand. 
\end{proposition}
In the proposition, the sign sequence $(+,-,-,+)$ can be cyclically rotated; that only changes the starting point on the boundary circle from which we read the sequence. 
\begin{proof}
    Let $F : w_0 \to w$ be the foam cobordism (with corners) shown below.
\begin{center}
   \includestandalone{images/degree_zero_foam}
\end{center}
In words, $F$ is formed from two cups and four singular saddles (zips). Note that $F$ is a degree-zero foam. Let $\b{F}: w \to w_0$ be its horizontal reflection. We claim that $\b{F}F = 2 \id_{w_0}$.

We will use the following planar diagram shorthand for foams. 

\begin{center}
    \begin{tikzpicture}[decoration={
    markings,
    mark=at position 0.55 with {\arrow{Stealth}}}]
        \draw (0,0) -- (2,0);
        \draw (0,-1) -- (2,-1);
        \draw[postaction=decorate] (1,-1) -- (1,0);

        \draw (3,0) -- (5,0);
        \draw (3,-1) -- (5,-1);

        \draw   (3,-.5) ellipse (.25 and .5);
        
        \draw[postaction={decorate},fill=gray, fill opacity=.5] (4,-1) arc (-90:90:.25 and .5);
        \draw[densely dashed, fill=gray, fill opacity=.5] (4,0) arc (90:270:.25 and .5);

        \draw (5,-1) arc (-90:90:.25 and .5);
        \draw[densely dashed] (5,0) arc (90:270:.25 and .5);

        \node at (2.35,-.5) {$:=$};
    \end{tikzpicture}
    \hskip4em
    \begin{tikzpicture}
        \draw (0,0) arc (90:-90:.5 and .5);
        
        \draw (1.5,-.5) ellipse (.25 and .5);
        \draw (1.5,0) arc (90:-90:.75 and .5);

        \node at (.85,-.5) {$:=$};
    \end{tikzpicture}
    \hskip4em
    \begin{tikzpicture}
        \draw (0,0) -- (1,0);
        \draw (0,-1) -- (1,-1);
\node at (.5,-.5) {$\bullet$};

\draw (2,0) --(3,0);
\draw (2,-1) -- (3,-1);
        \draw (2,-.5) ellipse (.25 and .5);
          \draw (3,-1) arc (-90:90:.25 and .5);
        \draw[densely dashed] (3,0) arc (90:270:.25 and .5);
        \node at (2.5,-.5) {$\bullet$};
        \node at (1.35,-.5) {$:=$};
    \end{tikzpicture}
\end{center}
The following identities hold:  
\begin{center}
\begin{tikzpicture}[decoration={
    markings,
    mark=at position 0.55 with {\arrow{Stealth}}}]
    \draw (0,0) -- (.5,0);
    \draw (0,-1) -- (.5,-1);
    \draw (.5,0) arc (90:-90:.5 and .5);
    \draw[postaction=decorate] (.5,0) -- (.5,-1);
    \node at (.75,-.5) {$\bullet$};

%%%%%%%%
    \begin{scope}[shift={(2.25,0)}]
         \draw (0,0) -- (.5,0);
    \draw (0,-1) -- (.5,-1);
    \draw (.5,0) arc (90:-90:.5 and .5);
    \draw[postaction=decorate] (.5,-1) -- (.5,0);
    \node at (.75,-.5) {$\bullet$};
\node at (-.25,-.5) {$-$};
    \end{scope}
    %%%%%%%%
    \begin{scope}[shift={(4.25,0)}]
         \draw (0,0) -- (.5,0);
    \draw (0,-1) -- (.5,-1);
    \draw (.5,0) arc (90:-90:.5 and .5);
    \end{scope}
%%%%%%%%%
\node at (1.5,-.5) {$=$};
\node at (3.75,-.5) {$=$};
\end{tikzpicture},
\hskip2em
\begin{tikzpicture}[decoration={
    markings,
    mark=at position 0.55 with {\arrow{Stealth}}}]
    \draw (0,0) -- (.5,0);
    \draw (0,-1) -- (.5,-1);
    \draw (.5,0) arc (90:-90:.5 and .5);
    \draw[postaction=decorate] (.5,0) -- (.5,-1);   
    \node at (1.6,-.5) {$=\ 0 $};
\end{tikzpicture},
\hskip1.8em
        \begin{tikzpicture}[decoration={
    markings,
    mark=at position 0.55 with {\arrow{Stealth}}}]

        \draw (0,0) -- (1.5,0);
        \draw (0,-1) -- (1.5,-1);
        \draw[postaction=decorate] (.75,-1) -- (.75,0);
        \node at (2,-.5) {$=$};

        \draw (2.5,0) arc (90:-90:.5 and .5); 
        \draw (4,0) arc (90:270:.5 and .5); 
        \node at (3.75,-.5) {$\bullet$};

 \draw (5,0) arc (90:-90:.5 and .5); 
        \draw (6.5,0) arc (90:270:.5 and .5); 
        \node at (4.5,-.5) {$-$};
        \node at (5.25,-.5) {$\bullet$};
    \end{tikzpicture} .
\end{center}
For the first two, see \cite[Lemma 4.18]{AK}, and for the third see relation (RD) in \cite[Lemma 2.3]{MVuniversal} as well as \cite[Theorem 4.26]{AK} for a discussion about the relationship between the $SL(3)$ foam evaluation defined in \cite{AK} and \cite{MVuniversal}. 
Then $\b{F}F$ is given by 
\begin{center}
    \begin{tikzpicture}[xscale=.95, decoration={
    markings,
    mark=at position 0.55 with {\arrow{Stealth}}}]
        %%%%%%%%
            \draw(0,0) -- (2,0);
        \draw (0,-1) -- (2,-1);
        \draw[postaction=decorate] (.25,0) -- (.25,-1);
        \draw[postaction=decorate] (1,-1) -- (1,0);
        \draw[postaction=decorate] (1.75,0) -- (1.75,-1);
        \node at (1.33,-.5) {$\bullet$};
    
         %%%%%%%%
        \begin{scope}[shift={(3,0)}]
            \draw(0,0) -- (2,0);
        \draw (0,-1) -- (2,-1);
        \draw[postaction=decorate] (.25,0) -- (.25,-1);
        \draw[postaction=decorate] (1,-1) -- (1,0);
        \draw[postaction=decorate] (1.75,0) -- (1.75,-1);
        \node at (.66,-.5) {$\bullet$};
        \end{scope}
        
        %%%%%%%%%%%%%%%%%%%%%%%%%%%%%%%%%%%%%
        \begin{scope}[shift={(6,0)}]
             \draw (0,0) -- (.25,0);
    \draw (0,-1) -- (.25,-1);
    \draw (.25,0) arc (90:-90:.5 and .5);
    \draw[postaction=decorate] (.25,0) -- (.25,-1);

    \draw (2,0) -- (1.75,0);
    \draw (2,-1) -- (1.75,-1);
    \draw (1.75,0) arc (90:270:.5 and .5);
    \draw[postaction=decorate] (1.75,0) -- (1.75,-1);
    \node at (1.5,-.35) {$\bullet$};
    \node at (1.5,-.65) {$\bullet$};
        \end{scope}

        %%%%%%%%%%%%%%%%%%
         \begin{scope}[shift={(9,0)}]
             \draw (0,0) -- (.25,0);
    \draw (0,-1) -- (.25,-1);
    \draw (.25,0) arc (90:-90:.5 and .5);
    \draw[postaction=decorate] (.25,0) -- (.25,-1);
    \node at (.5,-.5) {$\bullet$};

    \draw (2,0) -- (1.75,0);
    \draw (2,-1) -- (1.75,-1);
    \draw (1.75,0) arc (90:270:.5 and .5);
    \draw[postaction=decorate] (1.75,0) -- (1.75,-1);
    \node at (1.5,-.5) {$\bullet$};
   
        \end{scope}
            %%%%%%%%%%%%%%%%%%
         \begin{scope}[shift={(12,0)}]
             \draw (0,0) -- (.25,0);
    \draw (0,-1) -- (.25,-1);
    \draw (.25,0) arc (90:-90:.5 and .5);
    \draw[postaction=decorate] (.25,0) -- (.25,-1);
    \node at (.5,-.5) {$\bullet$};

    \draw (2,0) -- (1.75,0);
    \draw (2,-1) -- (1.75,-1);
    \draw (1.75,0) arc (90:270:.5 and .5);
    \draw[postaction=decorate] (1.75,0) -- (1.75,-1);
    \node at (1.5,-.5) {$\bullet$};
        \end{scope}
            %%%%%%%%%%%%%%%%%%
         \begin{scope}[shift={(15,0)}]
             \draw (0,0) -- (.25,0);
    \draw (0,-1) -- (.25,-1);
    \draw (.25,0) arc (90:-90:.5 and .5);
    \draw[postaction=decorate] (.25,0) -- (.25,-1);
    \node at (.5,-.35) {$\bullet$};
     \node at (.5,-.65) {$\bullet$};

    \draw (2,0) -- (1.75,0);
    \draw (2,-1) -- (1.75,-1);
    \draw (1.75,0) arc (90:270:.5 and .5);
    \draw[postaction=decorate] (1.75,0) -- (1.75,-1);
   
        \end{scope}

\begin{scope}[shift={(6.25,-1.5)}]

            \draw (0,0) -- (.25,0);
    \draw (0,-1) -- (.25,-1);
    \draw (.25,0) arc (90:-90:.5 and .5);
  
    \draw (2,0) -- (1.75,0);
    \draw (2,-1) -- (1.75,-1);
    \draw (1.75,0) arc (90:270:.5 and .5);

\end{scope}
        \node at (2.5,-.5) {$-$};
    \node at (5.5,-.5) {$=$};
    \node at (8.5,-.5) {$-$};
    \node at (11.5,-.5) {$-$};
    \node at (14.5,-.5) {$+$};
    \node at (5.5,-2) {$=$};
    \node at (6,-2
    ) {$2$};
    \end{tikzpicture}
\end{center}
which shows $\b{F}F = 2 \id_{w_0}$.

It is then straightforward to see that $e : = \frac{1}{2}(F\circ \b{F})$ is a nonzero idempotent endomorphism of $w$. Moreover, $e$ is not the identity since $P_w$ and $P_{w_0}$ have different graded ranks, which completes the proof.
\end{proof}

\begin{remark}
  If $\l(S)\leq 3$ then there are essentially two cases, $S=(+,-)$ and $S=(+,+,+)$, which have $3$ and $6$ non-elliptic webs respectively. Using Lemma \ref{lem:indecomposability condition} one can check by hand that $P_w$ is indecomposable whenever $w\in B^S$ for $\ell(S) \leq 3$. 
\end{remark}

\begin{remark}
    Robert \cite{Robert-thesis} proved that the converse of Lemma \ref{lem:indecomposability condition} holds in the planar setting. Moreover, it is shown that if $w\in \DD^2$ is a non-elliptic web and all faces of $w$ are adjacent to an outer face, then $P_w$ is an indecomposable module \cite[Theorem 3.2.4]{Robert-thesis}. Note that the web $w$ in Proposition \ref{prop:decomposable web} does not satisfy this condition. We hope to explore this in future work.  
\end{remark}

We now discuss bimodules assigned to (flat) annular webs. Let $R$ and $S$ be sign strings, and let $t$ be an $(R,S)$-web. Define an $(H^R_\A, H^S_\A)$-bimodule $\F(t)$ as follows. As an $R_\alpha$-module, we have 
\[
\F(t) = \bigoplus_{v\in B^R, u \in B^S} \brak{\b{v} t u}\{\ell(S)\}.
\]
The left and right actions are given by minimal foam cobordisms as usual. (Note that $R_{\alpha}$ denotes the ground ring and $R$ a sign string.)  

We have a decomposition 
\[
\F(t) = \bigoplus_{u\in B^S} \F(tu) \{\ell(S)\}
\]
of $\F(t)$ as a left $H^R_\A$-module, where we view $tu$ as an $(R,\varnothing)$-web. Reducing $tu$ to a sum of non-elliptic webs via the relations in Figure \ref{fig:sl3 relations}, we can write $\F(tu)$ as a direct sum (with grading shifts) of copies of projective modules $P_{u'}$ over some subset of $B^S$.  It follows that $\F(t)$ is a projective left $H^R_\A$-module. Similarly, $\F(t)$ is a projective right $H^S_\A$-module. 

\begin{proposition}
Let $Q, R$, and $S$ be sign strings.    If $t_0$ is an $(Q,R)$-web and $t_1$ is a $(R,S)$-web then there is a grading-preserving isomorphism of $(H^Q_\A, H^S_\A)$-bimodules
    \[
\F(t_0 t_1) \cong \F(t_0) \o_{H^R_\A} \F(t_1).
    \]
\end{proposition}

\begin{proof}
    Just as in the proof of Proposition \ref{prop:composition of tangles}, one may reduce this to the case where $t_0$ is an $(\varnothing, R)$-web, $t_1$ is an  $(R, \varnothing)$-web, and both are non-elliptic.
    This simple case follows from multiplying the canonical isomorphism $H^R_\A \o_{H^R_\A} H^R_\A \cong H^R_\A$ on the left and the right by the idempotents $1_{\b{t_0}}$ and $1_{t_1}$.  See also \cite[Lemma 2.2.16]{Robert-thesis} for the planar case.   
\end{proof}

We have the following analogue of Definition \ref{def:ATan} for $SL(3)$ webs and foams. 

\begin{definition}
    Let $\AWeb$ denote the $2$-category of flat annular webs and foams (with corners) between them. Objects of $\AWeb$ are sign strings $S \in \{\pm 1\}^n$ over all $n\in \Z_{\geq 0}$. Recall that for each $m\in \Z_{\geq 0}$ we have fixed $m$ points $\u{m} \subset \SS^1$. Given sign strings $R$ and $S$, the set of $1$-morphisms from $S$ to $R$ is the collection of all $(R,S)$-webs. Given two $(R,S)$-webs $t_0, t_1$, a $2$-morphism from $t_0$ to $t_1$ is an $SL(3)$ foam with corners $F \subset \A\times [0,1]$ whose boundary decomposes into the following four pieces:
\begin{itemize}
    \item $\partial F \cap \left( \A \times \{1\}  \right) = t_1$ and $\partial F \cap \left( \A \times \{0\}  \right) = -t_0$, where $-t_0$ denotes $t_0$ with reversed orientation. 
    \item $\partial F \cap \left( \SS^1 \times \{0\} \times [0,1]\right) = \u{\ell(S)} \times \{0\} \times [0,1]$,
        \item $\partial S \cap \left( \SS^1 \times \{1\} \times [0,1]\right) = \u{\ell(R)} \times \{1\} \times [0,1]$.
\end{itemize}
    We consider $2$-morphisms up to ambient isotopy of $\A\times [0,1]$ which fixes  $\partial (\A\times [0,1])$ pointwise. 
\end{definition}

A $2$-morphism $F$ in $\AWeb$ from an $(R,S)$-web $t_0$ to an $(R,S)$-web $t_1$ induces a morphism of bimodules 
\[
\F(F) : \F(t_0) \to \F(t_1)
\]
as follows. For $w\in B^S, v\in B^R$, there is a foam cobordism  $_v F_w: \b{v}t_0 w \to \b{v} t_1 w$ between closed webs obtained by gluing $F$ to the identity foam cobordisms $\id_{w}$ and $\id_{\b{v}}$ along their common boundary intervals. Then $\F(F)$ is the direct sum of the induced maps  
\[
\brak{_v F_w} : \brak{ \b{v}t_0 w} \to \brak{\b{v} t_1 w}
\]
over all $w\in B^S, v\in B^R$. It preserves $\adeg$ and increasing $\qdeg$ by 
\[
-2\chi(F) + \chi(t_0) + \chi(t_1) + 2d(F),
\]
where $d(F)$ is the number of dots on $F$. 

Denote by $\gggBimod$ the $2$-category whose objects are $\Z\oplus \Lambda$-graded $R_{\alpha}$-algebras, $1$-morphisms are graded bimodules, and $2$-morphisms are homogeneous maps of graded bimodules. Our results imply the following. 

\begin{proposition}
\label{prop:2-functor sl(3)}
The above construction assembles into a $2$-functor 
    \[
        \F : \AWeb \to \gggBimod
    \]
    sending an object $S$ to $H^S_\A$, an $(R,S)$-web ($1$-morphism) $t$ to $\F(t)$, and a $2$-morphism $F: t_0 \to t_1$ to $\F(F)$. 
\end{proposition}
 
Now let $D$ be a diagram for an  oriented $(n,m)$-tangle $T\subset \A\times [0,1]$ (see Definition \ref{def:annular tangle}). For $i=0,1$, let $S_i$ be the sign string (length $n$ for $i=1$, length $m$ for $i=0$) determined by the orientation of $T$ at its boundary points in $\A\times \{i\}$.

\begin{figure}
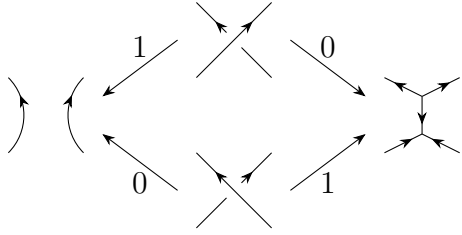

\centering 
\includestandalone{sl3_cube_of_resolutions}
\caption{The $0$- and $1$-smoothings used to define the $SL(3)$ chain complex.}\label{fig:sl3 crossing resolution for cube}
\end{figure}

Following \cite[Section 4]{sl3-link-homology}, form the cube of resolutions of $D$ according to Figure \ref{fig:sl3 crossing resolution for cube}.
Canonical cobordisms between the left and the right smoothings in Figure~\ref{fig:sl3 crossing resolution for cube} induce homogeneous homomorphisms between the corresponding bimodules. 
Applying functor $\F(-)$ to the above construction  yields a commutative cube of $(H^{S_1}_\A, H^{S_0}_\A)$-bimodules. Introducing signs to make the cube anticommute, collapsing to a chain complex, and adding homological and quantum grading shifts as in \cite[Figure 30]{sl3-link-homology} yields a chain complex $C(D)$ of $\Z\oplus \Lambda$-graded $(H^{S_1}_\A, H^{S_0}_\A)$-bimodules. 

\begin{theorem}
\label{thm_isotopy_invariance sl3}
    The chain homotopy type of the complex of $\Z\oplus \Lambda$-graded bimodules $C(D)$ depends only on the ambient isotopy class of $T$.
\end{theorem}
\begin{proof}
    Two diagrams representing $T$ are related by Reidemeister moves supported in the interior of $\A$.  
    The proof of invariance in \cite{MVuniversal} uses foam cobordisms, is entirely local, and all local foam relations imposed in \cite{MVuniversal} also hold for anchored foam evaluation (where they occur away from the anchor line $\line$) by \cite[Theorem 4.26]{AK}. It follows that the local chain homotopy equivalences defined in \cite{MVuniversal} give homotopy equivalences in our setting. 
\end{proof}

%\newpage

\bibliographystyle{amsalpha}
\bibliography{references}

\end{document}